\documentclass[11pt, oneside]{article}

\usepackage{epsfig}
\usepackage{tikz}
\usepackage{url}
\usepackage{amsmath,amstext,amssymb,amsfonts,amsthm}
\usepackage{mathrsfs}
\usepackage{xcolor}
\usepackage{mathtools}
\usepackage{graphicx}

\usetikzlibrary{fadings,shapes,arrows.meta,positioning}
\usepackage{hhline}

\usepackage{booktabs}
\usepackage{multicol}
\usepackage{multirow}
\usepackage[labelfont=bf]{caption}
\usepackage{subcaption}
\usepackage{siunitx}
\usepackage[title]{appendix}

\newtheorem{thm}{Theorem}[section]

\newtheorem{rem}{Remark}[section]

\newtheorem{exmp}{Example}

\numberwithin{equation}{section}
\numberwithin{figure}{section}

\newcommand{\dd}{\mathrm{d}}
\newcommand{\ii}{\mathrm{i}}

\tikzfading[name=fade right, left color=transparent!0, right color=transparent!100]
\tikzfading[name=fade left, right color=transparent!0, left color=transparent!100]
\tikzfading[name=fade down, bottom color=transparent!0, top color=transparent!100]
\tikzfading[name=fade up, top color=transparent!0, bottom color=transparent!100]
\allowdisplaybreaks[4]

\title{
Exponential Time Differencing Runge-Kutta Discontinuous Galerkin (ETD-RKDG) Methods for Nonlinear Degenerate Parabolic Equations 
}
\author{
Ziyao Xu\footnotemark[2]
\and 
Yong-Tao Zhang\footnotemark[3]
}
\date{}
\begin{document}

\maketitle
\renewcommand{\thefootnote}{\fnsymbol{footnote}}
\footnotetext[2]{Department of Applied and Computational Mathematics and Statistics,
University of Notre Dame, Notre Dame, IN 46556, USA. E-mail: zxu25@nd.edu}
\footnotetext[3]{Department of Applied and Computational Mathematics and Statistics,
University of Notre Dame, Notre Dame, IN 46556, USA. E-mail: yzhang10@nd.edu}

\begin{center}
\small
\begin{minipage}{0.9\textwidth}
\textbf{Abstract.}
In this paper, we study high-order exponential time differencing Runge-Kutta (ETD-RK) discontinuous Galerkin (DG) methods for nonlinear degenerate parabolic equations.
This class of equations exhibits hyperbolic behavior in degenerate regions and parabolic behavior in non-degenerate regions, resulting in sharp wave fronts in the solution profiles and a parabolic-type time-step restriction, $\tau \sim O(h^2)$, for explicit time integration.
To address these challenges and solve such equations in complex domains, we employ DG methods with appropriate stabilizing limiters on unstructured meshes to capture the wave fronts and use ETD-RK methods for time integration to resolve the stiffness of parabolic terms.
We extract the system's stiffness using the Jacobian matrix of the DG  discretization for diffusion terms and adopt a nodal formulation to facilitate its computation.
The algorithm is described in detail for two-dimensional triangular meshes. We also conduct a linear stability analysis in one spatial dimension and present computational results on three-dimensional simplex meshes, demonstrating significant improvements in stability and large time-step sizes.


\medskip

\end{minipage}
\end{center}
\setlength{\parindent}{2em}








\pagenumbering{arabic}

\section{Introduction}\label{Sect:intro}
In this paper, we study efficient high-order computational methods for nonlinear degenerate parabolic equations.
A prototypical example of this class of equations is the porous medium equation (PME):
\begin{equation}\label{eq:PME}
u_t=\Delta u^m, \quad \mathbf{x}\in\mathbb{R}^d,
\end{equation}
where $d$ is the spacial dimension, $\Delta :=\frac{\partial^2}{\partial x_1^2}+\cdots+\frac{\partial^2}{\partial{x}_{d}^{2}}$ is the Laplacian operator, and $m>1$ is a constant exponent.
The PME models the flow of an ideal gas in isentropic homogeneous porous medium \cite{Aronson}, derived by combining the equation of continuity $u_t+\nabla\cdot(\mathbf{v}u)=0$, Darcy's law $\mathbf{v}=-c_1\nabla p$ and the equation of state (EOS) for gases $p=c_2 u^{m-1}$.
In general, nonlinear degenerate parabolic equations may also include hyperbolic terms, in which case a more general formulation is given by
\begin{equation}\label{eq:GeneralEquation}
u_t + \nabla \cdot \mathbf{F}(u) = \Delta g(u), \quad \mathbf{x} \in \mathbb{R}^d, 
\end{equation}
where $\mathbf{F}(u) = \left(f_1(u), \ldots, f_d(u)\right)^T$ is the flux vector, and $g(u)$ is the degenerate parabolic term with $g'(u) = 0$ over some range of $u$.

The PME is known to exhibit hyperbolic behaviors in degenerate regions ($u=0$), with sharp wave fronts appearing at the boundary of its support and propagating at finite speed.
For instance, the well-known spherically symmetric Barenblatt solution \cite{Barenblatt_sol} in $d$ spacial dimensions is given by 
\begin{equation}\label{eq:barenblatt}
u(\mathbf{x},t)=t^{-p}\left[ \left(1-\frac{p(m-1)}{2dm}\frac{|\mathbf{x}|^2}{t^{2p/d}}\right)^{+}  \right]^{\frac{1}{m-1}},
\end{equation}
with the wave front located at 
\begin{equation*}
\{\mathbf{x}\in\mathbb{R}^{d}:|\mathbf{x}|^2=\frac{2dm}{p(m-1)}t^{2p/d}\},
\end{equation*}
where $p=\frac{1}{m-1+2/d}$ and $z^+:=\max\{z,0\}$.
An illustration of the solution profiles in one-, two-, and three-dimensional spaces for $m = 3$ at $t =1$ is shown in Figure \ref{fig:intro}.
\begin{figure}[!htbp]
 \centering
 \begin{subfigure}[b]{0.3\textwidth}
  \includegraphics[width=\textwidth]{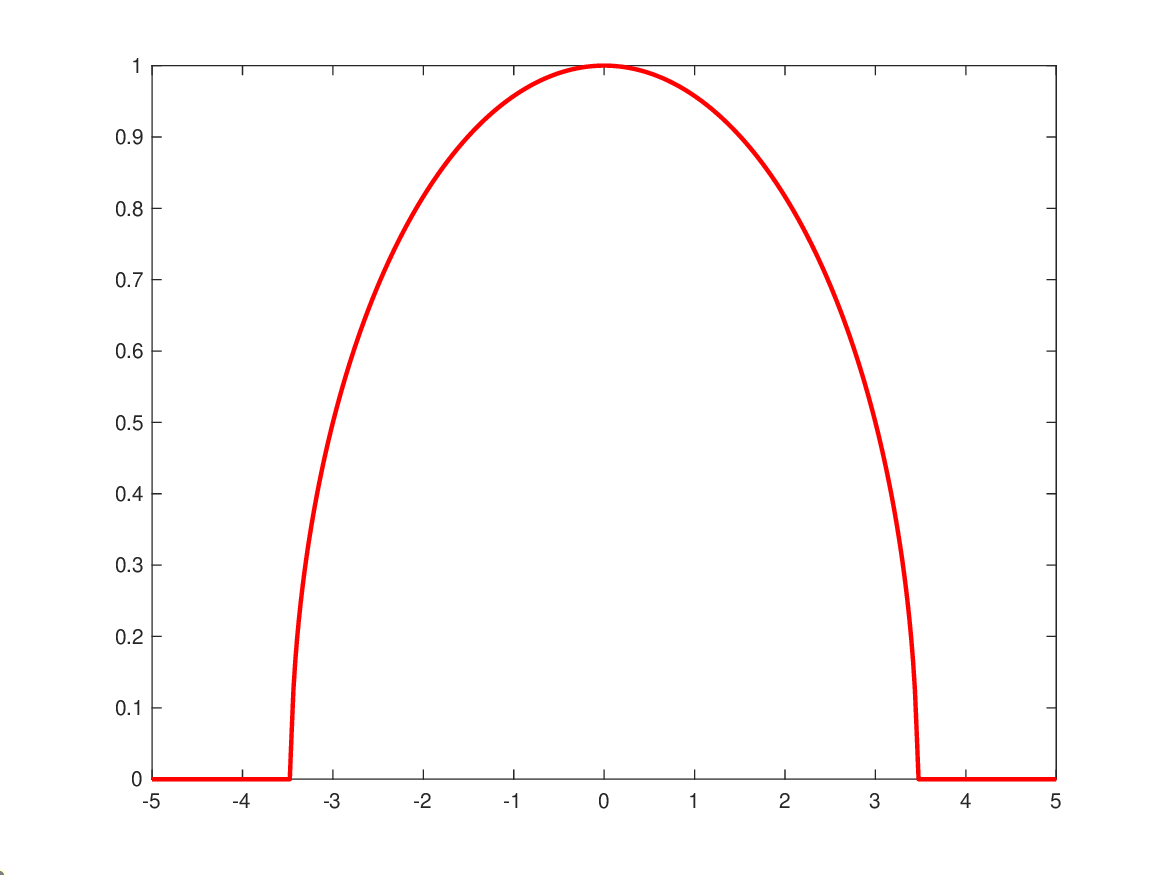}
  \caption{$d=1$}
 \end{subfigure}
 \begin{subfigure}[b]{0.3\textwidth}
  \includegraphics[width=\textwidth]{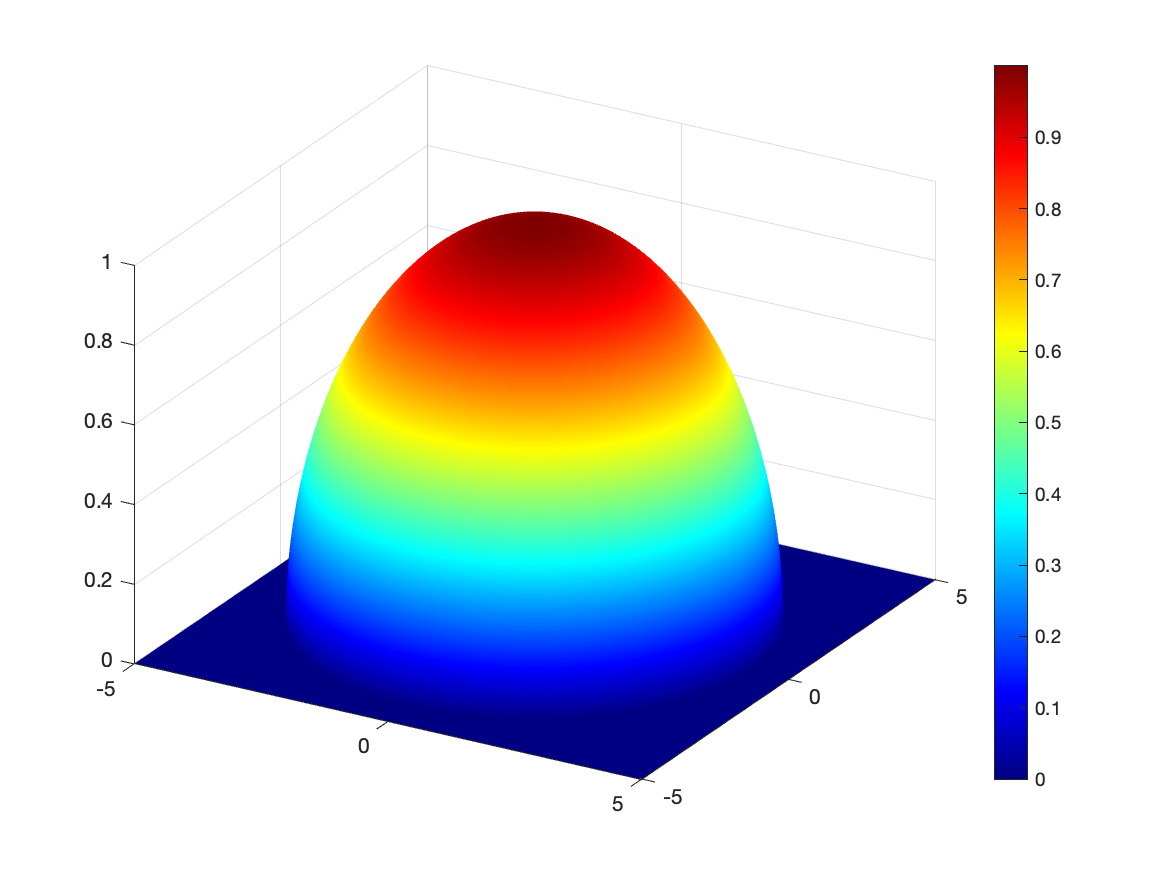}
  \caption{$d=2$}
 \end{subfigure}
 \begin{subfigure}[b]{0.3\textwidth}
  \includegraphics[width=\textwidth]{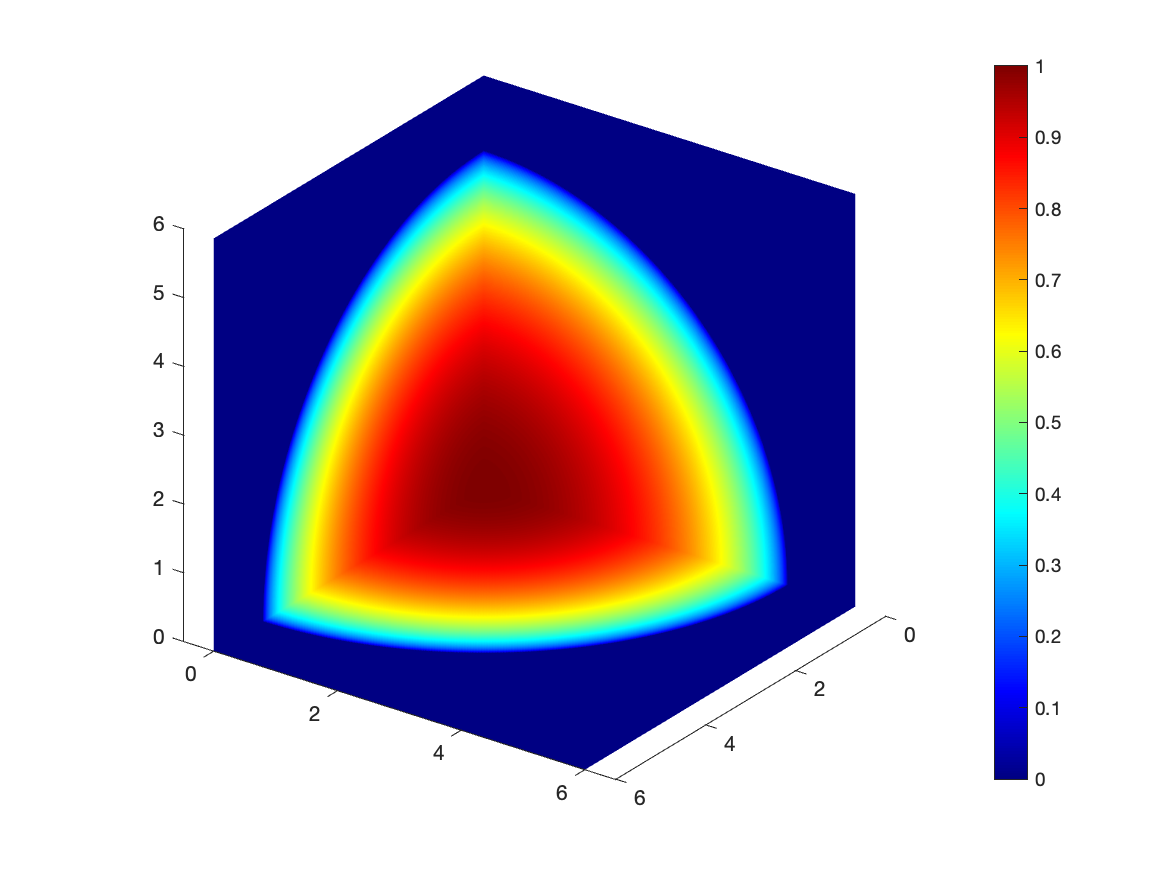}
  \caption{$d=3$}
 \end{subfigure} 
 \caption{\footnotesize{The Barenblatt solutions \eqref{eq:barenblatt} with $m=3$ at $t=1$ in different dimensions $d$.}}
 \label{fig:intro}
\end{figure}
On the other hand, the PME exhibits parabolic behaviors in non-degenerate regions ($u > 0$), with the diffusion coefficient $g'(u)=mu^{m-1}>0$, leading to smooth solution profiles and infinite propagation speed for perturbations.

Both the hyperbolic and parabolic features of the PME pose challenges for computational methods.
The sharp wave front at the boundary of the solution support often causes non-physical oscillations in numerical approximations, known as the Gibbs phenomenon, which weakens the robustness of simulations due to its anti-diffusive effects.
To address this issue, many numerical methods with nonlinear stabilization mechanisms have been developed, such as the weighted essentially non-oscillatory (WENO) methods \cite{Abe1,Abedian,arbogast2019finite,jiang2021high,liu2011high,zhang2022high}, discontinuous Galerkin methods \cite{liu2009direct,TaoLiuJiangLu,ZhWu,zhang2013maximum}, moving mesh methods \cite{baines2005moving,Ngo}, kinetic schemes \cite{Aregb}, relaxation schemes \cite{Cavalli}, among others \cite{Besse,Christ,Mage,Noch1,Qiu}.
The parabolicity, in contrast, smooths the solution profile in non-degenerate regions but simultaneously leads to an infinite propagation speed for information. 
As a result, the ODE systems resulting from various spatial discretization methods are typically stiffer than those arising from purely hyperbolic equations.
The widely used explicit time-marching methods for hyperbolic equations, including strong stability preserving Runge-Kutta and multi-step methods \cite{gottlieb2009high, gottlieb2011strong, gottlieb2001strong, shu1988total}, and Lax-Wendroff methods \cite{li2016two, qiu2005discontinuous} thus suffer from a severe time-step restriction, $\tau \sim O(h^2)$, where $\tau$ is the time-step size and $h$ is the spatial mesh size.
To relieve the time-step restriction, implicit treatments of the parabolic term can be adopted.
These methods include conditionally stable implicit Runge-Kutta methods \cite{ketcheson2009optimal} and L-stable Runge-Kutta methods \cite{arbogast2019finite}.
In the presence of hyperbolic terms, as in \eqref{eq:GeneralEquation}, implicit-explicit methods \cite{Kennedy, bürger2013regularized, boscarino2015linearly}, which handle the parabolic and hyperbolic terms separately, are widely used to allow for a time-step size $\tau \sim O(h)$.
Another important class of robust time-marching approaches is the exponential integrator methods \cite{Oster}, including exponential time differencing (ETD) multi-step and Runge-Kutta methods \cite{Beylk, cox2002exponential, DuQ, kassam2005fourth, liu2023exponential}, as well as implicit integration factor methods \cite{nie2006efficient, ChenZ, jiang2016krylov, liu2019krylov, LuZ}, which integrate the linear stiff components of ODEs exactly using exponential integrating factors.

In this paper, we focus on the discontinuous Galerkin (DG) method in its nodal formulation for solving the problem \eqref{eq:GeneralEquation}, and we use exponential time differencing Runge-Kutta (ETD-RK) methods to advance the resulting ODE system in time.
The DG method for hyperbolic equations was first proposed by Reed and Hill in 1973 to solve a steady linear neutron transport problem \cite{ReedHill}.
It employs piecewise polynomials that allow for discontinuities in the function space.
It was later developed by Cockburn et al. in a series of works \cite{DG1,DG2,DG3,DG4,DG5} into the Runge-Kutta DG methods for solving time-dependent nonlinear hyperbolic equations.
DG methods have been widely used ever since due to their advantages, including high-order accuracy, compact stencils, flexibility in handling complex geometries, ease of $h$-$p$ adaptivity, and high parallel efficiency.
The low regularity of solutions in degenerate parabolic equations makes DG methods also a natural choice for discretizing the parabolic terms.
Several valid DG methods have been developed for this purpose, including the interior penalty DG method \cite{arnold1982interior,wheeler1978elliptic}, direct DG method \cite{liu2009direct, liu2010direct}, local DG method \cite{LDG,TaoLiuJiangLu, ZhWu}, compact DG method \cite{peraire2008compact}, and ultra-weak DG method \cite{ultraweakDG, ChenZ}, among others.
In this work, we apply the DG formulations in \cite{DG2} and \cite{ChenZ} to discretize the hyperbolic and parabolic terms, respectively, in the problem \eqref{eq:GeneralEquation}.
There are many options for advancing the ODE system resulting from the DG discretization in time. Among these, widely used methods include explicit strong stability preserving Runge-Kutta (SSP-RK) and multi-step methods \cite{gottlieb2009high, gottlieb2011strong, gottlieb2001strong, shu1988total}, Lax-Wendroff methods \cite{qiu2005discontinuous, xu2022third}, and implicit \cite{qin2018implicit, xu2023conservation, xu2022third} or implicit-explicit (IMEX) methods \cite{Kennedy, bürger2013regularized, boscarino2015linearly}.
Explicit time marching methods are simple and efficient for advancing the solution at each time step. However, when applied to parabolic equations, they suffer from a restrictive CFL condition of $\tau\sim O(h^2)$.
Implicit treatment of the parabolic terms, on the other hand, allows for much larger time steps, $\tau\sim O(h)$, at the cost of higher computational effort per step.
Moreover, when the parabolic terms are nonlinear, iterative techniques such as Newton's method are typically required to ensure convergence of the implicit time marching scheme.
Designing efficient preconditioning methods \cite{dobrev2006two, kraus2008multilevel} to accelerate the convergence of these computations is an important topic for implicit methods.
The ETD-RK method belongs to the family of exponential integrators \cite{Oster}, which absorb and integrate the linear stiff part of an ODE system exactly using an integrating factor.
It also incorporates the idea of Runge-Kutta methods to perform time integration of the remaining terms, resulting in an explicit, single-step, high-order time marching scheme.
The ETD-RK schemes offer advantages such as relatively small numerical errors, large stability regions, and good steady-state preservation properties \cite{cox2002exponential, Beylk, kassam2005fourth}.
When the stiff component in an ODE system is nonlinear, exponential Rosenbrock-type methods \cite{hochbruck2009exponential} or exponential propagation iterative methods \cite{tokman2006efficient} are often adopted.
These methods extract and absorb the stiffness of the nonlinear term through the Jacobian matrix and integrate the resulting nonlinear residual explicitly.

In our previous work \cite{xu2025high}, a novel and effective semilinearization approach was developed to efficiently incorporate ETD-RK methods into the high-order WENO finite difference discretizations to solve nonlinear degenerate parabolic equations \eqref{eq:GeneralEquation}.
The numerical results show a significant advantage in the large CFL condition and efficiency compared to the explicit and implicit SSP-RK time-marching methods.
However, WENO finite difference methods have restrictions such as difficulty in dealing with complex domain geometries, relatively wide stencils, and relatively complicated vectorization procedure in their implementation. These motivate us to develop the ETD-RKDG methods on unstructured meshes for \eqref{eq:GeneralEquation} in this paper.
There are two primary formulations for implementing DG methods - 
 the modal and nodal approaches \cite{hesthaven2007nodal}.
The modal DG formulation typically employs orthogonal polynomial bases on each element and assumes exact integration or the use of sufficiently accurate quadrature rules to evaluate the weak form of the DG method.
In contrast, the nodal formulation uses Lagrange basis functions defined at interpolation points and replaces integrands in the DG formulation with their Lagrange interpolants. This leads to a more vectorized implementation, sparser stiffness matrices, and certain structure-preserving properties \cite{chen2017entropy, xu2024WBDG}.
To facilitate the computation of Jacobian matrices and other matrix-intensive operations in our algorithm, we adopt the nodal formulation.
To control spurious oscillations at each Runge-Kutta time step, we incorporate the total variation bounded (TVB) limiter from \cite{DG5}.
Alternative stabilization techniques, such as the oscillation-free method \cite{lu2021oscillation} or the oscillation-eliminating method \cite{peng2025oedg}, can also serve this purpose.
In addition to numerical experiments, we conduct a theoretical analysis of the stability of the ETD-RK methods for parabolic equations.
A Fourier analysis of a one-dimensional linear diffusion equation shows that the time-marching scheme is unconditionally stable when the linear stiff term absorbed by the exponential integrator exceeds a certain threshold.
Although the linear analysis does not fully capture the behavior of the proposed nonlinear schemes for the nonlinear degenerate parabolic equations, it provides an intuitive justification for the strong robustness of ETD-RK methods within the DG framework.
It is also worth mentioning that our previous study \cite{xu2025stability} demonstrates excellent stability of the ETD-RKDG methods for linear convection–diffusion equations when the linear parabolic term is integrated exactly, even with the convection term integrated approximately and explicitly in the exponential integrator.

The rest of the paper is organized as follows.
In Section \ref{Sect:algorithm}, we introduce the nodal DG spatial discretization for both the hyperbolic and parabolic terms, along with ETD-RK time-marching methods of various orders.
Section \ref{Sect:analysis} presents a linear stability analysis for the one-dimensional case using the Fourier method, serving as a heuristic justification for the enhanced stability provided by the Rosenbrock-type treatment.
In Section \ref{Sect:tests}, we conduct extensive numerical tests in two and three spatial dimensions in complex domains to demonstrate the efficiency and stability of our algorithm, and to compare it with other existing methods.
We conclude in Section \ref{Sect:summary} with a summary of the key techniques and achievements of our algorithm.
Finally, Appendices \ref{appd:LagrangeNodes} and \ref{appd:LocalMatrices} provide detailed computations of the nodal basis functions and local matrices that constitute the matrix equations in the nodal DG methods, to facilitate implementation.

\section{Numerical Algorithm}\label{Sect:algorithm}
In this section, we present the numerical algorithm of the ETD-RKDG method to solve degenerate parabolic equations.
We begin by describing the DG discretization of both the diffusion and convection terms using the nodal formulation.
Next, we introduce the ETD-RK time integration scheme for the resulting semi-discrete ODE system.
The computation of the Jacobian matrix in the nodal DG framework is discussed in detail, as it facilitates the Rosenbrock-type treatment.
We conclude with remarks on the efficient implementation of the ETD-RK methods and their conservation properties.

\subsection{DG spatial discretization}\label{Sect:DGscheme}
We describe the DG discretization for degenerate parabolic equations.
To fix ideas, we discuss triangular grids in two spatial dimensions throughout this subsection. However, the algorithm can be trivially extended to three-dimensional tetrahedral grids, as investigated in the numerical section.

\subsubsection{Preliminaries}\label{eq:prelim}

Consider the computational domain $\Omega \subset \mathbb{R}^2$ with a regular triangulation $\mathcal{T}$.
For any element $K \in \mathcal{T}$, let $|K|$ denote its area, and let $\mathbf{v}_i^K$, $i = 1, 2, 3$, be the three vertices of $K$ ordered counterclockwise.
We denote by $e_i^K$ the edge opposite to the vertex $\mathbf{v}_i^K$, with corresponding length $|e_i^K|$, unit outward normal vector $\mathbf{n}_i^K=(\mathbf{n}_{ix}^K, \mathbf{n}_{iy}^K)^T$, and neighboring element $K_i$, for $i = 1, 2, 3$; see Figure~\ref{fig:mesh} for an illustration.
Moreover, we define $E(i, K) \in \{1, 2, 3\}$ such that $e_{E(i,K)}^{K_i} = e_i^K$ for $i = 1, 2, 3$.

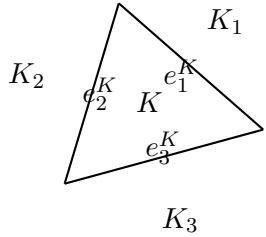
\begin{figure}[!htbp]
 \centering
  \begin{tikzpicture}[scale=1.2]

    \draw[thick] (-0.8,-0.8) -- (1.4,-0.2);
    \draw[thick] (1.4,-0.2) -- (-0.2,1.2);
    \draw[thick] (-0.2,1.2) -- (-0.8,-0.8);
    \coordinate [label=left:$K$] (d) at (0.4,0.1);
    \coordinate [label=left:$K_1$] (d) at (1.3,1);
    \coordinate [label=left:$K_2$] (d) at (-0.9,0.4);
    \coordinate [label=left:$K_3$] (d) at (0.8,-1.2);

    \coordinate [label=left:$e^K_1$] (d) at (0.8,0.4);
    \coordinate [label=left:$e^K_2$] (d) at (-0.1,0.2);
    \coordinate [label=left:$e^K_3$] (d) at (0.6,-0.4);

    \draw[thick][thick, path fading=fade right]  (1.4,-0.2) -- (1.7,1.7);
    \draw[thick][thick, path fading=fade right]  (-0.2,1.2) -- (1.7,1.7);
    \draw[thick][thick, path fading=fade left]  (-0.8,-0.8) -- (-2.3,0.7);
    \draw[thick][thick, path fading=fade left]  (-0.2,1.2) -- (-2.3,0.7);
    \draw[thick][thick, path fading=fade up]  (-0.8,-0.8) -- (0.9,-2.3);
    \draw[thick][thick, path fading=fade up]  (1.4,-0.2) -- (0.9,-2.3);
    \draw[thick][thick, path fading=fade up] (-0.8,-0.8) -- (-1.3,-2.4);
    \draw[thick][thick, path fading=fade up] (-0.8,-0.8) -- (-2.8,-1.0);
    \draw[thick][thick, path fading=fade up] (1.4,-0.2) -- (3.0,-1.7);
    \draw[thick][thick, path fading=fade right] (1.4,-0.2) -- (3.2,0.4);
    \draw[thick][thick, path fading=fade down] (-0.2,1.2) -- (0.35,2.85);
    \draw[thick][thick, path fading=fade down] (-0.2,1.2) -- (-1.7,2.7);
  \end{tikzpicture}  
  \caption{Triangulation of a local region in the computational domain $\Omega$.}\label{fig:mesh}
\end{figure}

The DG finite element space is defined as 
\begin{equation}
V^k_{h}=\{v\in L^2(\Omega): v|_{K}\in \mathcal{P}^{k}(K),\, \forall K\in\mathcal{T}\},
\end{equation}
where $\mathcal{P}^{k}(K)$ denotes the space of polynomials of degree at most $k$ on the element $K$.

There are two primary choices for the basis of the finite element space $V_h^k$:
\begin{equation}\label{eq:FEMbasis}
\begin{split}
V_{h}^{k}&=\bigoplus_{K\in\mathcal{T}} \text{span}\{{\phi}_i^{K}({\mathbf{x}}), \, i=1,2,\ldots, N_k\}\\
&=\bigoplus_{K\in\mathcal{T}} \text{span}\{{\ell}_i^K({\mathbf{x}}),\, i=1,2,\ldots, N_k\},
\end{split}
\end{equation}
where $N_k=\frac12{(k+1)(k+2)}$ is the number of degrees of freedom (DoFs) in $\mathcal{P}^{k}(K)$, and ${\phi}_i^{K}$ and ${\ell}_i^K$ denote the orthogonal (modal) and Lagrange (nodal) basis functions in $\mathcal{P}^{k}(K)$, respectively.
Moreover, we let $N=\text{dim} V_{h}^{k}$ denote the total number of degrees of freedom in $V_{h}^{k}$. 
On the reference element
\begin{equation}
\widehat{K}=\{\widehat{\mathbf{x}}:=(\widehat{x}_1,\widehat{x}_2)^T: \widehat{x}_{1}\geq0, \widehat{x}_2\geq 0, \widehat{x}_1+\widehat{x}_2\leq 1\},
\end{equation}
we construct the orthonormal basis $\widehat{\phi}_{i}$ using the Gram-Schmidt process applied to the polynomials $\{\widehat{x}_1^i\widehat{x}_2^{j}: i,j\geq0, i+j\leq k\}$, and define the Lagrange basis $\widehat{\ell}_i$ as the interpolants at the Lagrange nodes $\{\widehat{\mathbf{x}}_i\}_{i=1}^{N_k}$, such that
\begin{equation}
\int_{\widehat{K}}\widehat{\phi}_{i}\widehat{\phi}_j\,d\widehat{\mathbf{x}}=\delta_{j}^{i},\quad \text{and}\quad \widehat{\ell}_{i}(\widehat{\mathbf{x}}_{j})=\delta_{j}^{i}, \quad \text{for}\,\, i,j=1,\ldots, N_{k}.
\end{equation}
On the physical elements $K\in\mathcal{T}$, we define $\phi_{i}^{K}=\widehat{\phi}_{i} \circ \mathbb{T}^K$ and $\ell_{i}^{K} = \widehat{\ell}_{i} \circ \mathbb{T}^K$, where $\mathbb{T}^K$ is the affine mapping from $K$ to $\widehat{K}$.
In \cite{hesthaven2007nodal}, a general formula for the modal basis $\widehat{\phi}_{i}$ based on Jacobi polynomials is provided for arbitrary polynomial degree.
The authors also propose a set of Lagrange nodes $\{\widehat{\mathbf{x}}_i\}_{i=1}^{N_{k}}$ that minimizes the Lagrange constant; see Figure~\ref{fig:Lagrange_nodes} for the distribution of the nodes for different polynomial orders.
The exact coordinates of the nodes are provided in the Appendix \ref{appd:LagrangeNodes}.
They coincide with the $(k+1)$-point Legendre-Gauss-Lobatto (LGL) nodes on each edge.
\begin{figure}[!htbp]
 \centering
 \begin{subfigure}[b]{0.22\textwidth}
  \includegraphics[width=\textwidth]{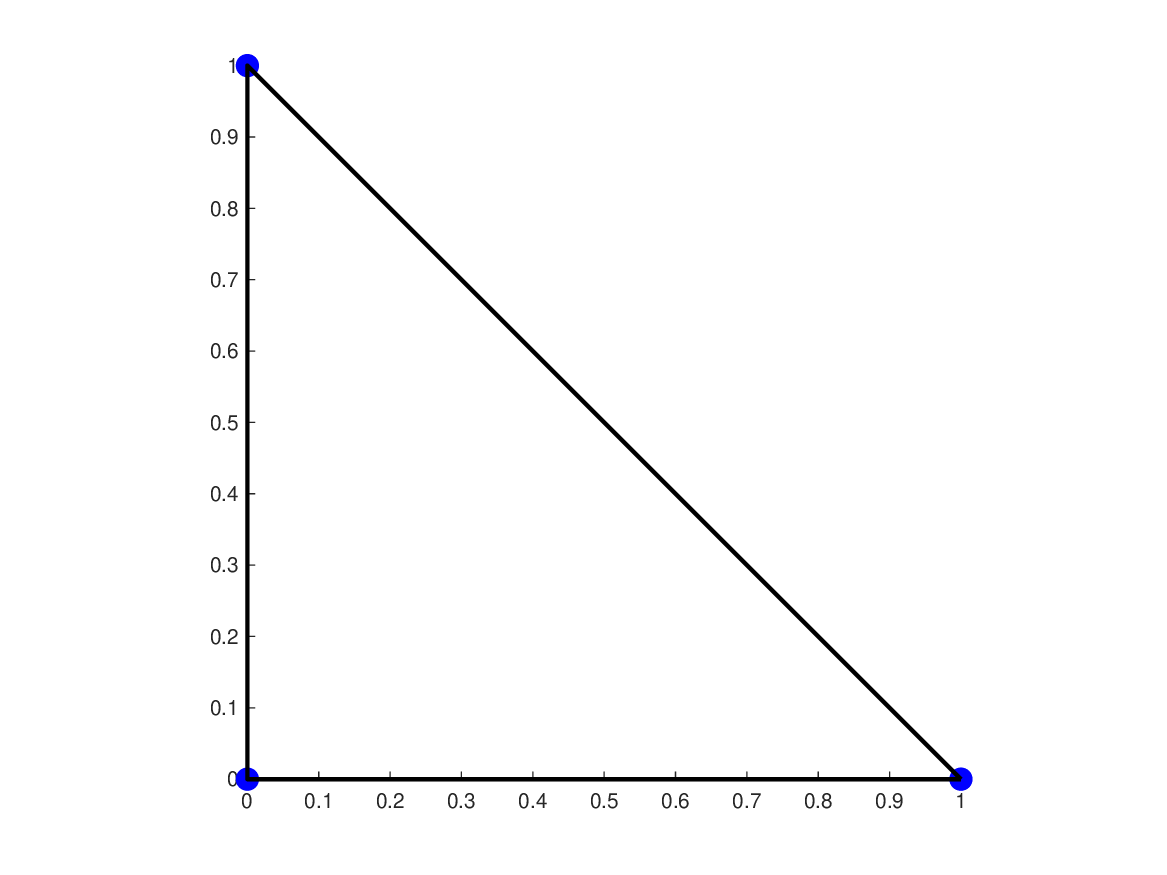}
  \caption{$\mathcal{P}^1(\widehat{K})$}
 \end{subfigure}
  \begin{subfigure}[b]{0.22\textwidth}
  \includegraphics[width=\textwidth]{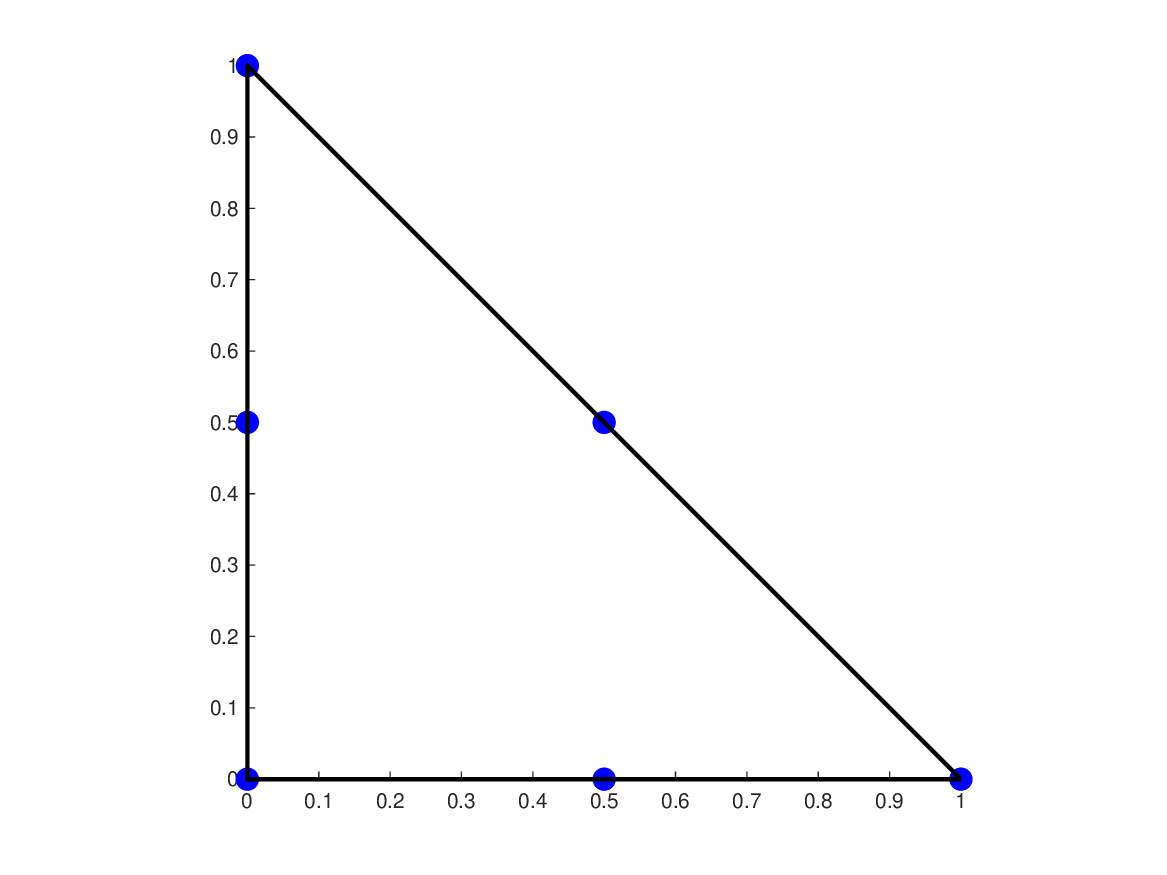}
  \caption{$\mathcal{P}^2(\widehat{K})$}
 \end{subfigure}
 \begin{subfigure}[b]{0.22\textwidth}
  \includegraphics[width=\textwidth]{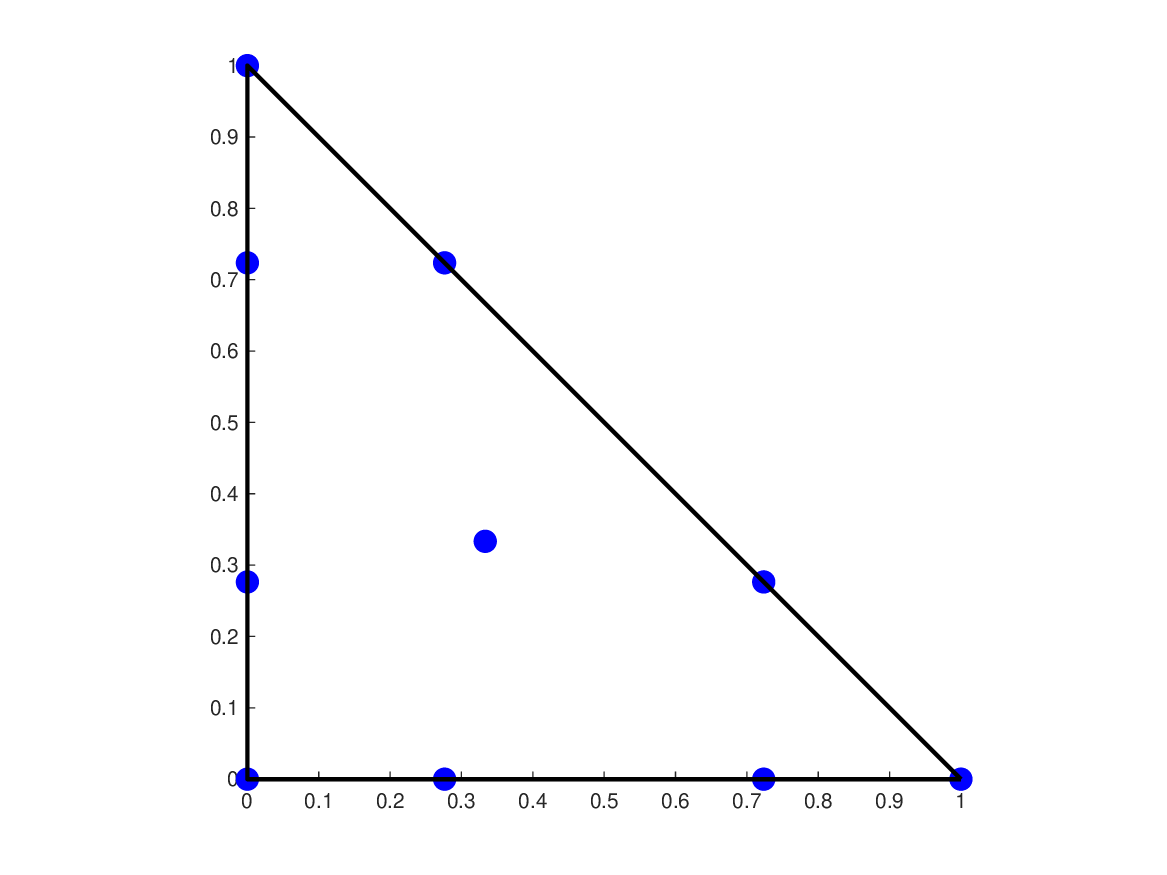}
  \caption{$\mathcal{P}^3(\widehat{K})$}
 \end{subfigure}
 \begin{subfigure}[b]{0.22\textwidth}
  \includegraphics[width=\textwidth]{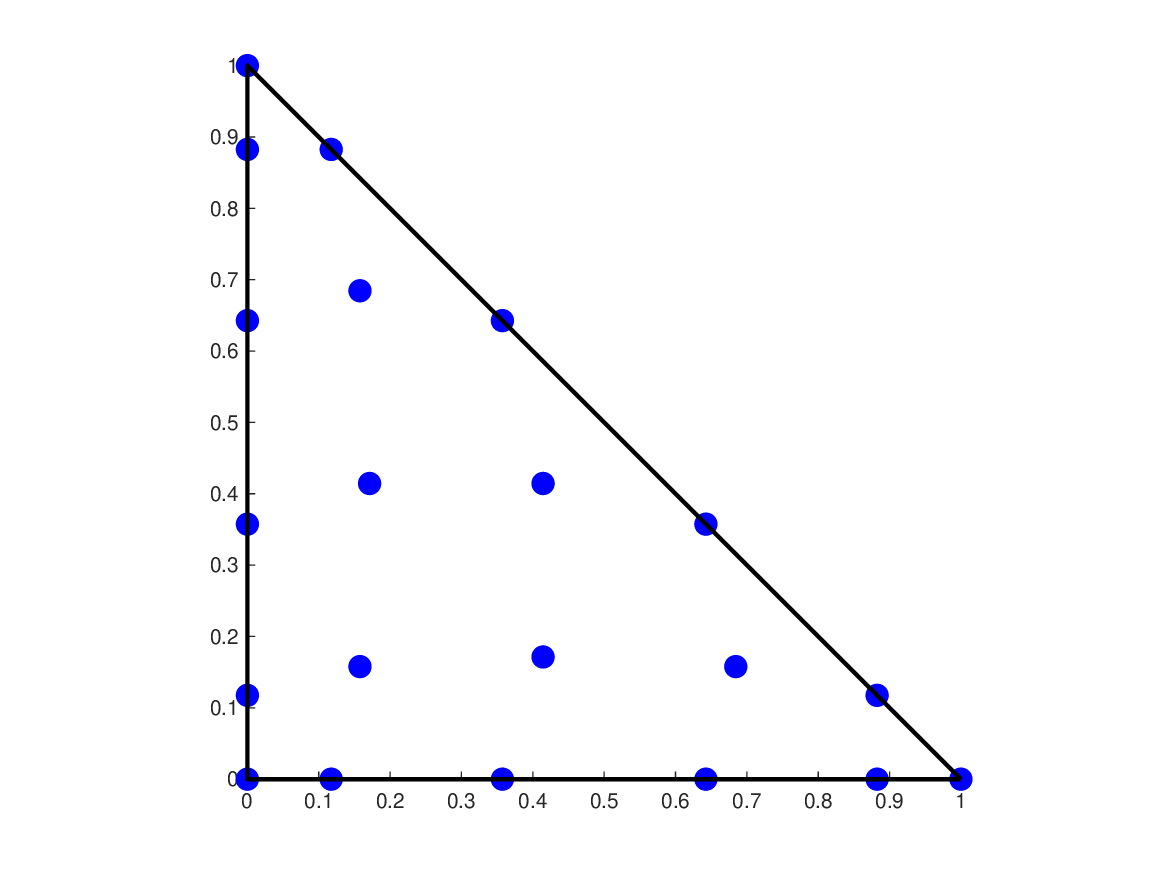}
  \caption{$\mathcal{P}^5(\widehat{K})$}
 \end{subfigure}
 \caption{Distribution of the two-dimensional Lagrange nodes (blue dots) on $\widehat{K}$. 
 The Lagrange nodes coincide with the $(k+1)$-point Legende-Gauss-Lobatto (LGL) nodes on each edge.}
 \label{fig:Lagrange_nodes}
\end{figure}

Finally, we introduce the inner products for $f,g\in C^\infty(K)$ and $\bar{f}\in C^\infty(K'), \bar{g}\in C^\infty(K'')$, which will be used in the DG formulation:
\begin{equation}
\begin{split}
\big(f,g\big)_{K}&=\int_{K}f(\mathbf{x})g(\mathbf{x})\,d\mathbf{x},\\
\big<f,g\big>_{K}&=\int_{K}\Bigg(\sum_{i=1}^{N_{k}}f(\mathbf{x}_{i}^{K})\ell_{i}^{K}(\mathbf{x})\Bigg) \Bigg(\sum_{j=1}^{N_{k}}g(\mathbf{x}_{j}^{K})\ell_{j}^{K}(\mathbf{x})\Bigg)d\mathbf{x},\\
\big[\mathcal{L}_1{\bar{f}},\mathcal{L}_2{\bar{g}} \big]_{e}&=\int_{e} \Bigg(\sum_{i=1}^{N_{k}}\bar{f}(\mathbf{x}_{i}^{K'})\mathcal{L}_1\ell_{i}^{K'}\Bigg) \Bigg(\sum_{j=1}^{N_{k}}\bar{g}(\mathbf{x}_{j}^{K''})\mathcal{L}_2\ell_{j}^{K''}\Bigg)ds
\end{split}
\end{equation}
where $\mathbf{x}_i^K=(\mathbb{T}^K)^{-1}(\widehat{\mathbf{x}}_i)$, for $i=1,\ldots, N_{k}$, denote the Lagrange nodes on $K$, $e$ is a common edge of $K'$ and $K''$, and $\mathcal{L}_{1}$ and $\mathcal{L}_2$ are linear operators on $C^\infty(K')$ and $C^\infty(K'')$, respectively, such as the identity operator or the differentiation operator.
Extending the definitions above, we define $\big<\mathbf{f},\mathbf{g}\big>_{K}=\sum_{i=1}^{2}\big<f_{i}, g_{i}\big>_{K}$ for vector-valued functions $\mathbf{f},\mathbf{g}\in [C^\infty(K)]^2$, and $\big[\mathcal{L}_1\bar{f}, \mathcal{L}_2\bar{g} \big]_{\partial K}=\sum_{i=1}^{3}\big[\mathcal{L}_1 \bar{f}, \mathcal{L}_2 \bar{g}\big]_{e^i}$ for $\mathcal{L}_1, \mathcal{L}_2, \bar{f},\bar{g}$ defined on $K$ and its neighboring elements.

\subsubsection{DG discretization for diffusion equations}
We consider the two-dimensional diffusion equation
\begin{equation}
u_t=\Delta g(u),
\end{equation}
in the computational domain $\Omega$, subject to mixed boundary conditions:
\begin{equation}
u=u_D ~ \text{on}~\Gamma_D,\quad \frac{\partial g(u)}{\partial\mathbf{n}}=g_N ~\text{on}~\Gamma_N, 
\end{equation}
where $\Gamma_D$ and $\Gamma_N=\partial\Omega\setminus\Gamma_D$ are the Dirichlet and Neumann boundaries, respectively, and $\mathbf{n}$ denotes the unit outer normal vector on the boundary.

We adopt the semi-discrete DG scheme \cite{ChenZ} to find $u(t)\in V_{h}^{k}$ such that, for all $v\in\mathcal{P}^k(K)$,
\begin{equation}\label{eq:DG_Diff}
\big<\frac{\partial u}{\partial t},v\big>_{K}=\big< g(u),\Delta v\big>_{K}-\big[\widehat{g}, \frac{\partial v}{\partial \mathbf{n}^K}\big]_{\partial K}+\big[\widehat{\nabla g}, v\big]_{\partial K},\quad K\in\mathcal{T},
\end{equation}
where $\mathbf{n}^K$ is the unit outer normal on $\partial K$, and $\widehat{g}$ and $\widehat{\nabla g}$ are the numerical fluxes defined as follows.
On interior edges:
\begin{equation}\label{eq:DiffusionBdr1}
\begin{cases}
~\widehat{g}=\frac12(g^{\text{int}}+g^{\text{ext}}),\\
~\widehat{\nabla g}=\frac{1}{2}((\nabla g)^{\text{int}}+(\nabla g)^{\text{ext}})\cdot\mathbf{n}^K+\beta (u^{\text{ext}}-u^{\text{int}}). 
\end{cases}
\end{equation}
On the Dirichlet boundary $\Gamma_D$:
\begin{equation}\label{eq:DiffusionBdr2}
\begin{cases}
~\widehat{g}=g(u_D),\\
~\widehat{\nabla g}=(\nabla g)^{\text{int}}\cdot\mathbf{n}+\beta(u_D-u^{\text{int}}).
\end{cases}
\end{equation}
On the Neumann boundary $\Gamma_N$:
\begin{equation}\label{eq:DiffusionBdr3}
\begin{cases}
~\widehat{g}=g^{\text{int}},\\
~\widehat{\nabla g}=g_N.
\end{cases}
\end{equation}
Here, the superscripts "int" and "ext" refer to values taken from the interior and exterior of the element K, respectively. The penalty parameter $\beta=O(\frac{1}{h})$ is included to ensure stability.
Note that for periodic boundaries, no boundary condition is imposed, and the fluxes are always computed using \eqref{eq:DiffusionBdr1}.

\subsubsection{DG discretization for convection equations}
We consider the two-dimensional convection equation
\begin{equation}
u_t+\nabla\cdot F(u)=0,
\end{equation}
where $F(u)=(f_1(u), f_2(u))^T$, in the computational domain $\Omega$, subject to the boundary condition 
\begin{equation}
u=u_{\text{in}}~\text{on}~\Gamma_{\text{in}},
\end{equation}
where the inflow boundary is defined as $\Gamma_{\text{in}}=\{(x,y)^T \in\partial\Omega:F'(u_{\text{in}})\cdot\mathbf{n}<0\}$.

We adopt the semi-discrete DG scheme \cite{DG2} to find $u(t)\in V_{h}^{k}$ such that, for all $v\in\mathcal{P}^{k}(K)$,
\begin{equation}\label{eq:DG_Conv}
\big<\frac{\partial u}{\partial t}, v\big>_{K} = \big<F(u), \nabla v\big>_{K} - \big[\hat{F}, v \big]_{\partial K},\quad K\in\mathcal{T},
\end{equation}
where $\hat{F}$ is the Lax-Friedrichs numerical flux, defined as
\begin{equation}\label{eq:LFflux}
\hat{F}=\frac12(f_1^{\text{int}}\mathbf{n}^{K}_{x}+f_2^{\text{int}}\mathbf{n}^{K}_{y}+f_1^{\text{ext}}\mathbf{n}^{K}_{x}+f_2^{\text{ext}}\mathbf{n}^{K}_{y})-\frac{\alpha}{2} (u^{\text{ext}}-u^{\text{int}}),
\end{equation}
with the viscosity parameter $\alpha=\max_{K\in\mathcal{T},\forall u}|f'_1(u)\mathbf{n}^{K}_x+f'_2(u)\mathbf{n}^{K}_y|$.

On the inflow boundary $\Gamma_{\text{in}}$, we set
\begin{equation}
f_1^{\text{ext}}=f_1(u_{\text{in}}), \quad f_{2}^{\text{ext}}=f_2(u_{\text{in}}),\quad u^{\text{ext}}=u_{\text{in}},
\end{equation}
and on the outflow boundary $\Gamma_{\text{out}}=\partial\Omega\setminus\Gamma_{\text{in}}$, we take 
\begin{equation}
f_1^{\text{ext}}=f_1^{\text{int}},\quad  f_{2}^{\text{ext}}=f_2^{\text{int}},\quad u^{\text{ext}}=u^{\text{int}}.
\end{equation}
Note that for periodic boundaries, no boundary condition is imposed, and the fluxes are always computed using \eqref{eq:LFflux}.

\subsection{ETD-RK time integration}

\subsubsection{Methodology and ETD-RK schemes}
We first consider the following ODE system:
\begin{equation}
\mathbf{u}_t=L\mathbf{u}+N(\mathbf{u}),
\end{equation}
where $\mathbf{u}(t) \in \mathbb{R}^{N}$ is a vector of degrees of freedom, $L \in \mathbb{R}^{N \times N}$ is a constant matrix, and $N(\mathbf{u}) \in \mathbb{R}^{N}$ is a nonlinear vector-valued function.
We assume that $L\mathbf{u}$ and $N(\mathbf{u})$ represent the stiff and non-stiff components of the ODE system, respectively.

In the context of semi-discrete DG formulations, $L$ typically arises from the discretization of the linear diffusion term, while $N(\mathbf{u})$ corresponds to the convection term.
It is well known that explicitly treating $L\mathbf{u}$ typically results in a severe time step restriction, $\tau\sim O(h^2)$.
Instead, multiplying both sides by the exponential integrating factor $e^{-Lt}$, the stiff component $L\mathbf{u}$ can be absorbed:
\begin{equation}\label{eq:ETD_ODEs1}
\frac{d}{dt}\left(e^{-Lt}\mathbf{u}\right)=e^{-Lt} N(\mathbf{u}),
\end{equation}
and consequently, the stiffness from the linear term is fully resolved in the exponential time integration:
\begin{equation}\label{eq:ETD_ODEs2}
\mathbf{u}(t^{n+1})=e^{\tau L}\mathbf{u}(t^{n})+\int_{0}^{\tau}e^{(\tau-s)L} N(\mathbf{u}(t^n+s))\,ds,
\end{equation}
which is obtained by integrating \eqref{eq:ETD_ODEs1} over the interval $[t^n, t^{n+1}]$, where $\tau:=t^{n+1}-t^{n}$.

The ETD-RK methods are motivated by equation \eqref{eq:ETD_ODEs2}, where the nonlinear integral is approximated using multi-stage values within a single time step.
Below, we present the ETD-RK schemes from first to fourth order \cite{cox2002exponential}.
\begin{itemize}
    \item ETD-RK1
    \begin{equation}\label{eq:ETDRK1}
    \begin{split}
    \mathbf{u}^{n+1}=&e^{\tau L}\mathbf{u}^{n} + L^{-1}\left(e^{\tau L}-I\right)N(\mathbf{u}^n),
    \end{split}
    \end{equation}
   \item ETD-RK2
   \begin{equation}\label{eq:ETDRK2}
   \begin{split}
       \mathbf{a}^n=&\mathbf{u}^n+\tau \varphi_1 (\tau L) (L\mathbf{u}^{n}+N(\mathbf{u}^n)),\\
       \mathbf{u}^{n+1}=&\mathbf{a}^n+\tau\varphi_{2}(\tau L)(-N(\mathbf{u}^n)+N(\mathbf{a}^n)),
   \end{split}
   \end{equation}
    \item ETD-RK3
    \begin{equation}\label{eq:ETDRK3}
    \begin{split}
        \mathbf{a}^n=&\mathbf{u}^n+\frac{\tau}{2}\varphi_1(\frac{\tau}{2}L)(L\mathbf{u}^n+N(\mathbf{u}^n)),\\
        \mathbf{b}^n=&\mathbf{u}^n+\tau\varphi_{1}(\tau L)(L\mathbf{u}^n-N(\mathbf{u}^n)+2N(\mathbf{a}^n)),\\
        \mathbf{u}^{n+1}=&\mathbf{u}^n+\tau\varphi_{1}(\tau L)(L\mathbf{u}^n+N(\mathbf{u}^n))\\
        &+\tau\varphi_{2}(\tau L)(-3N(\mathbf{u}^n)+4N(\mathbf{a}^n)-N(\mathbf{b}^n))\\
        &+\tau\varphi_{3}(\tau L)(4N(\mathbf{u}^n)-8N(\mathbf{a}^n)+4N(\mathbf{b}^n)),
    \end{split}
    \end{equation}
    \item ETD-RK4
    \begin{equation}\label{eq:ETDRK4}
    \begin{split}
        \mathbf{a}^n=&\mathbf{u}^n+\frac{\tau}{2}\varphi_{1}(\frac{\tau}{2}L)(L\mathbf{u}^n+N(\mathbf{u}^n)),\\
        \mathbf{b}^n=&\mathbf{u}^{n}+\frac{\tau}{2}\varphi_{1}(\frac{\tau}{2} L)(L\mathbf{u}^n+N(\mathbf{a}^n)),\\
        \mathbf{c}^n=&\mathbf{a}^n+\frac{\tau}{2}\varphi_{1}(\frac{\tau}{2} L)(L\mathbf{a}^n-N(\mathbf{u}^n)+2N(\mathbf{b}^n)),\\
        \mathbf{u}^{n+1}=&\mathbf{u}^n+\tau\varphi_{1}(\tau L)(L\mathbf{u}^n+N(\mathbf{u}^n))\\
        &+\tau\varphi_{2}(\tau L)(-3N(\mathbf{u}^{n})+2N(\mathbf{a}^n)+2N(\mathbf{b}^n)-N(\mathbf{c}^n))\\
        &+\tau\varphi_{3}(\tau L)(4N(\mathbf{u}^n)-4N(\mathbf{a}^n)-4N(\mathbf{b}^n)+4N(\mathbf{c}^n)),
    \end{split}
    \end{equation}
\end{itemize}
where these $\varphi$-functions are defined as \cite{Oster}:
\begin{equation*}
\varphi_1(z)=\frac{e^z-1}{z},\quad \varphi_2(z)=\frac{e^z-1-z}{z^2},\quad \varphi_3(z)=\frac{e^{z}-1-z-\frac12 z^2}{z^3},
\end{equation*}
and $\mathbf{u}^n$ is the numerical approximation to $\mathbf{u}(t^n)$.

It is clear that the ETD-RK1 scheme \eqref{eq:ETDRK1} is obtained by applying a first-order approximation $N(\mathbf{u}(t^n+s))\approx N(\mathbf{u}^n)$ in the integrand. 
When $N(\mathbf{u})\equiv 0$, all ETD-RK schemes reduce to the exact solution formula $\mathbf{u}^{n+1}=e^{\tau L}\mathbf{u}^{n}$.

When the right-hand side of the ODE system is fully nonlinear, a common approach (Rosenbrock-type treatment) \cite{hochbruck2009exponential, tokman2006efficient} is to extract a linear stiff component using the Jacobian matrix, while treating the remaining nonlinear part as non-stiff.
In particular, for the degenerate parabolic problem \eqref{eq:GeneralEquation}, the corresponding ODE system can be written as
\begin{equation}\label{eq:ODEs}
\mathbf{u}_{t}=D(\mathbf{u})+A(\mathbf{u}),
\end{equation}
where $D(\mathbf{u})$ and $A(\mathbf{u})$ arise from the DG discretizations \eqref{eq:DG_Diff} and \eqref{eq:DG_Conv} of the diffusion and convection terms, respectively.
Accordingly, we define $L=D'(\mathbf{u}^{n})$ as the Jacobian matrix of the diffusion term $D(\mathbf{u})$ evaluated at time $t^n$, and let $N(\mathbf{u})=D(\mathbf{u})+A(\mathbf{u})-D'(\mathbf{u}^{n})\mathbf{u}$ be the residual nonlinear term.
We then apply the ETD-RK schemes \eqref{eq:ETDRK1}--\eqref{eq:ETDRK4} to advance the solution in time.
At the end of each time step, the total variation bounded (TVB) limiter \cite{DG5} is applied to suppress spurious oscillations near sharp wave fronts in the solution profile.

\begin{rem}
Since the diffusion term is the source of stiffness, this treatment is expected to alleviate the severe time step restriction \( \tau \sim O(h^2) \) imposed by fully explicit methods.  
Indeed, in our previous study \cite{xu2025stability}, we showed that when the diffusion term is fully absorbed into the exponential integrator, the ETD-RKDG schemes for linear convection-diffusion equations exhibit enhanced stability compared to the pure advection case.  
The resulting scheme is *unconditionally stable* in a weak sense: \( \tau \leq C \), where \( C \) is a constant that depends only on the order of the ETD-RK time integrator and the coefficients of the PDE, and is independent of the spatial discretization.  

In the problems considered in this work, the diffusion term is nonlinear; consequently, only its Jacobian component is absorbed into the exponential integrator, and the previous conclusion may not fully apply.
To address this, we conduct a linear stability analysis for diffusion equations, demonstrating that the ETD-RK scheme remains unconditionally stable as long as the portion of the diffusion term absorbed by the exponential integrator constitutes a sufficiently large part of the total diffusion operator.  

Together, these results provide a heuristic justification for the enhanced stability from a theoretical perspective. Numerical evidence is presented in the numerical section.
\end{rem}

\begin{rem}

Fast computation of exponential and $\varphi$-functions of $L$ applied to vectors $\mathbf{b}$ (e.g., $e^{L}\mathbf{b}$, $\varphi_{i}(L)\mathbf{b}$, or their linear combinations) is a critical component for the efficiency of ETD-RK schemes.
Due to the large size but highly sparse structure of $L$, the Krylov subspace method \cite{TrefBau} is a standard approach.
The basic idea is to project the large-dimensional matrix $L$ and vector $\mathbf{b}$ onto a Krylov subspace of much smaller dimension $m\ll N$, and then compute \cite{moler2003nineteen} the matrix exponential and $\varphi$-functions in this reduced space.
Other techniques that accelerate the computation include time-stepping procedures with adaptive subspace dimensions and time step-sizes \cite{niesen2012algorithm}, and the incomplete orthogonalization method \cite{gaudreault2016efficient, luan2019further}.
In this work, we use the MATLAB package \emph{phipm}, which is based on the algorithm proposed in \cite{niesen2012algorithm}. For further details on fast computation, we refer the reader to Section 2.2.3 of our recent work \cite{xu2025high} and references therein.
\end{rem}


\subsubsection{Vectorized DG formulation and its Jacobian matrix}

In Section \ref{Sect:DGscheme}, we presented the variational formulation of the DG methods. 
To facilitate the fast computation and implementation of the ETD-RKDG methods, we introduce the vectorized DG formulation and its associated Jacobian matrix in this subsection. 
The periodic boundary condition is assumed throughout this section to simplify the discussion, but other boundary conditions can be incorporated into the scheme without essential difficulty.

We adopt the Lagrange (nodal) basis $\{\ell_{i}^{K}(\mathbf{x})\}_{i=1}^{N_k},\, K\in\mathcal{T}$ for the finite element space $V_{h}^{k}$, and represent the solution $u\in V_{h}^{k}$ as
\begin{equation}
u(\mathbf{x}, t)=\sum_{K\in\mathcal{T}}\sum_{i=1}^{N_k}u(\mathbf{x}_{i}^{K}, t)\ell_{i}^{K}(\mathbf{x}),
\end{equation}
where the nodal values $u(\mathbf{x}_{i}^{K}, t)$ at the Lagrange nodes $\{\mathbf{x}_{i}^{K}\}_{i=1}^{N_k},\,K\in\mathcal{T}$, serve as the DoFs of the numerical solution.
We denote the vector of local DoFs on an element $K$ by $\mathbf{u}^K=\big(u_{1}^{K},\ldots, u_{N_k}^{K}\big)^T\in\mathbb{R}^{N_k}$, and the global DoF vector by $\mathbf{u}=\big(\cdots,(\mathbf{u}^{K})^T,\cdots\big)^T\in\mathbb{R}^{N}$, where the concatenation is taken over all $K\in\mathcal{T}$.
For any scalar function $f$, we let $f(\mathbf{u}^K)$ and $f(\mathbf{u})$ denote the vectors obtained by applying $f$ component-wise to $\mathbf{u}^K$ and $\mathbf{u}$, respectively.

The local vectorized formulation of the DG scheme \eqref{eq:DG_Diff} in an element $K\in\mathcal{T}$ is given by
\begin{equation}\label{eq:DG_Diff_vect}
\begin{split}
\mathcal{M}_{K}\mathbf{u}^K_t=&\left((\mathcal{S}_{K}^{(2,0)})^T+(\mathcal{S}_{K}^{(0,2)})^T\right)g(\mathbf{u}^K)\\
&+\sum_{m=1}^{3}\left(
\frac{\mathbf{n}_{mx}^{K}}{2}\mathcal{B}_{m}^{K(1,0)}+\frac{\mathbf{n}_{my}^{K}}{2}\mathcal{B}_{m}^{K(0,1)}-\frac{\mathbf{n}_{mx}^{K}}{2}(\mathcal{B}_{m}^{K(1,0)})^T-\frac{\mathbf{n}_{my}^{K}}{2}(\mathcal{B}_{m}^{K(0,1)})^T
\right)g(\mathbf{u}^K)\\
&+\sum_{m=1}^{3}\left(
\frac{\mathbf{n}_{mx}^{K}}{2}\mathcal{B}_{m,E(m,K)}^{K,K_m(1,0)}
+\frac{\mathbf{n}_{my}^{K}}{2}\mathcal{B}_{m,E(m,K)}^{K,K_m(0,1)}
-\frac{\mathbf{n}_{mx}^{K}}{2}(\mathcal{B}_{E(m,K),m}^{K_{m},K(1,0)})^T
-\frac{\mathbf{n}_{my}^{K}}{2}(\mathcal{B}_{E(m,K),m}^{K_m,K(0,1)})^T
\right)g(\mathbf{u}^{K_m})\\
&+\beta\sum_{m=1}^{3}\mathcal{B}_{m,E(m,K)}^{K,K_m}\mathbf{u}^{K_m}
-\beta\sum_{m=1}^{3}\mathcal{B}_{m}^{K}\mathbf{u}^{K},
\end{split}
\end{equation}
and that of \eqref{eq:DG_Conv} is given by
\begin{equation}\label{eq:DG_Conv_vect}
\begin{split}
\mathcal{M}_{K}\mathbf{u}_{t}^{K}=&(\mathcal{S}_{K}^{(1,0)})^{T}f_1(\mathbf{u}^{K})+(\mathcal{S}_{K}^{(0,1)})^{T}f_{2}(\mathbf{u}^K)\\
&-\sum_{m=1}^{3}\frac{\mathbf{n}_{mx}^{K}}{2}\mathcal{B}_{m}^{K}f_{1}(\mathbf{u}^K)-\sum_{m=1}^{3}\frac{\mathbf{n}_{my}^{K}}{2}\mathcal{B}_{m}^{K}f_{2}(\mathbf{u}^K)\\
&-\sum_{m=1}^{3}\frac{\mathbf{n}_{mx}^{K}}{2}\mathcal{B}_{m,E(m,K)}^{K,K_m}f_1(\mathbf{u}^{K_m})-\sum_{m=1}^{3}\frac{\mathbf{n}_{my}^{K}}{2}\mathcal{B}_{m,E(m,K)}^{K,K_m}f_2(\mathbf{u}^{K_m})\\
&+\sum_{m=1}^{3}\frac{\alpha}{2}\mathcal{B}_{m,E(m,K)}^{K,K_m}\mathbf{u}^{K_m}-\sum_{m=1}^{3}\frac{\alpha}{2}\mathcal{B}_{m}^{K}\mathbf{u}^K,
\end{split}
\end{equation}
where the local mass and stiffness matrices are defined as
\begin{equation}\label{eq:LocalMatrices1}
\begin{split}
&({\mathcal{M}}_{K})_{ij}=\big<\ell_{i}^{K}, {\ell_{j}^{K}} \big>_{K},\quad ({\mathcal{S}}^{(1,0)}_{K})_{ij}=\big<\ell_{i}^{K}, \frac{\partial{\ell}_{j}^{K}}{\partial{{x}_1}}\big>_{K},\quad ({\mathcal{S}}^{(0,1)}_{K})_{ij}=\big<\ell_{i}^{K}, \frac{\partial{\ell}_{j}^{K}}{\partial{{x}_2}}\big>_{K},\\
&({\mathcal{S}}^{(2,0)}_{K})_{ij}=\big<\ell_{i}^{K}, \frac{\partial^2{\ell}_{j}^{K}}{\partial{{x}_1^2}}\big>_{K},\quad ({\mathcal{S}}^{(0,2)}_{K})_{ij}=\big<\ell_{i}^{K}, \frac{\partial^2{\ell}_{j}^{K}}{\partial{{x}_2^2}}\big>_{K},\quad ({\mathcal{S}}^{(1,1)}_{K})_{ij}=\big<\ell_{i}^{K}, \frac{\partial^2{\ell}_{j}^{K}}{\partial x_{1}\partial{{x}_2}}\big>_{K},
\end{split}
\end{equation}
and the local boundary matrices are defined as
\begin{equation}\label{eq:LocalMatrices2}
\begin{split}
&({\mathcal{B}}_{m}^{K})_{ij}=\big[{\ell}_{i}^K, {\ell}_j^K \big]_{e_{m}^{K}},\quad ({\mathcal{B}}^{K(1,0)}_{m})_{ij}=\big[{\ell}_{i}^K, \frac{\partial {\ell}_j^K}{\partial{x}_1 } \big]_{e_{m}^{K}},\quad ({\mathcal{B}}^{K(0,1)}_{m})_{ij}=\big[{\ell}_{i}^K, \frac{\partial{\ell}_j^K}{\partial{x}_{2} } \big]_{e_{m}^{K}},\\
&({\mathcal{B}}^{K, K'}_{m,n})_{ij}=\big[{\ell}_{i}^{K}, {\ell}^{K'}_j  \big]_{e_{m}^{K}},\quad ({\mathcal{B}}^{K, K'(1,0)}_{m,n})_{ij}=\big[{\ell}_{i}^K, \frac{\partial {\ell}^{K'}_j}{\partial{x}_1 } \big]_{e_{m}^{K}},\quad ({\mathcal{B}}^{K, K'(0,1)}_{m,n})_{ij}=\big[{\ell}_{i}^K, \frac{\partial {\ell}^{K'}_j}{\partial {x}_2} \big]_{e_{m}^{K}}.
\end{split}
\end{equation}

The global vectorized DG formulation for the problem \eqref{eq:GeneralEquation} can be written as
\begin{equation}\label{eq:DG_ConvDiff_Vect}
\mathbf{u}_{t}=G\, g(\mathbf{u})+F_1\,f_{1}(\mathbf{u})+F_2\,f_2(\mathbf{u})+(\beta+\frac{\alpha}{2})P\,\mathbf{u},
\end{equation}
where $G, F_1, F_2, P\in\mathbb{R}^{N\times N}$ are global matrices assembled from the local matrices in \eqref{eq:DG_Diff_vect} and \eqref{eq:DG_Conv_vect}.
To construct these global matrices, we first invert the local mass matrices $\mathcal{M}_{K}$ and apply them to the right-hand sides of the local equations. 
The resulting local terms are then assembled into the global matrices according to the following rules:
\begin{itemize}
  \item Local matrices multiplying \( g(\mathbf{u}^K) \), \( f_1(\mathbf{u}^K) \), \( f_2(\mathbf{u}^K) \), or \( \mathbf{u}^K \) are assembled into the {diagonal blocks} of \( G \), \( F_1 \), \( F_2 \), or \( P \), respectively, with {rows and columns corresponding to the DoFs on element \( K \)}.

  \item Local matrices multiplying \( g(\mathbf{u}^{K_m}) \), \( f_1(\mathbf{u}^{K_m}) \), \( f_2(\mathbf{u}^{K_m}) \), or \( \mathbf{u}^{K_m} \) are assembled into the {off-diagonal blocks}, with {rows corresponding to the DoFs on \( K \)} and {columns corresponding to the DoFs on the neighboring element \( K_m \)}.
\end{itemize}
The global matrices $G, F_1, F_2, P\in\mathbb{R}^{N\times N}$ only need to be computed once at the beginning of the program and should be stored using a sparse matrix data structure for efficient use.
We let
\begin{equation}
D(\mathbf{u})=G\, g(\mathbf{u})+(\beta+\frac{\alpha}{2})P\,\mathbf{u},\quad A(\mathbf{u})=F_1\,f_{1}(\mathbf{u})+F_2\,f_2(\mathbf{u})
\end{equation}
in the ODE formulation \eqref{eq:ODEs}.
The Jacobian matrix $L$ for the stiff component $D(\mathbf{u})$ can be efficiently computed by
\begin{equation}
L=G.*g'(\mathbf{u})^T + (\beta+\frac{\alpha}{2})\,P,
\end{equation}
where \( .* \) (borrowed from MATLAB syntax) denotes the operation of broadcasting the vector \( g'(\mathbf{u})^T \) to match the size of \( G \), followed by element-wise matrix multiplication.
(If one wishes to compute the Jacobian matrix of the entire right-hand side of \eqref{eq:ODEs} for use in a Rosenborg-type method, it can be expressed as $L = G .* g'(\mathbf{u})^T + F_1 .* f_1'(\mathbf{u})^T + F_2 .* f_2'(\mathbf{u})^T + \left(\beta + \frac{\alpha}{2}\right) P$.)

It is worth noting that a straightforward computation of the Lagrange basis functions on a simplex is not easily achieved.
Therefore, the local matrices in \eqref{eq:LocalMatrices1} and \eqref{eq:LocalMatrices2} are computed using the modal basis functions in \eqref{eq:FEMbasis}.
For completeness, we provide the details of their computation in Appendix B.

\begin{rem}
For a convection-diffusion-reaction equation of the form
\[u_t + \nabla \cdot \mathbf{F}(u) = \Delta g(u) + r(u),\quad \mathbf{x}\in\mathbb{R}^d,\]
the corresponding global vectorized DG formulation is given by
\[
\mathbf{u}_{t} = G\, g(\mathbf{u}) + F_1\, f_{1}(\mathbf{u}) + F_2\, f_2(\mathbf{u}) + r(\mathbf{u}) + \left(\beta + \frac{\alpha}{2}\right) P\, \mathbf{u},
\]
which is obtained by simply adding the reaction term $r(\mathbf{u})$ to the right-hand side, taking advantage of the nodal formulation.
\end{rem}

\subsubsection{Other time marching methods}

In the numerical section, we compare the performance of our ETD-RKDG algorithm with that of other time-marching methods.
We present two representatives of fourth-order strong stability preserving Runge-Kutta (SSP-RK) time integration methods: the fourth-order five-stage explicit SSP-RK method and the fourth-order four-stage diagonally implicit SSP-RK method, both of which will be used in the comparison.
For the ODE system
\begin{equation}
\mathbf{u}_t=R(\mathbf{u}),
\end{equation}
they are given as follows:
\begin{itemize}
\item Fourth-order five-stage explicit SSP-RK
\begin{equation*}
\begin{split}
{\bf u}^{(1)} = &~ \gamma_{10}{\bf u}^{n} + a_{10}\Delta t {R}({\bf u}^n),\\
{\bf u}^{(2)} = &~ \gamma_{20}{\bf u}^{n}
+\gamma_{21}{\bf u}^{(1)}+a_{21}\Delta t{R}({\bf u}^{(1)}),\\
{\bf u}^{(3)} = &~ \gamma_{30}{\bf u}^{n} + \gamma_{32}{\bf u}^{(2)}+ a_{32}\Delta t{R}({\bf u}^{(2)}),\\
{\bf u}^{(4)} = &~ \gamma_{40}{\bf u}^{n} + \gamma_{43}{\bf u}^{(3)} + a_{43}\Delta t {R}({\bf u}^{(3)}),\\
{\bf u}^{n+1} = &~ \gamma_{50}{\bf u}^{n} + \gamma_{52}{\bf u}^{(2)} +
\gamma_{53}{\bf u}^{(3)} + \gamma_{54}{\bf u}^{(4)} + a_{53} \Delta t {R}({\bf u}^{(3)}) + a_{54} \Delta t {R}({\bf u}^{(4)}),
\end{split}
\end{equation*}
where the coefficients are given as
\begin{equation*}
\begin{split}
&\gamma_{10}=1, \gamma_{20}=0.44437049406734, \gamma_{21}=0.55562950593266,\\
&\gamma_{30}=0.62010185138540,
\gamma_{32}=0.37989814861460,
\gamma_{40}=0.17807995410773,\\
&\gamma_{43}=0.82192004589227,
\gamma_{50}=0.00683325884039,
\gamma_{52}=0.51723167208978,\\
&\gamma_{53}=0.12759831133288,
\gamma_{54}=0.34833675773694,
\end{split}
\end{equation*}
and 
\begin{equation*}
\begin{split}
& a_{10}=0.39175222700392, a_{21}=0.36841059262959, a_{32}=0.25189177424738,\\
& a_{43}=0.54497475021237,
a_{53}=0.08460416338212, 
a_{54}=0.22600748319395.
\end{split}
\end{equation*}
\item Fourth-order four-stage diagonally implicit SSP-RK
\begin{equation}
\begin{split}
    \mathbf{u}^{(1)} = &~ \mathbf{u}^{n} + \Delta t(a_{11}R(\mathbf{u}^{(1)})),\\
    \mathbf{u}^{(2)} = &~ \mathbf{u}^{n} + \Delta t(a_{21}R(\mathbf{u}^{(1)})+a_{22}R(\mathbf{u}^{(2)})),\\
    \mathbf{u}^{(3)} = &~ \mathbf{u}^{n} + \Delta t(a_{31}R(\mathbf{u}^{(1)})+a_{32}R(\mathbf{u}^{(2)})+a_{33}R(\mathbf{u}^{(3)})),\\
    \mathbf{u}^{(4)} = &~ \mathbf{u}^{n} + \Delta t(a_{41}R(\mathbf{u}^{(1)})+a_{42}R(\mathbf{u}^{(2)})+a_{43}R(\mathbf{u}^{(3)})+a_{44}R(\mathbf{u}^{(4)})),\\
    \mathbf{u}^{n+1} = &~ \mathbf{u}^{n} + \Delta t(a_{51}R(\mathbf{u}^{(1)})+a_{52}R(\mathbf{u}^{(2)})+a_{53}R(\mathbf{u}^{(3)})+a_{54}R(\mathbf{u}^{(4)})),\\
\end{split}
\end{equation}
where the coefficients are given as
\begin{equation*}
\begin{split}
&a_{11}=0.119309657880174,
a_{21}=0.345451290033902,
a_{22}=0.070605579799433,\\
&a_{31}=0.2761333428381144, 
a_{32}=0.23720218154772324, 
a_{33}=0.070606483961727,\\
&a_{41}=0.25953180979776486, 
a_{42}=0.2229412458210437, 
a_{43}=0.2789071933072292, \\
&a_{44}=0.119309875536981,
a_{51}=0.27664498628127077, 
a_{52}=0.22335414879969517,\\
&a_{53}=0.22335532068802738, 
a_{54}=0.2766455442310059.
\end{split}
\end{equation*}

\end{itemize}

\section{Linear stability analysis}\label{Sect:analysis}

\subsection{Motivation}
In the previous sections, we developed ETD-RKDG methods for nonlinear degenerate parabolic equations.
By applying a Rosenborg-type treatment, we extract the linear stiff component of the diffusion term and absorb it exactly using an exponential integrating factor, while integrating the residue explicitly.
This approach essentially involves splitting a nonlinear diffusion equation 
\begin{equation*}
u_t=(a(u)u_x)_x
\end{equation*}
into the form
\begin{equation*}
u_t=a_0u_{xx}+\left((a(u)-a_0)u_{x}\right)_x,
\end{equation*}
where the first term on the right-hand side is treated exactly using an exponential integrator, and the second term is integrated explicitly. Here, $a_0$ is a constant chosen to be close to $a(u)$ at each timestep.

It is expected that such a numerical method will exhibit enhanced stability and allow for significantly larger time-step sizes.
To provide a \emph{heuristic} justification, we perform a linear stability analysis for the split form
\begin{equation}\label{eq:unscaled_eq}
u_t=a_0u_{xx} + (a-a_0)u_{xx},
\end{equation}
where the first term on the right-hand side is absorbed by the exponential integrating factor, and the second term is integrated explicitly. To facilitate the analysis, we assume that both $a_0$ and $a$ are positive constants.
A study in \cite{xu2025stability} reveals that the stability behavior of semi-discrete (discrete in time, continuous in space) ETD-RK methods closely aligns with that of the fully discrete ETD-RKDG methods.
Therefore, we focus on the analysis of the semi-discrete ETD-RK methods in this section and briefly outline the numerical investigation of the fully discrete ETD-RKDG methods to demonstrate consistency.

We shall answer two questions:
\begin{enumerate}
    \item How close $a_0$ to $a$ will ensure a stability?
    \item What kind of stability does the ETD-RK methods possess when $a_0$ closely approximate $a$?
\end{enumerate}

\subsection{Linear stability of semi-discrete ETD-RK methods}\label{Sect:Semi_Discrete_Fourier}
Before we proceed, a rescaling of the equation to its dimensionless form
\begin{equation}\label{eq:scaled_eq}
u_{t'}=\theta u_{x'x'}+(1-\theta)u_{x'x'}
\end{equation}
through change of variables
\begin{equation}
t'=t/\tau'\quad\text{and}\quad x'=x/\sqrt{a\tau'}
\end{equation}
would greatly simplify the analysis.
Here, $\tau'>0$ is an arbitrary positive constant and $\theta=\frac{a_0}{a}>0$ denotes the ratio of the stength of the two parts.
If a time marching method is stable when applied to \eqref{eq:scaled_eq} with the time-step size $\tau=1$ and for a certain $\theta>0$, then the same method is stable for \eqref{eq:unscaled_eq} with the time-step size $\tau=\tau'$, $a_0=\theta a$.

The semi-discrete (discrete in time, continuous in space) ETD-RK methods for \eqref{eq:scaled_eq} with the time-step size $\tau=1$ is given as
\begin{itemize}
\item ETD-RK1:
\begin{equation}\label{eq:ETDRK1_semi}
\begin{split}
{u}^{n+1}=&e^{ \theta\partial_{xx}}{u}^{n} + \theta^{-1}(1-\theta)\left(e^{ \theta\partial_{xx}}-I\right){u}^n,
\end{split}
\end{equation}
\item ETD-RK2:
   \begin{equation}\label{eq:ETDRK2_semi}
   \begin{split}
       {a}^n=&{u}^n+ \varphi_1 ( \theta\partial_{xx}) \partial_{xx}{u}^{n},\\
       {u}^{n+1}=&{a}^n+\varphi_{2}( \theta\partial_{xx})(-(1-\theta)\partial_{xx}{u}^n+(1-\theta)\partial_{xx}{a}^n),
   \end{split}
   \end{equation}
\item ETD-RK3:
    \begin{equation}\label{eq:ETDRK3_semi}
    \begin{split}
        {a}^n=&{u}^n+\frac{1}{2}\varphi_1(\frac{1}{2}\theta\partial_{xx})\partial_{xx}{u}^n,\\
        {b}^n=&{u}^n+\varphi_{1}( \theta\partial_{xx})((2\theta-1)\partial_{xx}{u}^n+2(1-\theta)\partial_{xx}{a}^n),\\
        {u}^{n+1}=&{u}^n+\varphi_{1}( \theta\partial_{xx})\partial_{xx}{u}^n\\
        &+\varphi_{2}( \theta\partial_{xx})(-3(1-\theta)\partial_{xx}{u}^n+4(1-\theta)\partial_{xx}{a}^n-(1-\theta)\partial_{xx}{b}^n)\\
        &+\varphi_{3}( \theta\partial_{xx})(4(1-\theta)\partial_{xx}{u}^n-8(1-\theta)\partial_{xx}{a}^n+4(1-\theta)\partial_{xx}{b}^n),
    \end{split}
    \end{equation}
\item ETD-RK4:
    \begin{equation}\label{eq:ETDRK4_semi}
    \begin{split}
        {a}^n=&{u}^n+\frac{1}{2}\varphi_{1}(\frac{1}{2}\theta\partial_{xx})(\partial_{xx}{u}^n),\\
        {b}^n=&{u}^{n}+\frac{1}{2}\varphi_{1}(\frac{1}{2} \theta\partial_{xx})(\theta\partial_{xx}{u}^n+(1-\theta)\partial_{xx}{a}^n),\\
        {c}^n=&{a}^n+\frac{1}{2}\varphi_{1}(\frac{1}{2} \theta\partial_{xx})(-(1-\theta)\partial_{xx}{u}^n+\theta\partial_{xx}{a}^n+2(1-\theta)\partial_{xx}{b}^n),\\
        {u}^{n+1}=&{u}^n+\varphi_{1}( \theta\partial_{xx})\partial_{xx}{u}^n\\
        &+\varphi_{2}( \theta\partial_{xx})(-3(1-\theta)\partial_{xx}{u}^{n}+2(1-\theta)\partial_{xx}{a}^n+2(1-\theta)\partial_{xx}{b}^n-(1-\theta)\partial_{xx}{c}^n)\\
        &+\varphi_{3}( \theta\partial_{xx})(4(1-\theta)\partial_{xx}{u}^n-4(1-\theta)\partial_{xx}{a}^n-4(1-\theta)\partial_{xx}{b}^n+4(1-\theta)\partial_{xx}{c}^n),
    \end{split}
    \end{equation}
\end{itemize}
We now assume $u^{n}, u^{n+1}\in L^{2}(\mathbb{R})$ and study the growth pattern of their $L^2$ norms.
Denoting the Fourier transform of a function $f\in L^{2}(\mathbb{R})$ by $\mathcal{F}[f]=\widehat{f}(\xi)\in L^{2}(\mathbb{R})$, it is well-known that $\mathcal{F}:L^{2}(\mathbb{R})\rightarrow L^2(\mathbb{R})$ is an isometric transform (up to a constant factor), with the property,
\begin{equation*}
\mathcal{F}[\partial_x^\alpha f]=(\ii\xi)^{\alpha}\mathcal{F}[f],
\end{equation*}
where $\ii=\sqrt{-1}$ is the imaginary unit.
We apply the Fourier transform to both sides of the semidiscrete ETD-RK methods \eqref{eq:ETDRK1_semi} -- \eqref{eq:ETDRK4_semi} to obtain
\begin{equation}\label{eq:Fourier_growth}
\widehat{u}^{n+1}(\xi)=\widehat{G}(\theta,\xi)\widehat{u}^{n}(\xi),
\end{equation}
where $\widehat{G}(\theta,\xi)$ are the growth factors in the Fourier space defined as 
\begin{itemize}
\item ETD-RK1:
\begin{equation}\label{eq:ETDRK1_G}
\begin{split}
\widehat{G}(\theta,\xi)=&e^{ -\theta\xi^2} + \theta^{-1}(1-\theta)\left(e^{ -\theta\xi^2}-1\right),
\end{split}
\end{equation}

\item ETD-RK2:
   \begin{equation}\label{eq:ETDRK2_G}
   \begin{split}
       \widehat{a}^n=&1- \varphi_1 ( -\theta\xi^2) \xi^2,\\
       \widehat{G}(\theta,\xi)=&\widehat{a}^n+\varphi_{2}( -\theta\xi^2)((1-\theta)\xi^2-(1-\theta)\xi^2\widehat{a}^n),
   \end{split}
   \end{equation}
\item ETD-RK3: 
    \begin{equation}\label{eq:ETDRK3_G}
    \begin{split}
        \widehat{a}^n=&1-\frac{1}{2}\varphi_1(-\frac{1}{2}\theta\xi^2)\xi^2,\\
        \widehat{b}^n=&1+\varphi_{1}(-\theta\xi^2)(-(2\theta-1)\xi^2-2(1-\theta)\xi^2\widehat{a}^n),\\
        \widehat{G}(\theta,\xi)=&1-\varphi_{1}( -\theta\xi^2)\xi^2\\
        &+\varphi_{2}(-\theta\xi^2)(3(1-\theta)\xi^2-4(1-\theta)\xi^2\widehat{a}^n+(1-\theta)\xi^2\widehat{b}^n)\\
        &+\varphi_{3}(-\theta\xi^2)(-4(1-\theta)\xi^2+8(1-\theta)\xi^2\widehat{a}^n-4(1-\theta)\xi^2\widehat{b}^n),
    \end{split}
    \end{equation}
\item ETD-RK4:
    \begin{equation}\label{eq:ETDRK4_G}
    \begin{split}
        \widehat{a}^n=&1-\frac{1}{2}\varphi_{1}(-\frac{1}{2}\theta\xi^2)\xi^2,\\
        \widehat{b}^n=&1+\frac{1}{2}\varphi_{1}(-\frac{1}{2} \theta\xi^2)(-\theta\xi^2-(1-\theta)\xi^2\widehat{a}^n),\\
        \widehat{c}^n=&\widehat{a}^n+\frac{1}{2}\varphi_{1}(-\frac{1}{2} \theta\xi^2)((1-\theta)\xi^2-\theta\xi^2\widehat{a}^n-2(1-\theta)\xi^2\widehat{b}^n),\\
        \widehat{G}(\theta,\xi)=&1-\varphi_{1}(- \theta\xi^2)\xi^2\\
        &+\varphi_{2}(-\theta\xi^2)(3(1-\theta)\xi^2-2(1-\theta)\xi^2\widehat{a}^n-2(1-\theta)\xi^2\widehat{b}^n+(1-\theta)\xi^2\widehat{c}^n)\\
        &+\varphi_{3}(-\theta\xi^2)(-4(1-\theta)\xi^2+4(1-\theta)\xi^2\widehat{a}^n+4(1-\theta)\xi^2\widehat{b}^n-4(1-\theta)\xi^2\widehat{c}^n),
    \end{split}
    \end{equation}
\end{itemize}
From \eqref{eq:Fourier_growth} and the isometry of the Fourier transform, we have $||u^{n+1}||_{L^2(\mathbb{R})}\leq ||u^{n}||_{L^2(\mathbb{R})}$ if $|\widehat{G}(\theta,\xi)|\leq 1, \forall \xi\in\mathbb{R}$.
The following theorem illustrates the stability condition of the semi-discrete ETD-RK1 method \eqref{eq:ETDRK1_semi}.
\begin{thm}
The semidiscrete ETD-RK1 scheme \eqref{eq:ETDRK1_semi} is stable under the condition $\theta\geq\theta_0:=\frac{1}{2}$, with the growth factor $|\widehat{G}(\theta,\xi)|\leq1$ for all $\xi$.
\end{thm}
\begin{proof}
One can verify that $\widehat{G}(\theta,0) = 1$ and $\lim_{\xi\rightarrow\infty}\widehat{G}(\theta,\xi) = \frac{\theta - 1}{\theta}$.  
Since $\widehat{G}(\theta,\xi)$ is an even function of $\xi$, and
\begin{equation*}
\frac{\partial \widehat{G}(\theta,\xi)}{\partial \xi} = -2\xi e^{-\theta \xi^2} < 0, \quad \text{for } \xi > 0,
\end{equation*}
it follows that $\widehat{G}(\theta,\xi)$ is decreasing on the interval $\xi \in [0,\infty)$ and symmetric with respect to the $y$-axis.
Therefore, we have
\[
\sup_{\xi\in\mathbb{R}}|\widehat{G}(\theta,\xi)| = \max\left\{|\widehat{G}(\theta,0)|, |\widehat{G}(\theta,\infty)|\right\}=\max\{1,\left|\frac{\theta-1}{\theta}\right|\} = 1,
\]
for $\theta \geq\theta_0:=\frac{1}{2}$.
\end{proof}
In addition, one can calculate that $\widehat{G}(\theta,0)=1$ and 
\begin{equation}
\lim_{\xi \rightarrow \infty} \widehat{G}(\theta,\xi) = 
\frac{(\theta - 1)^2}{\theta^2}, \quad 
\frac{(\theta - 1)^2(\theta - 2)}{\theta^3}, \quad \text{and} \quad 
\frac{(\theta - 1)^2(\theta^2 - 4\theta + 2)}{\theta^4},
\end{equation}
for ETD-RK2, ETD-RK3, and ETD-RK4, respectively.
As a necessary condition for $\left|\widehat{G}(\theta,\xi)\right| \leq 1$, we require $\theta \geq \theta_0$, where
\begin{equation}
\theta_0 = \frac{1}{2}, \quad 
\frac{2}{3} + \sqrt[3]{-\frac{1}{27} + \frac{\sqrt{78}}{36}} + \sqrt[3]{-\frac{1}{27} - \frac{\sqrt{78}}{36}} \approx 0.6034, \quad \text{and} \quad \frac{1}{2}
\end{equation}
for ETD-RK2, ETD-RK3, and ETD-RK4, respectively, to ensure that $\left|\widehat{G}(\theta,\infty)\right| \leq 1$.
On the other hand, from the graphs of $\widehat{G}(\theta,\xi)$ versus $\xi \in [0,\infty)$, we observe that the maximum of $\left|\widehat{G}(\theta,\xi)\right|$ is always attained at either $\xi = 0$ or $\xi = \infty$ when $\theta \geq \theta_0$, making the condition also sufficient. A typical plot of $|\widehat{G}(\theta_0,\xi)|$ is shown in Figure~\ref{fig:GrowthFactors}, verifying our discussion.
\begin{figure}[htbp!]
 \centering
 \begin{subfigure}[b]{0.35\textwidth}
  \includegraphics[width=\textwidth]{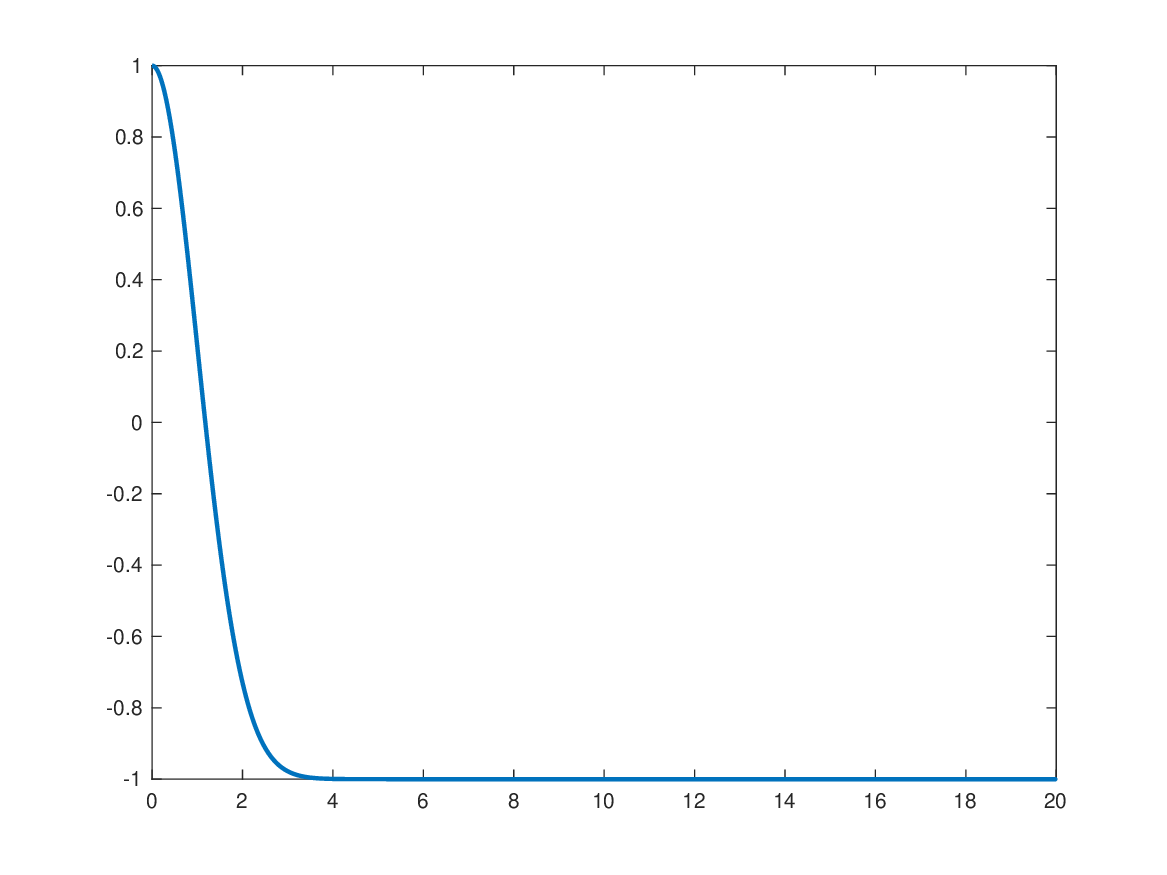}
  \caption{$\widehat{G}(\theta_0,\xi)$ for ETD-RK1}
 \end{subfigure}
 \begin{subfigure}[b]{0.35\textwidth}
  \includegraphics[width=\textwidth]{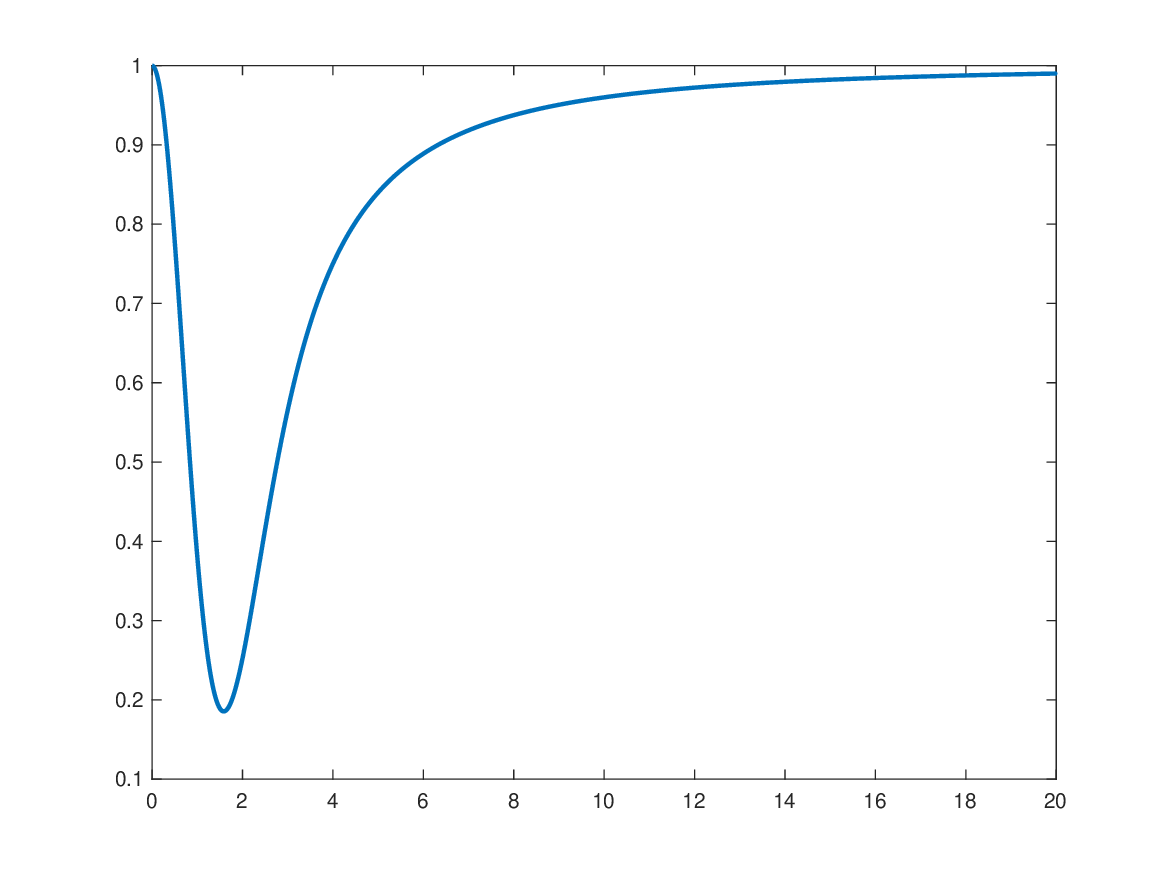}
  \caption{$\widehat{G}(\theta_0,\xi)$ for ETD-RK2}
 \end{subfigure}\\
 \begin{subfigure}[b]{0.35\textwidth}
  \includegraphics[width=\textwidth]{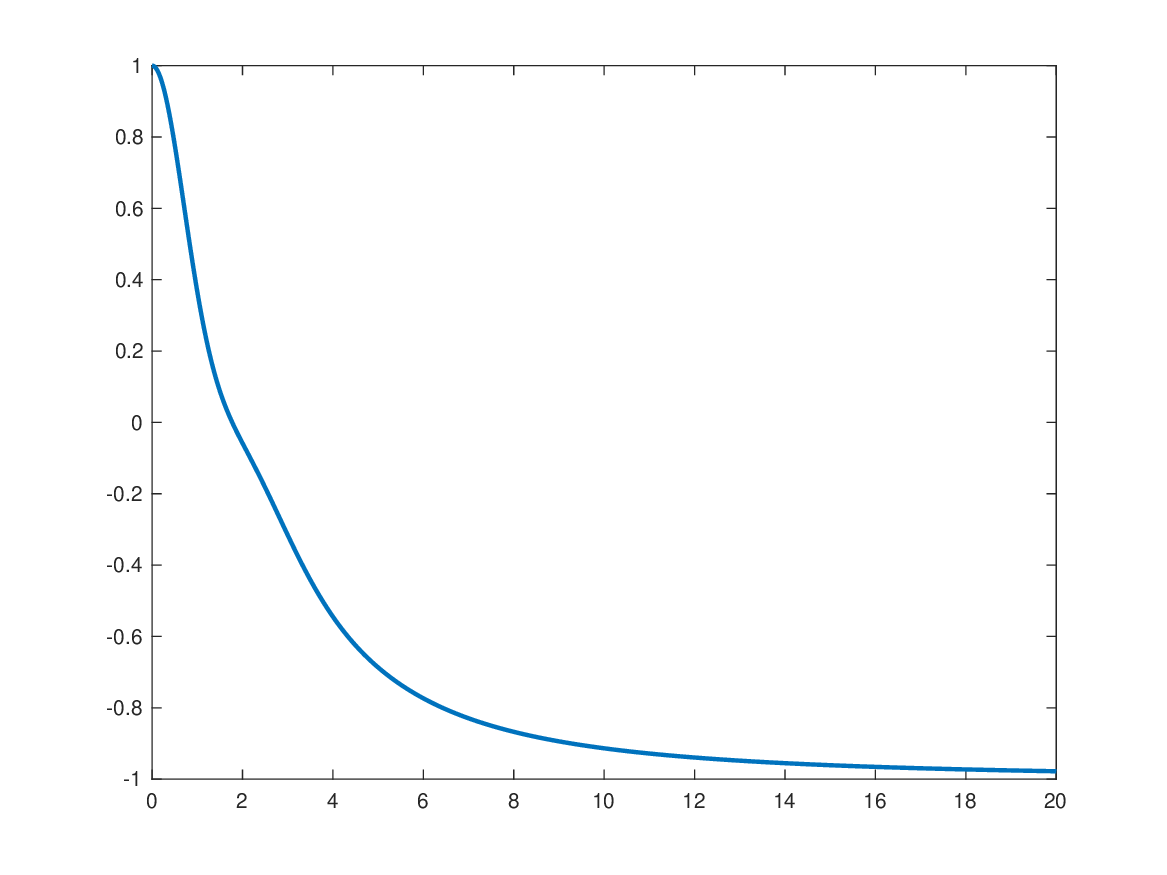}
  \caption{$\widehat{G}(\theta_0,\xi)$ for ETD-RK3}
 \end{subfigure} 
 \begin{subfigure}[b]{0.35\textwidth}
  \includegraphics[width=\textwidth]{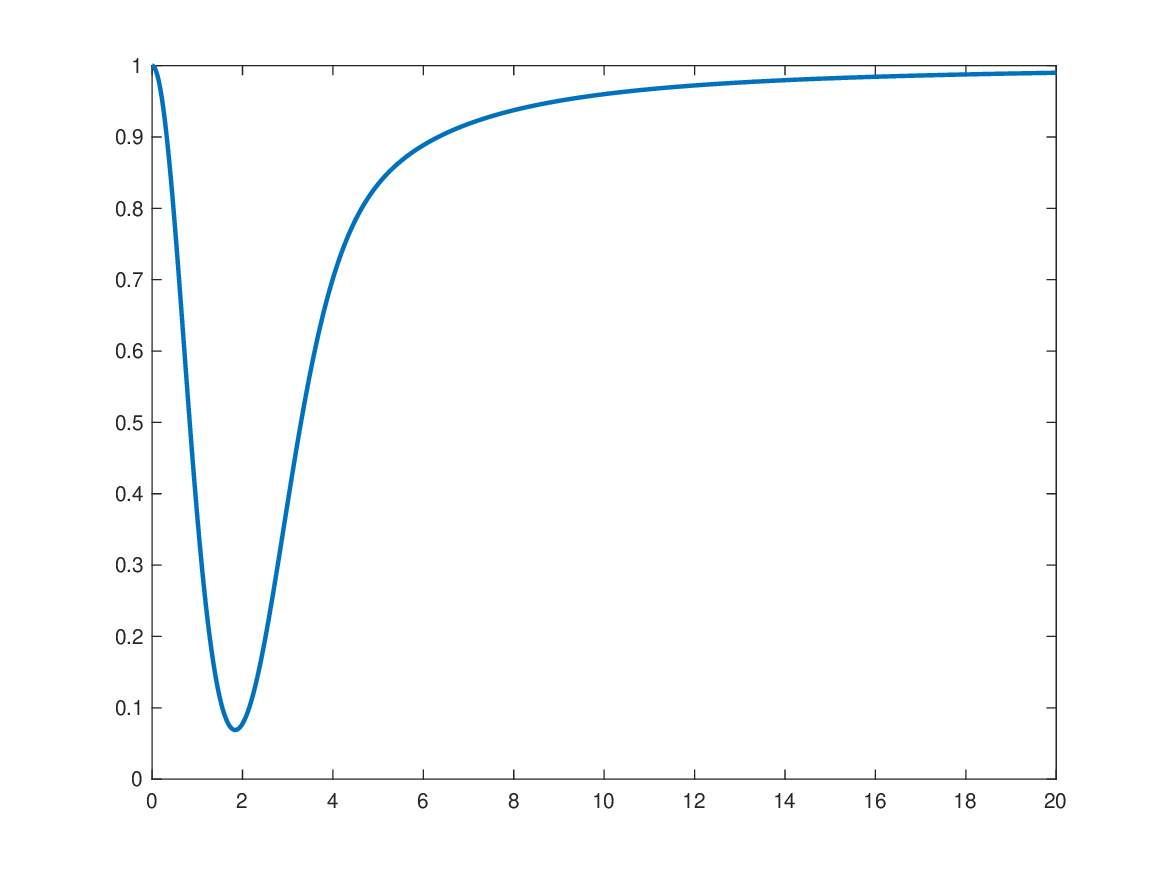}
  \caption{$\widehat{G}(\theta_0,\xi)$ for ETD-RK4}
 \end{subfigure}  
 \caption{The growth factor $\widehat{G}(\theta_0,\xi)$, with $\theta_0 = \frac{1}{2}, \frac{1}{2}, \frac{2}{3} + \sqrt[3]{-\frac{1}{27} + \frac{\sqrt{78}}{36}} + \sqrt[3]{-\frac{1}{27} - \frac{\sqrt{78}}{36}} \approx 0.6034$, and $\frac{1}{2}$ for ETD-RK1, ETD-RK2, ETD-RK3, and ETD-RK4 respectively.}
 \label{fig:GrowthFactors}
\end{figure}

\begin{rem}
We have analyzed the stability of the diffusion equation $u_t=a_0 u_{xx}+(a-a_0)u_{xx}$ and revealed its unconditional stability when $a_0 \geq \theta_0 a$, indicating that the ETD-RKDG method should provide excellent stability for the porous medium equation, which is a purely diffusive equation.
For the advection-diffusion equation, the enhancement in stability is studied in detail in \cite{xu2025stability}. Although the analysis is limited to linear problems in one-dimensional space, the study provides justification for the expected good stability of the ETD-RKDG methods.
In the next subsection, we present extensive numerical tests to further support the improvement in stability.
\end{rem}

\subsection{Linear stability of fully discrete ETD-RKDG methods}

In this subsection, we numerically investigate the stability of the fully discrete ETD-RKDG methods for \eqref{eq:scaled_eq}, and show its consistency with that of the semi-discrete ETD-RK methods.
To facilitate the Fourier analysis \cite{zhang2003analysis, xu2025stability}, we assume periodic boundary conditions and use uniform meshes.

Consider the domain $\Omega=[0,2\pi]$ with the partition $\Omega=\cup_{j=1}^{N} I_{j}=[x_{j-\frac12}, x_{j+\frac12}]$, where $x_{j+\frac12}=jh$ and $h=\frac{2\pi}{N}$ for $j=0,1,\ldots, N$.
Similar to the setup in Section \ref{eq:prelim}, we adopt the finite element space of piecewise polynomials
\begin{equation}
V^k_{h}=\{v\in L^2(\Omega): v|_{I_{j}}\in \mathcal{P}^{k}(I_{j}),\,~ \text{for}~~ j=1,2,\ldots,N\},
\end{equation}
and use values at the $(k+1)$-point LGL nodes on each cell $I_{j}$ as the DoFs for a numerical solution $u\in V_{h}^{k}$.

Using a one-dimensional version of the variational formulation \eqref{eq:DG_Diff}, we obtain the local vectorized formulation of the DG scheme on the element $I_j$ as follows:
\begin{equation}\label{eq:DG_Diff_vect1D}
\mathbf{u}^{j}_{t}=\theta\big(D_{-1}\mathbf{u}^{j-1}+D_{0}\mathbf{u}^j+D_{1}\mathbf{u}^{j+1}\big)+(1-\theta)\big(D_{-1}\mathbf{u}^{j-1}+D_{0}\mathbf{u}^j+D_{1}\mathbf{u}^{j+1}\big),
\end{equation}
where $\mathbf{u}^{j}\in\mathbb{R}^{k+1}$ is the vector of DoFs of the solution $u\in V_{h}^{k}$ on $I_{j}$ for $j=1,2,\ldots, N$, and the superscripts $j\pm1$ are understood in the cyclic sense.

Consider a Fourier mode $\mathbf{u}^{j}(t)=\widehat{\mathbf{u}}(t)e^{\ii\omega jh}$ for $j=1,\ldots,N$.
Here, $\omega$ denotes the frequency of the mode, ranging from $-\left\lfloor\frac{ N }{2}\right\rfloor$ to $\left\lfloor\frac{ N }{2}\right\rfloor$.
Substituting this into the matrix equation \eqref{eq:DG_Diff_vect1D}, we obtain the following equation governing the growth of the mode:
\begin{equation}\label{eq:ModelEq}
\frac{\dd \widehat{\mathbf{u}}(t)}{\dd t}=\theta\widehat{D}(h,\xi)\widehat{\mathbf{u}}+(1-\theta)\widehat{D}(h,\xi)\widehat{\mathbf{u}},
\end{equation}
where $\widehat{D}(h,\xi)=D_{-1}e^{-\ii\xi}+D_{0}+D_{1}e^{\ii\xi}$ and $\xi=\omega h\in[-\pi,\pi]$.

We apply the ETD-RK schemes \eqref{eq:ETDRK1} -- \eqref{eq:ETDRK4} to the model equation \eqref{eq:ModelEq}, with the treatment that the first part is absorbed into the integrating factor, while the second part is integrated explicitly.
The resulting update equations all share the following common form:
\begin{equation}
\widehat{\mathbf{u}}^{n+1}=\widehat{G}(\theta,h,\xi)\widehat{\mathbf{u}}^{n},
\end{equation}
where $\widehat{G}(\theta,h,\xi)\in\mathbb{R}^{(k+1)\times(k+1)}$ denotes the matrix growth factor of the fully discrete ETD-RKDG scheme for the Fourier mode.
These growth factors are analogues of \eqref{eq:ETDRK1_G}--\eqref{eq:ETDRK4_G}, with $-\xi^2$ replaced by $\widehat{D}(h,\xi)$.
We denote by $\rho(\widehat{G}(\theta,h,\xi))$ the spectral radius of the matrix growth factor, and seek the range of $\theta$ such that the stability condition
\begin{equation}
\sup_{\xi\in[-\pi,\pi]} \rho\big(\widehat{G}(\theta,h,\xi)\big)\leq 1, \quad \forall h>0.
\end{equation}
is satisfied.
We sample a sufficient number of $\xi$ values in $[-\pi, \pi]$ and consider various values of $h > 0$. 
An extensive numerical search reveals that the spectral radii of the growth factors remain bounded above by 1 as long as $\theta \geq \theta_0$, where the values of $\theta_0$ coincide with those obtained for the semi-discrete ETD-RK methods in Section~\ref{Sect:Semi_Discrete_Fourier}.
The stability results are independent of the spatial discretization, including both the mesh size $h$ and the polynomial degree $k$.

\section{Numerical tests}\label{Sect:tests}

In this section, we validate the performance of our scheme through numerical tests. 

\begin{exmp}
\textbf{Accuracy tests}\label{ex:accuracy}
\end{exmp}
In this example, we test the accuracy of our method on both linear and nonlinear problems.
We first consider the linear convection-diffusion equation \eqref{eq:conv-diff}
\begin{equation}\label{eq:conv-diff}
u_t + u_x + u_y = u_{xx} + u_{yy},
\end{equation}
on the domain $\Omega = [0, 2\pi]^2$ with periodic boundary conditions. The initial condition is given by $u_0(x,y) = \sin(x)\sin(y)$, and the exact solution is consequently $u(x,y,t) = e^{-2t}\sin(x - t)\sin(y - t)$.

We then consider the nonlinear diffusion-reaction equation \eqref{eq:diff-react}
\begin{equation}\label{eq:diff-react}
u_t = \Delta u^2 + r(u),
\end{equation}
with the reaction term $r(u) = (u^2 - 2)\left(2 - \frac{1}{u}\right)$ on the same domain $\Omega = [0, 2\pi]^2$. Periodic boundary conditions are applied, and the initial condition is $u_0(x,y) = \sqrt{\sin(x)\sin(y) + 2}$. The exact solution is thus given by $u(x,y,t) = \sqrt{e^{-2t}\sin(x)\sin(y) + 2}$.

We test different combinations of ETD-RK and DG methods on various meshes with different levels of refinement.
The coarsest mesh (level $0$) is shown in Figure \ref{fig:mesh_level0}.
Meshes at level $i$ are obtained by refining each triangle from level $i-1$ into four sub-triangles by connecting the midpoints of its edges.
The $L^2$ errors and orders of convergence for the linear problem \eqref{eq:conv-diff} and the nonlinear problem \eqref{eq:diff-react} are given in Tables \ref{tab:accuracy_linear} and \ref{tab:accuracy_nonlinear}, respectively.

\begin{figure}[!htbp]
 \centering
  \includegraphics[width=0.5\textwidth]{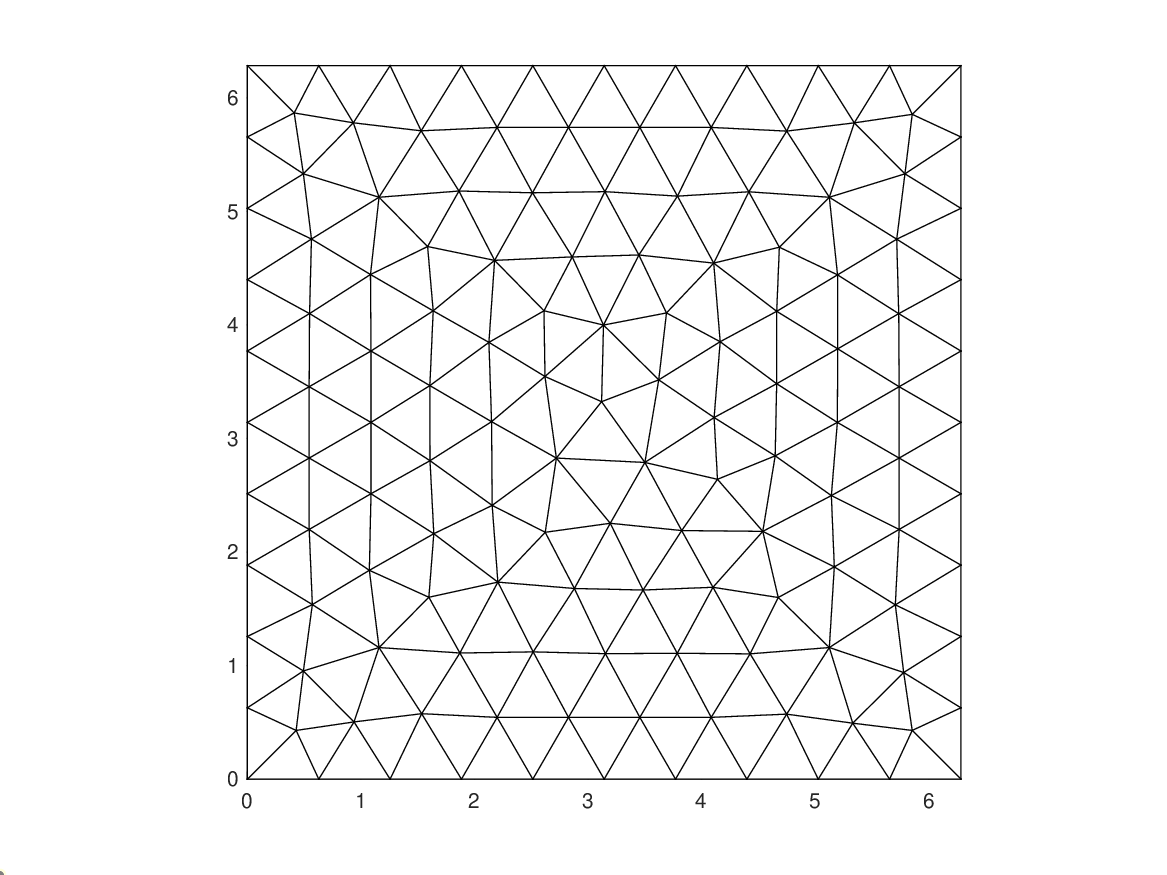}
 \caption{\textbf{Example \ref{ex:accuracy}. Accuracy test.} 
Computational mesh of level $0$ (coarsest).}
 \label{fig:mesh_level0}
\end{figure}

\begin{table}[!htbp]
\centering
\begin{tabular}{ccccccc}
\toprule[1.5pt]
\multicolumn{7}{c}{ETD-RK1} \\
\midrule
& $\mathcal{P}^1$-DG & & $\mathcal{P}^2$-DG & & $\mathcal{P}^3$-DG\\
\cline{2-3} \cline{4-5} \cline{6-7}  
$i$ & $L^2$ Error & Order & $L^2$ Error & Order& $L^2$ Error & Order\\
\midrule
$0$ & $1.20\times10^{-1}$ & -- & $1.58\times10^{-1}$ & -- & $1.58\times10^{-1}$ & --\\
$1$ & $6.16\times10^{-2}$ & $0.97$ & $7.00\times10^{-2}$ & $1.17$ & $7.00\times10^{-2}$ & $1.18$\\
$2$ & $3.14\times10^{-2}$ & $0.97$ & $3.34\times10^{-2}$ & $1.07$ & $3.34\times10^{-2}$ & $1.07$\\
$3$ & $1.57\times10^{-2}$ & $1.00$ & $1.62\times10^{-2}$ & $1.05$ & $1.62\times10^{-2}$ & $1.05$\\
\midrule
\multicolumn{7}{c}{ETD-RK2} \\
\midrule
& $\mathcal{P}^1$-DG & & $\mathcal{P}^2$-DG & & $\mathcal{P}^3$-DG\\
\cline{2-3} \cline{4-5} \cline{6-7}  
$i$ & $L^2$ Error & Order & $L^2$ Error & Order& $L^2$ Error & Order\\
\midrule
$0$ & $5.93\times10^{-2}$ & -- & $1.37\times10^{-2}$ & -- & $1.36\times10^{-2}$ & --\\
$1$ & $1.58\times10^{-2}$ & $1.91$ & $3.42\times10^{-3}$ & $2.00$ & $3.42\times10^{-3}$ & $1.99$\\
$2$ & $4.05\times10^{-3}$ & $1.97$ & $8.71\times10^{-4}$ & $1.97$ & $8.71\times10^{-4}$ & $1.97$\\
$3$ & $1.02\times10^{-3}$ & $1.99$ & $2.17\times10^{-4}$ & $2.01$ & $2.17\times10^{-4}$ & $2.01$\\
\midrule
\multicolumn{7}{c}{ETD-RK3} \\
\midrule
& $\mathcal{P}^1$-DG & & $\mathcal{P}^2$-DG & & $\mathcal{P}^3$-DG\\
\cline{2-3} \cline{4-5} \cline{6-7}  
$i$ & $L^2$ Error & Order & $L^2$ Error & Order& $L^2$ Error & Order\\
\midrule
$0$ & $5.76\times10^{-2}$ & -- & $1.61\times10^{-3}$ & -- & $9.09\times10^{-4}$ & --\\
$1$ & $1.51\times10^{-2}$ & $1.93$ & $1.82\times10^{-4}$ & $3.14$ & $1.14\times10^{-4}$ & $3.00$\\
$2$ & $3.81\times10^{-3}$ & $1.98$ & $2.20\times10^{-5}$ & $3.05$ & $1.45\times10^{-5}$ & $2.97$\\
$3$ & $9.55\times10^{-4}$ & $2.00$ & $2.68\times10^{-6}$ & $3.03$ & $1.80\times10^{-6}$ & $3.01$\\
\midrule
\multicolumn{7}{c}{ETD-RK4} \\
\midrule
& $\mathcal{P}^1$-DG & & $\mathcal{P}^2$-DG & & $\mathcal{P}^3$-DG\\
\cline{2-3} \cline{4-5} \cline{6-7}  
$i$ & $L^2$ Error & Order & $L^2$ Error & Order& $L^2$ Error & Order\\
\midrule
$0$ & $5.71\times10^{-2}$ & -- & $1.08\times10^{-3}$ & -- & $9.23\times10^{-5}$ & --\\
$1$ & $1.50\times10^{-2}$ & $1.93$ & $1.26\times10^{-4}$ & $3.09$ & $5.95\times10^{-6}$ & $3.95$\\
$2$ & $3.80\times10^{-3}$ & $1.98$ & $1.55\times10^{-5}$ & $3.03$ & $3.81\times10^{-7}$ & $3.97$\\
$3$ & $9.54\times10^{-4}$ & $1.99$ & $1.93\times10^{-6}$ & $3.01$ & $2.38\times10^{-8}$ & $4.00$\\
\bottomrule[1.5pt]
\end{tabular}
\caption{\textbf{Example \ref{ex:accuracy}. Accuracy test: linear problem.} 
$L^2$ errors and orders of convergence for the linear convection-diffusion equation \eqref{eq:conv-diff}, solved using various ETD-RKDG methods. 
Unstructured grids with different levels of refinement are used, and the time-step size is set to $\tau = h$.}
\label{tab:accuracy_linear}
\end{table}

\begin{table}[!htbp]
\centering
\begin{tabular}{ccccccc}
\toprule[1.5pt]
\multicolumn{7}{c}{ETD-RK1} \\
\midrule
& $\mathcal{P}^1$-DG & & $\mathcal{P}^2$-DG & & $\mathcal{P}^3$-DG\\
\cline{2-3} \cline{4-5} \cline{6-7}  
$i$ & $L^2$ Error & Order & $L^2$ Error & Order& $L^2$ Error & Order\\
\midrule
$0$ & $1.07\times10^{-1}$ & -- & $2.26\times10^{-1}$ & -- & $2.27\times10^{-1}$ & --\\
$1$ & $9.63\times10^{-2}$ & $0.15$ & $1.36\times10^{-1}$ & $0.73$ & $1.36\times10^{-1}$ & $0.74$\\
$2$ & $6.35\times10^{-2}$ & $0.60$ & $7.46\times10^{-2}$ & $0.86$ & $7.46\times10^{-2}$ & $0.86$\\
$3$ & $3.62\times10^{-2}$ & $0.81$ & $3.92\times10^{-2}$ & $0.93$ & $3.94\times10^{-2}$ & $0.86$\\
\midrule
\multicolumn{7}{c}{ETD-RK2} \\
\midrule
& $\mathcal{P}^1$-DG & & $\mathcal{P}^2$-DG & & $\mathcal{P}^3$-DG\\
\cline{2-3} \cline{4-5} \cline{6-7}  
$i$ & $L^2$ Error & Order & $L^2$ Error & Order& $L^2$ Error & Order\\
\midrule
$0$ & $1.17\times10^{-1}$ & -- & $6.10\times10^{-2}$ & -- & $6.32\times10^{-2}$ & --\\
$1$ & $3.16\times10^{-2}$ & $1.89$ & $1.76\times10^{-2}$ & $1.79$ & $1.78\times10^{-2}$ & $1.83$\\
$2$ & $7.96\times10^{-3}$ & $1.99$ & $4.69\times10^{-3}$ & $1.91$ & $4.70\times10^{-3}$ & $1.92$\\
$3$ & $1.98\times10^{-3}$ & $2.01$ & $1.21\times10^{-3}$ & $1.96$ & $1.21\times10^{-3}$ & $1.96$\\
\midrule
\multicolumn{7}{c}{ETD-RK3} \\
\midrule
& $\mathcal{P}^1$-DG & & $\mathcal{P}^2$-DG & & $\mathcal{P}^3$-DG\\
\cline{2-3} \cline{4-5} \cline{6-7}  
$i$ & $L^2$ Error & Order & $L^2$ Error & Order& $L^2$ Error & Order\\
\midrule
$0$ & $1.80\times10^{-1}$ & -- & $2.85\times10^{-3}$ & -- & $4.89\times10^{-3}$ & --\\
$1$ & $4.85\times10^{-2}$ & $1.89$ & $5.51\times10^{-4}$ & $2.37$ & $6.88\times10^{-4}$ & $2.83$\\
$2$ & $1.24\times10^{-2}$ & $1.97$ & $8.17\times10^{-5}$ & $2.75$ & $9.10\times10^{-5}$ & $2.92$\\
$3$ & $3.13\times10^{-3}$ & $1.99$ & $1.10\times10^{-5}$ & $2.89$ & $1.17\times10^{-5}$ & $2.96$\\
\midrule
\multicolumn{7}{c}{ETD-RK4} \\
\midrule
& $\mathcal{P}^1$-DG & & $\mathcal{P}^2$-DG & & $\mathcal{P}^3$-DG\\
\cline{2-3} \cline{4-5} \cline{6-7}  
$i$ & $L^2$ Error & Order & $L^2$ Error & Order& $L^2$ Error & Order\\
\midrule
$0$ & $1.85\times10^{-1}$ & -- & $1.78\times10^{-3}$ & -- & $3.38\times10^{-4}$ & --\\
$1$ & $4.91\times10^{-2}$ & $1.91$ & $1.17\times10^{-4}$ & $3.92$ & $2.40\times10^{-5}$ & $3.82$\\
$2$ & $1.25\times10^{-2}$ & $1.97$ & $8.89\times10^{-6}$ & $3.72$ & $1.58\times10^{-6}$ & $3.92$\\
$3$ & $3.14\times10^{-3}$ & $1.99$ & $8.46\times10^{-7}$ & $3.39$ & $1.09\times10^{-7}$ & $3.87$\\
\bottomrule[1.5pt]
\end{tabular}
\caption{\textbf{Example \ref{ex:accuracy}. Accuracy test: nonlinear problem.} 
$L^2$ errors and orders of convergence for the nonlinear diffusion-reaction equation \eqref{eq:diff-react}, solved using various ETD-RKDG methods. 
Unstructured grids with different levels of refinement are used, and the time-step size is set to $\tau = 0.2h$.}
\label{tab:accuracy_nonlinear}
\end{table}

\begin{exmp}\label{ex:Barenblatt}
\textbf{Barenblatt solution}
\end{exmp}
In this example, we test our method on the porous medium equation \eqref{eq:PME} using the Barenblatt solutions \eqref{eq:barenblatt}.
In the two-dimensional case, we simulate the solution from the initial time $t_0 = 1$ to $t = 2$ on the computational domain $\Omega = {(x, y): x^2 + y^2 < 64}$.
In the three-dimensional case, we simulate the solution from $t_0 = 1$ to $t = 3$ on the domain $\Omega = {(x, y, z): x^2 + y^2 + z^2 < 36,\ x, y, z \geq 0}$.
Homogeneous Neumann boundary conditions are imposed in both cases.
The computational meshes are shown in Figure~\ref{fig:Barenblatt2D_mesh}.
We compute the solution using ETD-RK3 and the $\mathcal{P}^2$-DG space.
Solutions and corresponding errors in two dimensions for $m = 2, 3, 5, 8$ are presented in Figure~\ref{fig:2DBarenblatt}, and the solution and corresponding error in three dimensions for $m = 3$ are shown in Figure~\ref{fig:3DBarenblatt}.

\begin{figure}[!htbp]
 \centering
 \begin{subfigure}[b]{0.45\textwidth}
  \includegraphics[width=\textwidth]{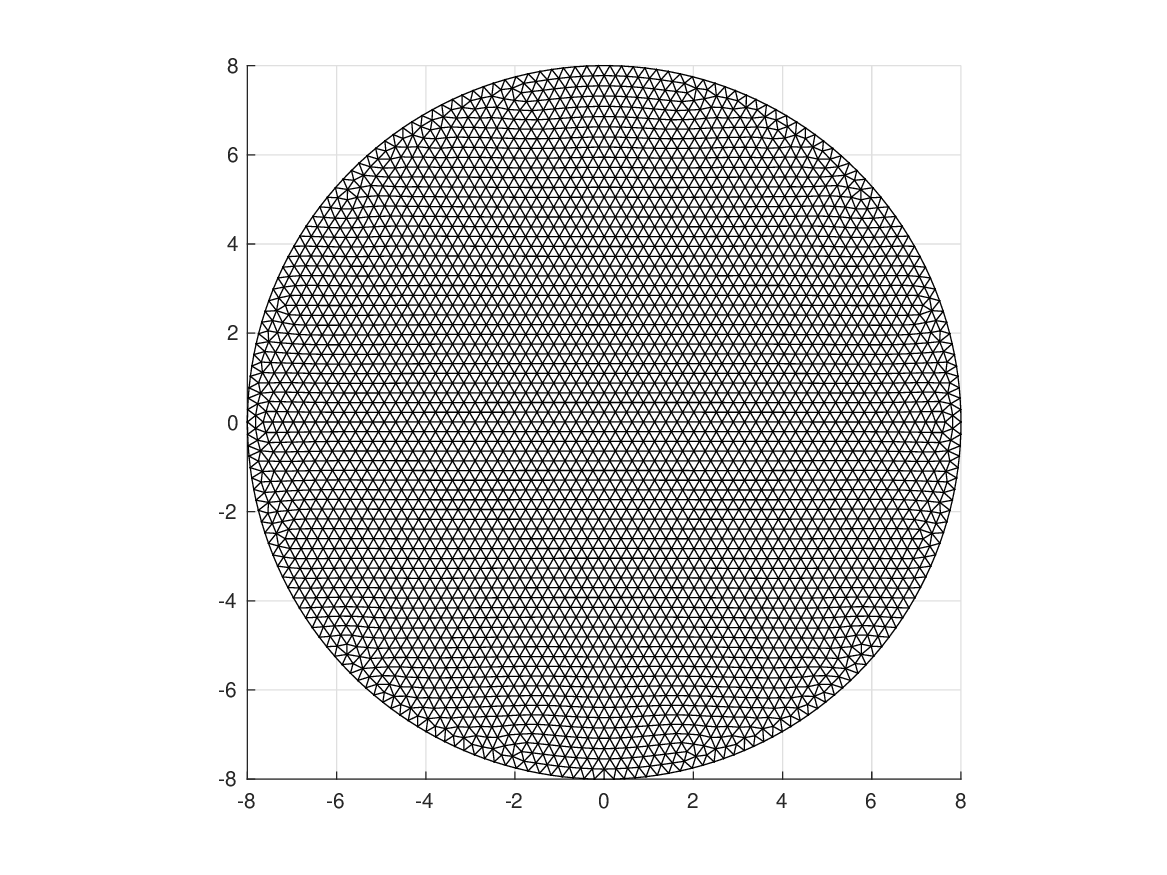}
  \caption{$2$D}
 \end{subfigure}
 \begin{subfigure}[b]{0.45\textwidth}
  \includegraphics[width=\textwidth]{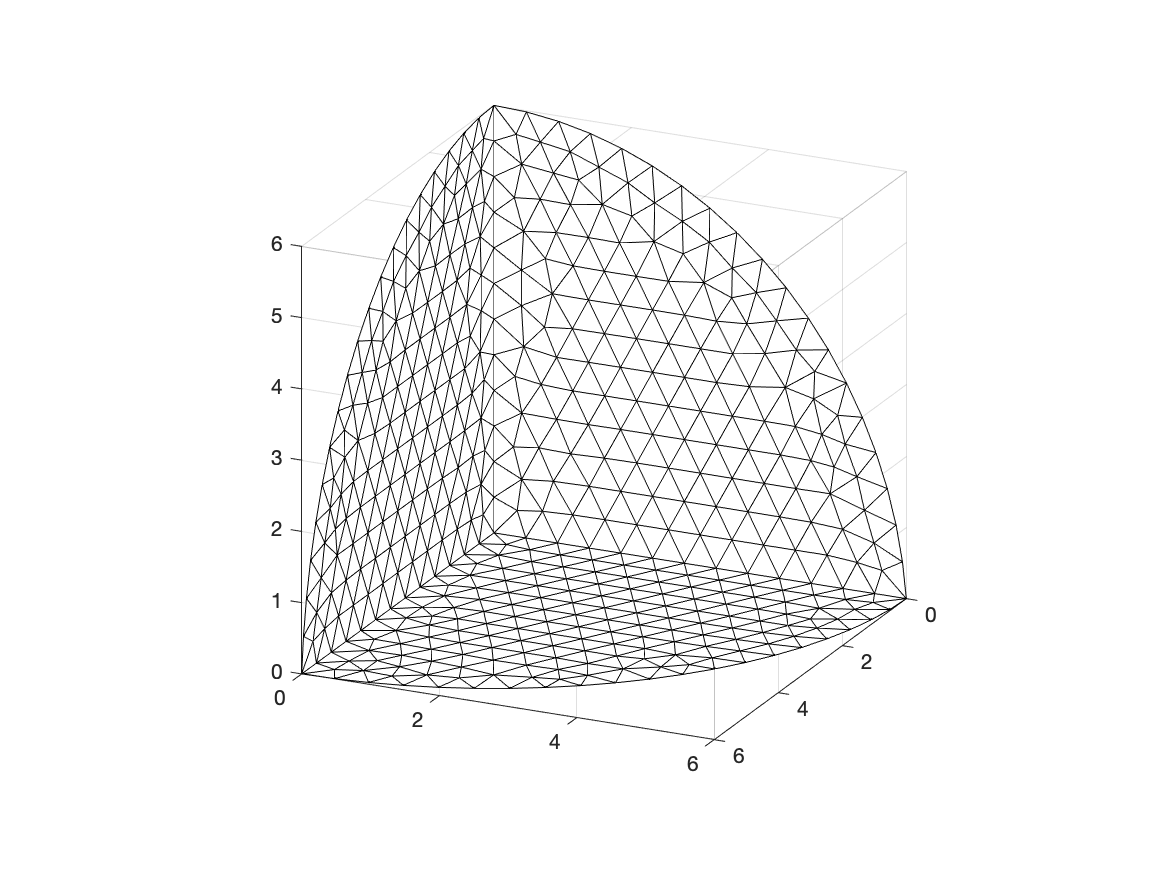}
  \caption{$3$D}
 \end{subfigure} 
 \caption{\textbf{Example \ref{ex:Barenblatt}. Barenblatt solution.} 
Computational meshes.}
 \label{fig:Barenblatt2D_mesh}
\end{figure}

\begin{figure}[!htbp]
 \centering
 \begin{subfigure}[b]{0.35\textwidth}
  \includegraphics[width=\textwidth]{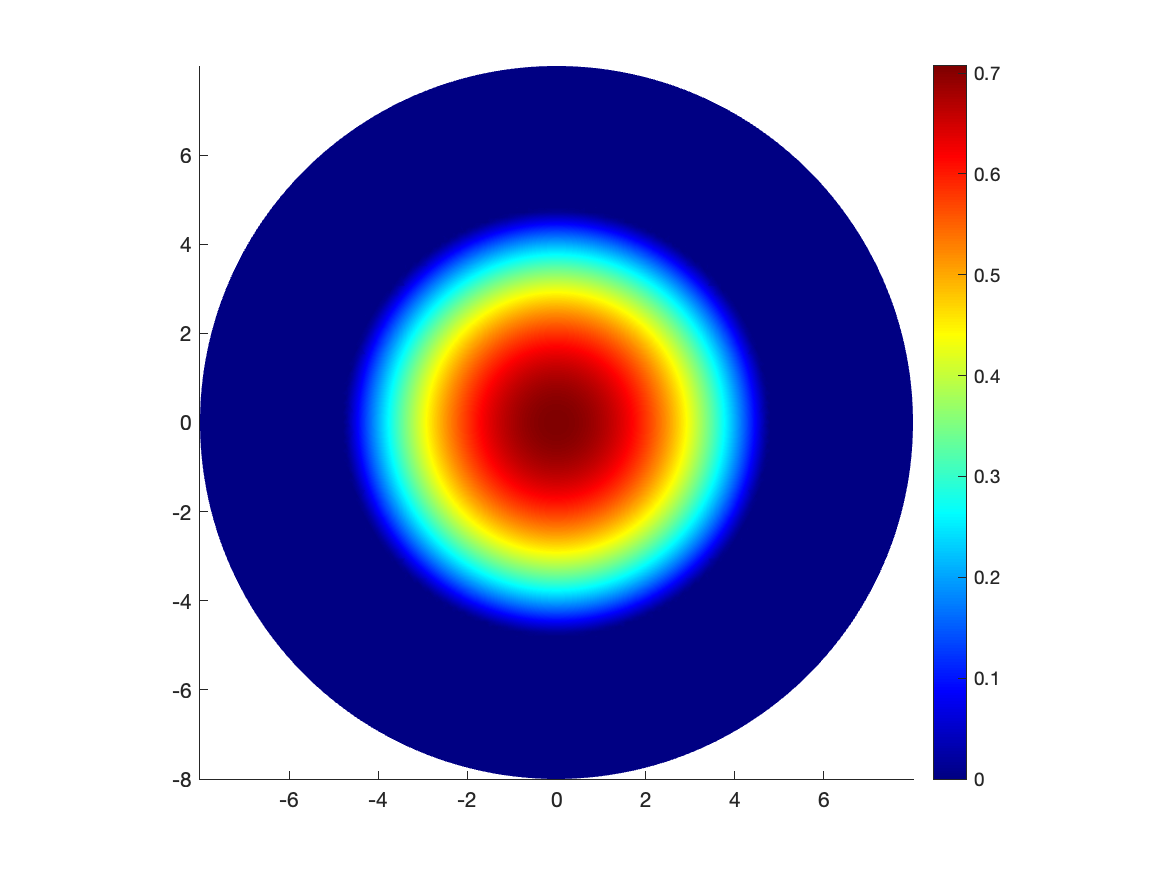}
  \caption{Solution, $m=2$}
 \end{subfigure}
 \begin{subfigure}[b]{0.35\textwidth}
  \includegraphics[width=\textwidth]{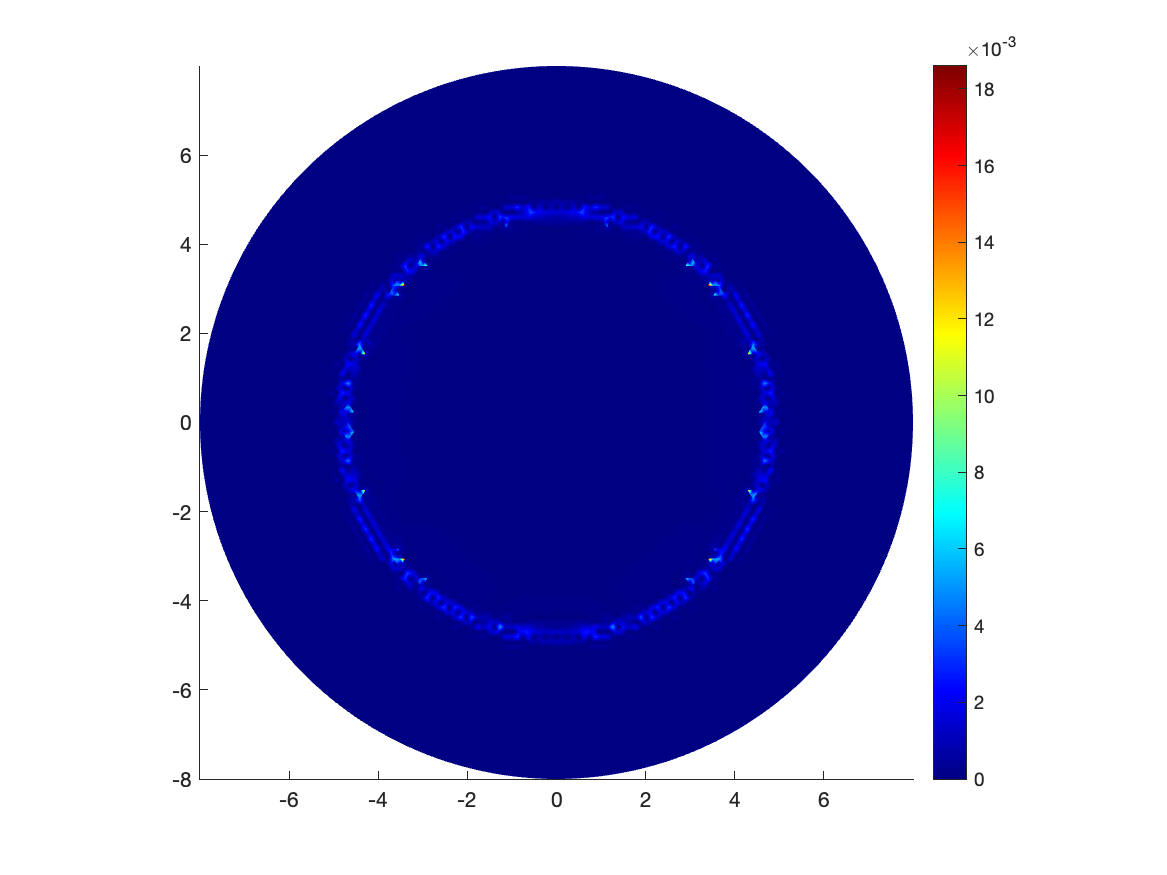}
  \caption{Error, $m=2$}
 \end{subfigure}
 \begin{subfigure}[b]{0.35\textwidth}
  \includegraphics[width=\textwidth]{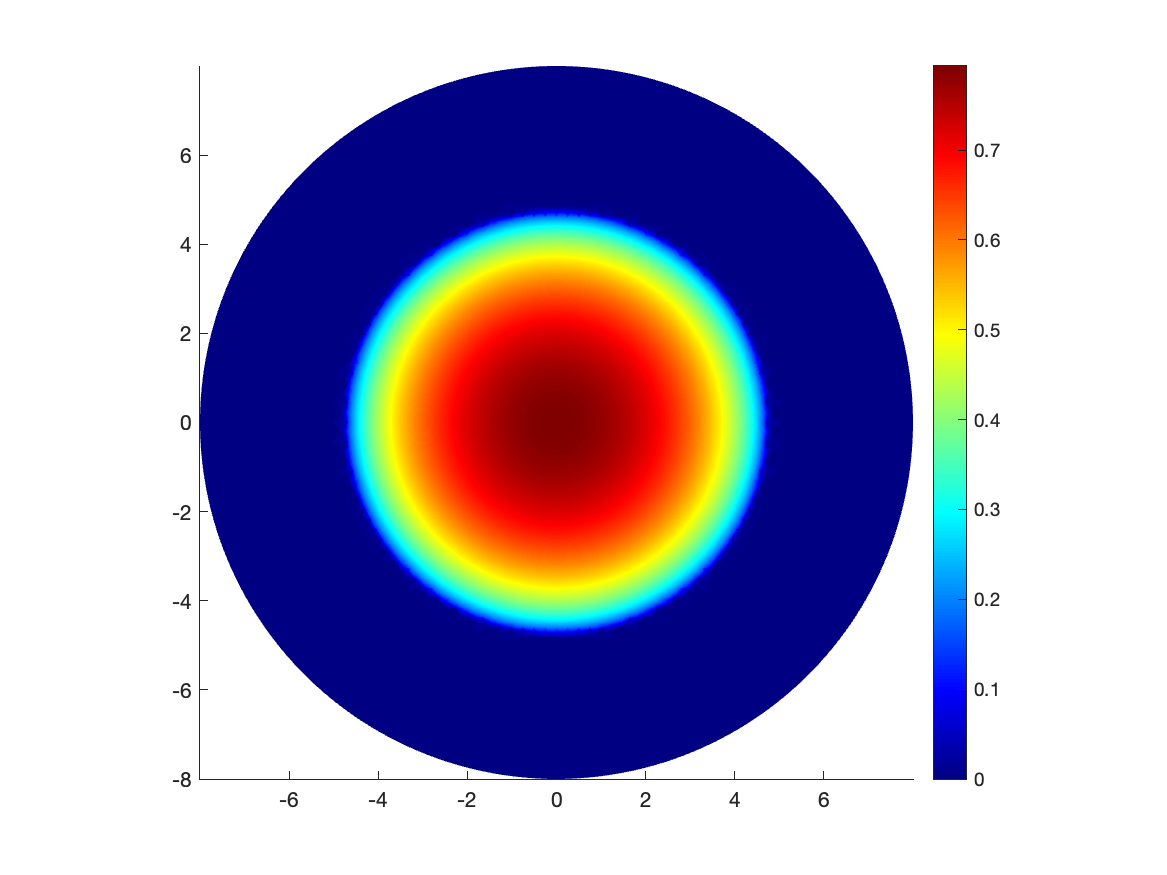}
  \caption{Solution, $m=3$}
 \end{subfigure}
 \begin{subfigure}[b]{0.35\textwidth}
  \includegraphics[width=\textwidth]{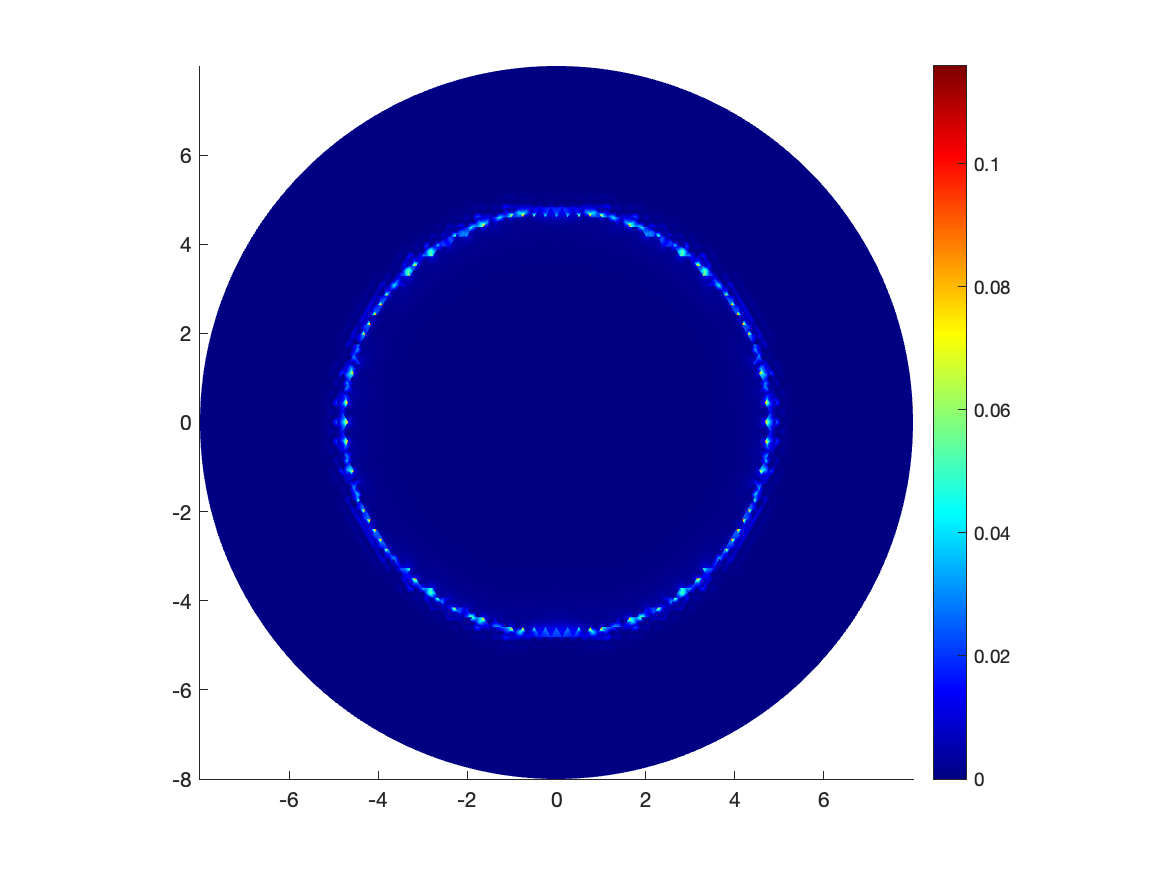}
  \caption{Error, $m=3$}
 \end{subfigure} 
 \begin{subfigure}[b]{0.35\textwidth}
  \includegraphics[width=\textwidth]{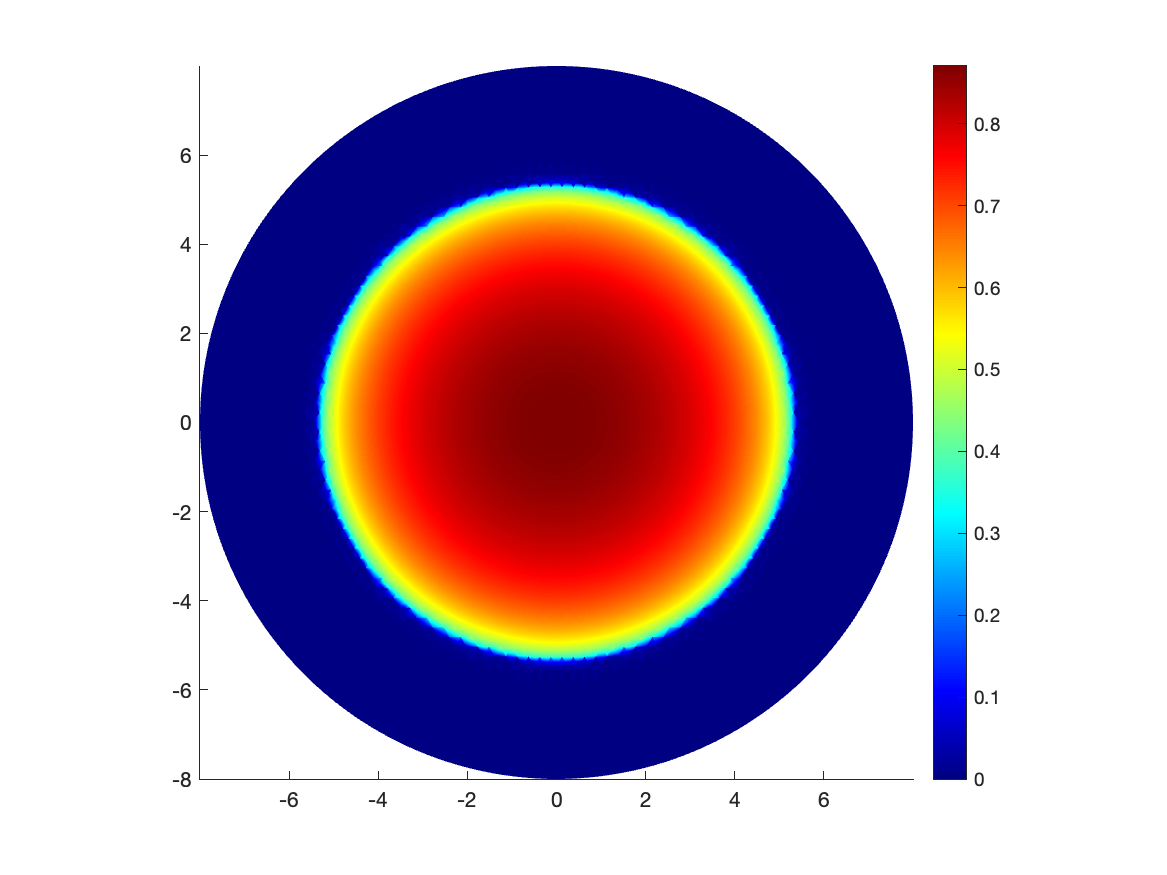}
  \caption{Solution, $m=5$}
 \end{subfigure}
 \begin{subfigure}[b]{0.35\textwidth}
  \includegraphics[width=\textwidth]{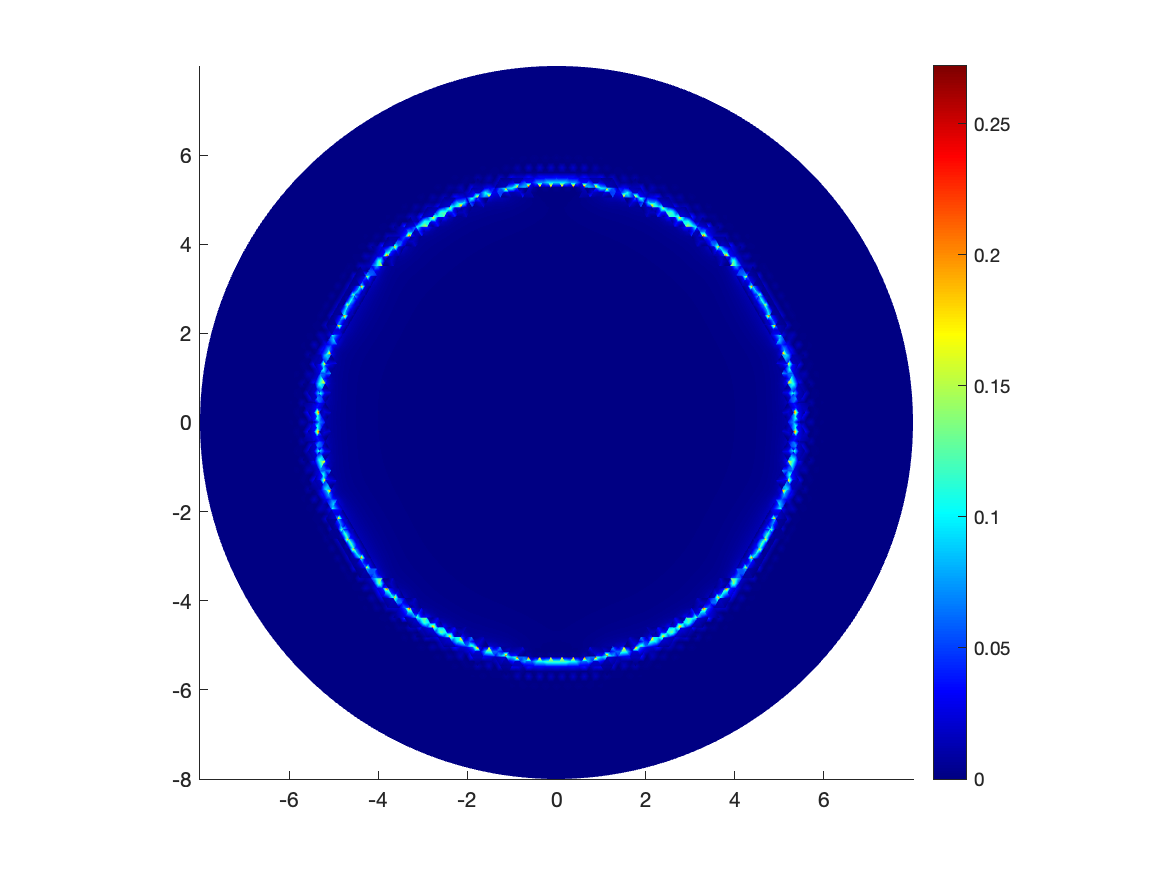}
  \caption{Error, $m=5$}
 \end{subfigure}
 \begin{subfigure}[b]{0.35\textwidth}
  \includegraphics[width=\textwidth]{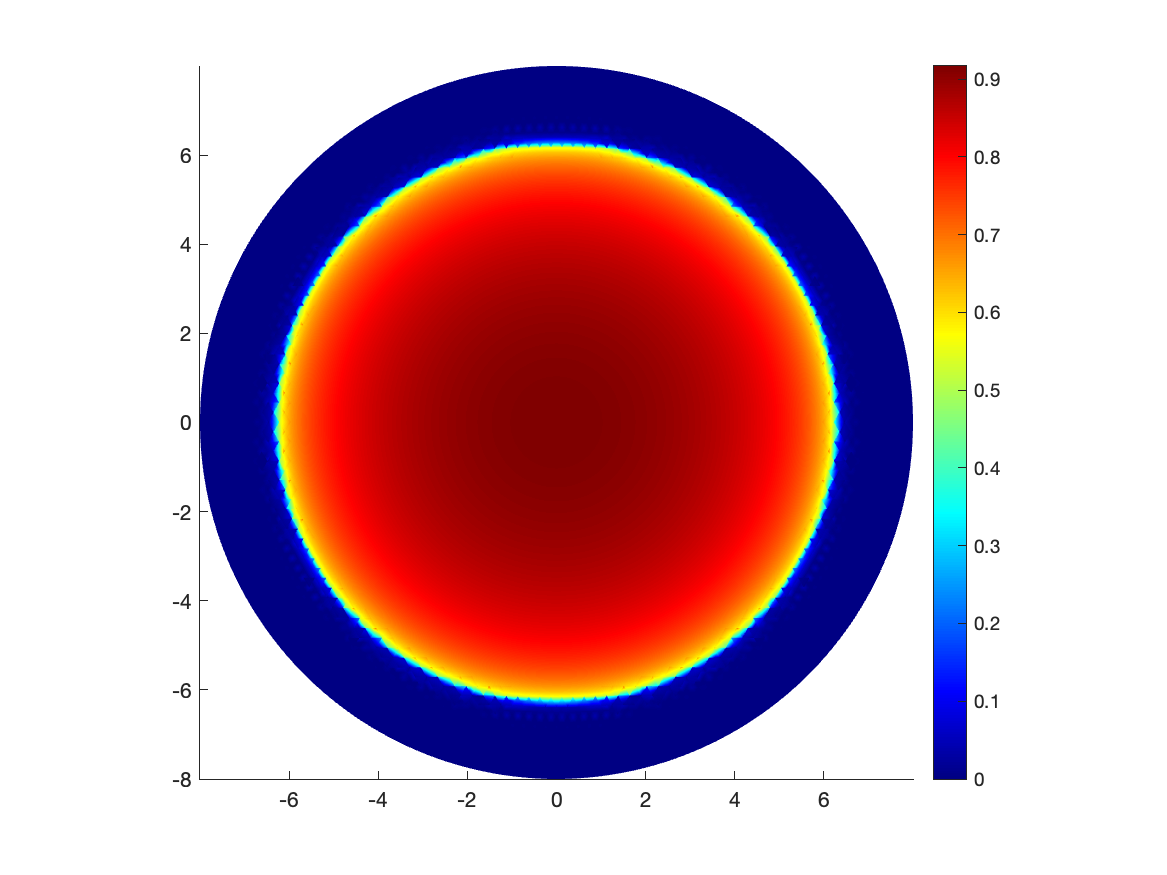}
  \caption{Solution, $m=8$}
 \end{subfigure}
 \begin{subfigure}[b]{0.35\textwidth}
  \includegraphics[width=\textwidth]{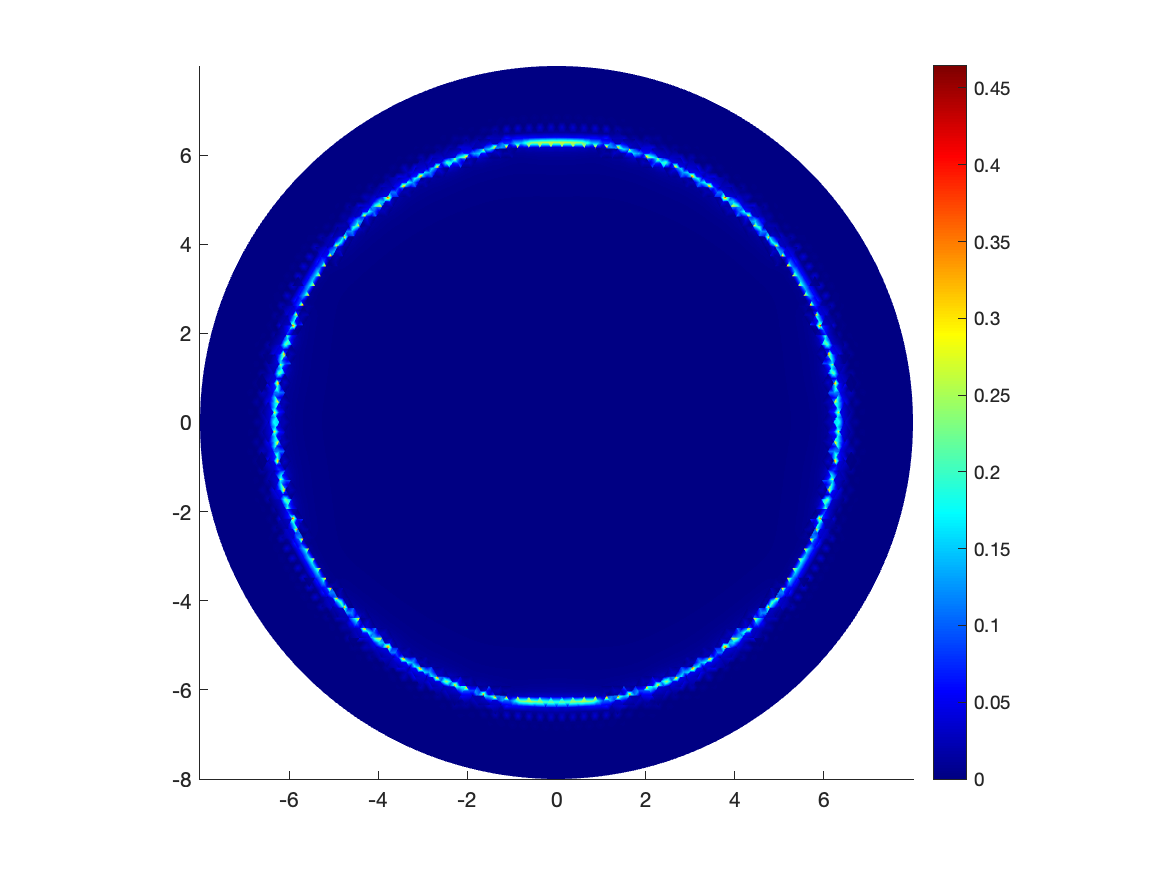}
  \caption{Error, $m=8$}
 \end{subfigure} 
 \caption{\textbf{Example \ref{ex:Barenblatt}: Barenblatt solution in 2D.} 
Numerical approximations and errors at $T=2$ to the Barenblatt solution of the two-dimensional PME.}
 \label{fig:2DBarenblatt}
\end{figure}

\begin{figure}[!htbp]
 \centering
 \begin{subfigure}[b]{0.45\textwidth}
  \includegraphics[width=\textwidth]{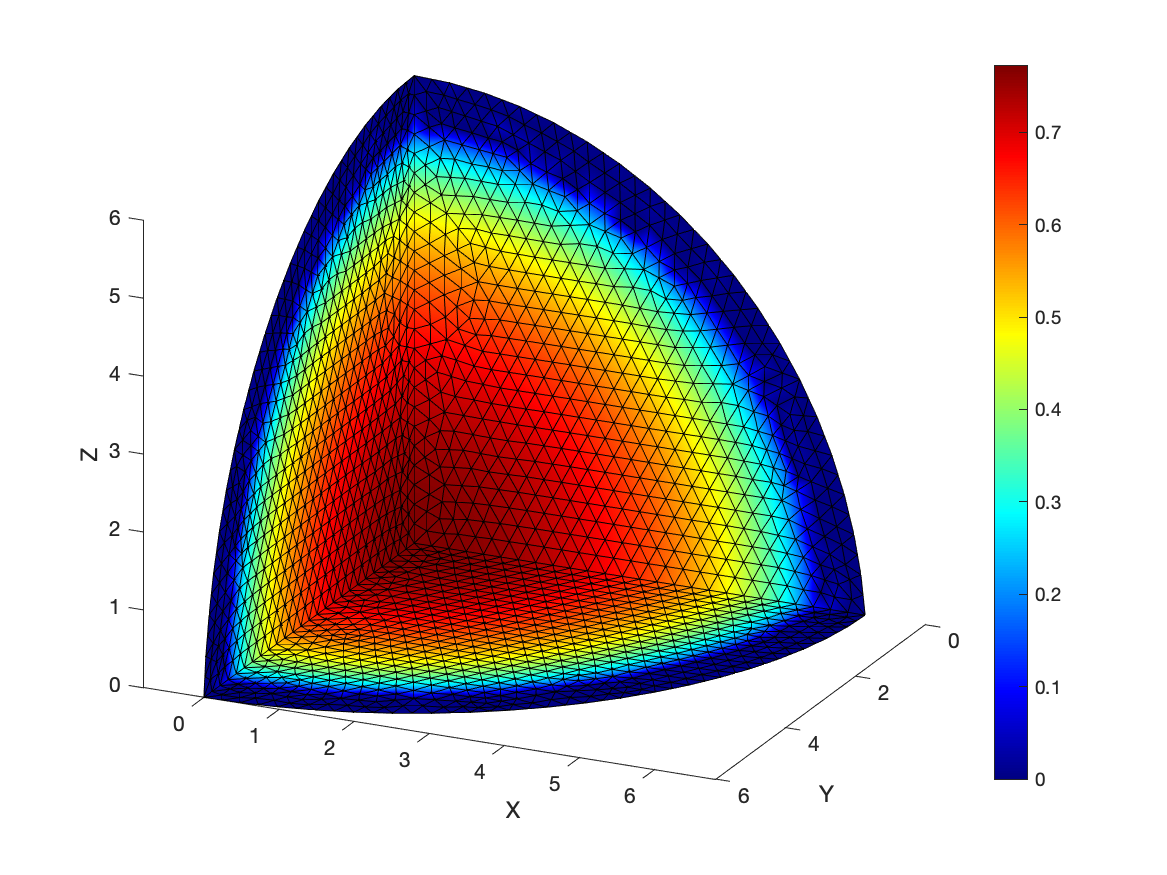}
  \caption{Solution}
 \end{subfigure}
 \begin{subfigure}[b]{0.45\textwidth}
  \includegraphics[width=\textwidth]{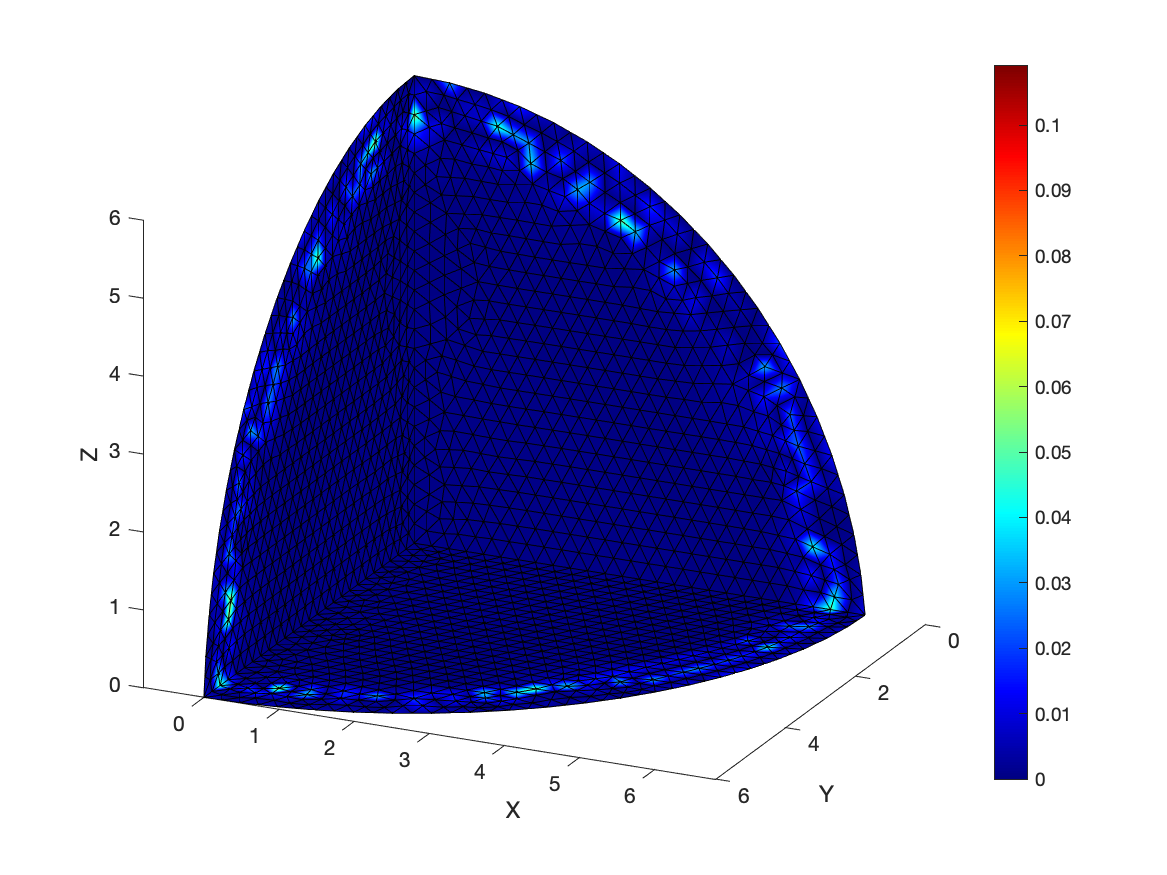}
  \caption{Error}
 \end{subfigure}
 \caption{\textbf{Example \ref{ex:Barenblatt}: Barenblatt solution in 3D.} 
Numerical approximation and error at $T=3$ to the Barenblatt solution of the three-dimensional PME with $m=3$.}
 \label{fig:3DBarenblatt}
\end{figure}







\begin{exmp}
\textbf{PME on a torus}\label{ex:Torus}
\end{exmp}
In this example, we solve the porous medium equation \eqref{eq:PME} on a three-dimensional, doughnut-shaped geometry.
The domain $\Omega$ is a torus centred at the origin with major radius \(R = 2\) and minor radius \(r = 0.5\).
A visualization of the domain and the computational mesh is shown in Figure~\ref{fig:3DTorusMesh}.

\begin{figure}[!htbp]
 \centering
  \includegraphics[width=0.5\textwidth]{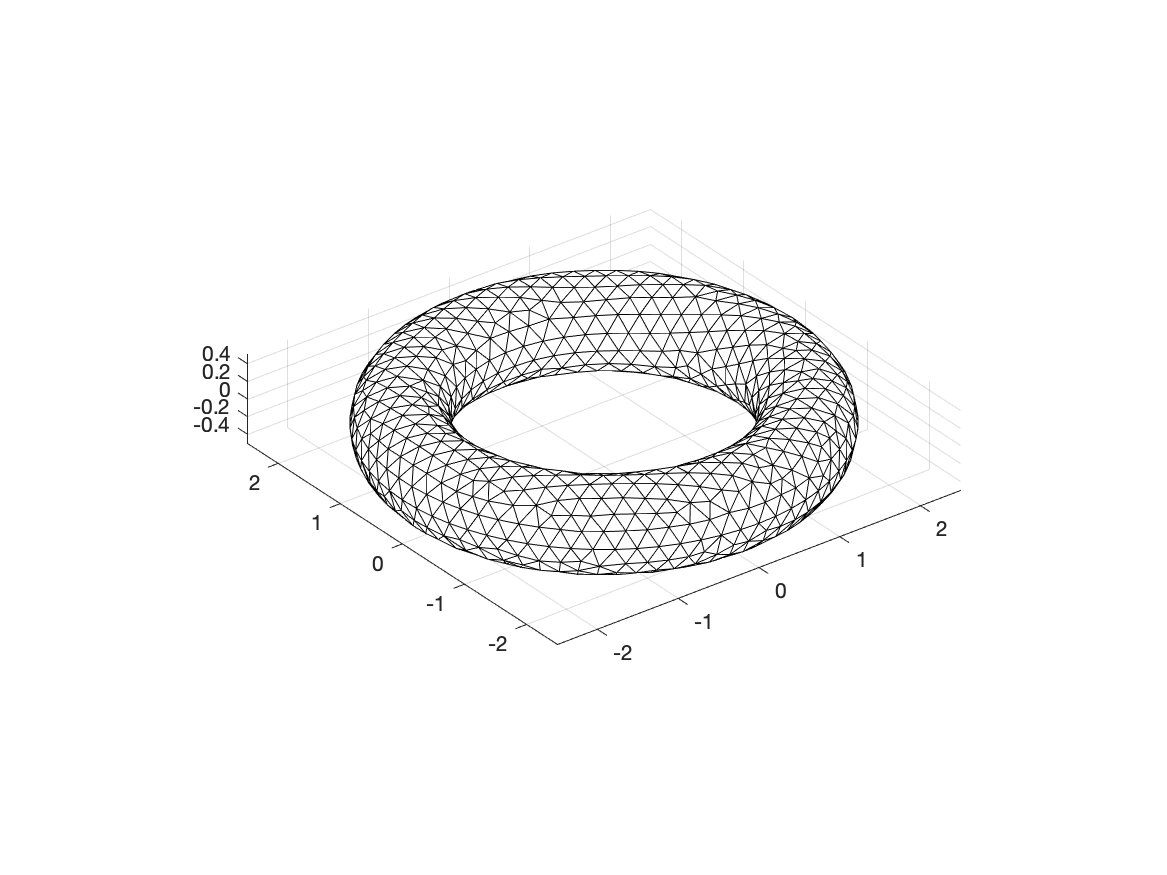}
 \caption{\textbf{Example \ref{ex:Torus}: PME on a torus.} 
Computational mesh.}
 \label{fig:3DTorusMesh}
\end{figure}

The initial condition on the torus is given by
\[
u_{0}(x,y,z)=\bigl(2.5 - r(x,y,z)\bigr)
\Bigl(
B\bigl(\theta(x,y);\theta_{1}\bigr)
+
B\bigl(\theta(x,y);\theta_{2}\bigr)
\Bigr),
\]
where \[
r(x,y,z) = \sqrt{x^{2}+y^{2}+z^{2}}, \qquad
\theta(x,y) = \operatorname{atan2}(y,x), \qquad
\Delta\theta = \frac{\pi}{5}, \quad
\theta_{1} = \frac{\pi}{4}, \quad
\theta_{2} = -\frac{\pi}{4},
\] and the bump function is defined as
\[
B(\theta;\theta_c) =
\begin{cases}
\displaystyle \exp\!\left(-\dfrac{0.2}{\Delta\theta^{2}-(\theta-\theta_c)^{2}}\right), & |\theta-\theta_c| < \Delta\theta, \\[10pt]
0, & \text{otherwise}.
\end{cases}
\]
A homogeneous boundary condition is imposed on the surface of the torus.

We compute the solution using the ETD-RK3 time integrator and $\mathcal{P}^2$-DG discretization.
Solutions at times \(T = 0,\;0.2,\;0.5,\;5\) are shown in Figure \ref{fig:3DTorus}.

\begin{figure}[!htbp]
 \centering
 \begin{subfigure}[b]{0.45\textwidth}
  \includegraphics[width=\textwidth]{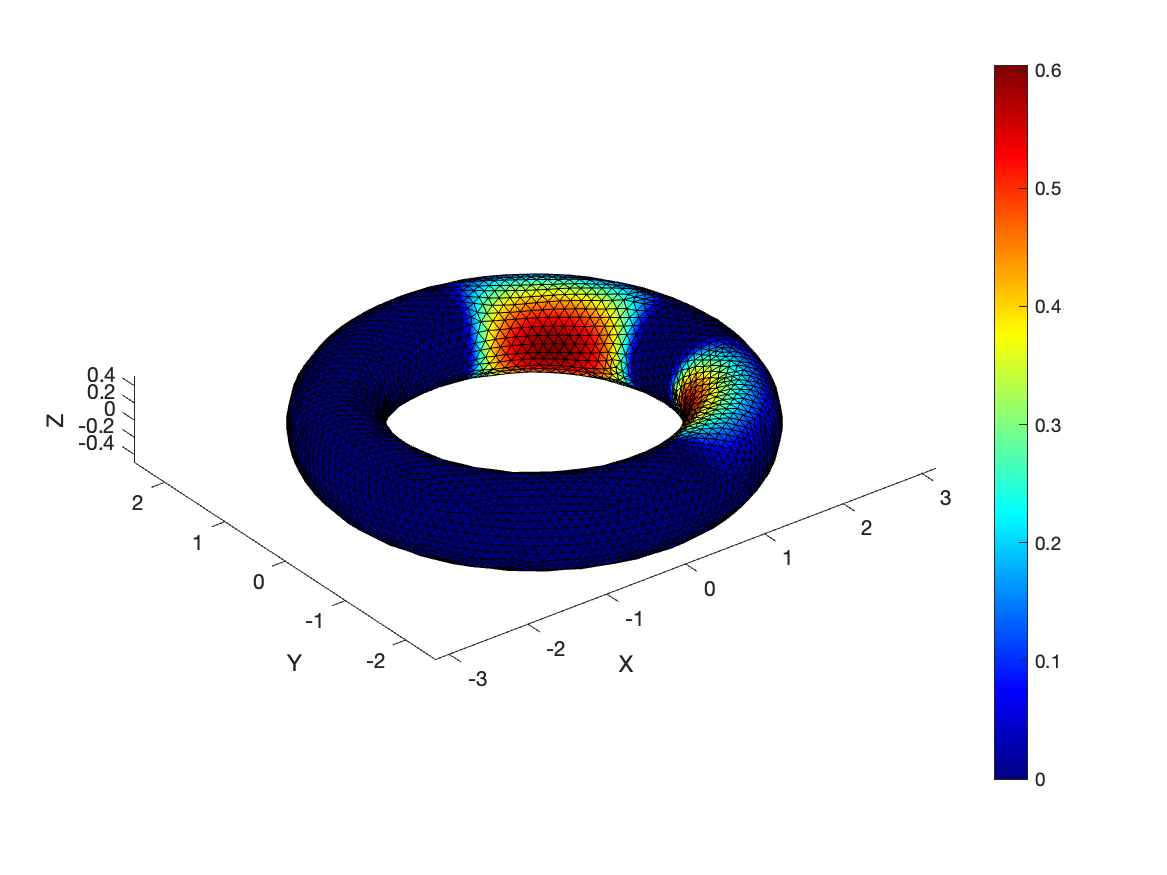}
  \caption{$T=0$}
 \end{subfigure}
 \begin{subfigure}[b]{0.45\textwidth}
  \includegraphics[width=\textwidth]{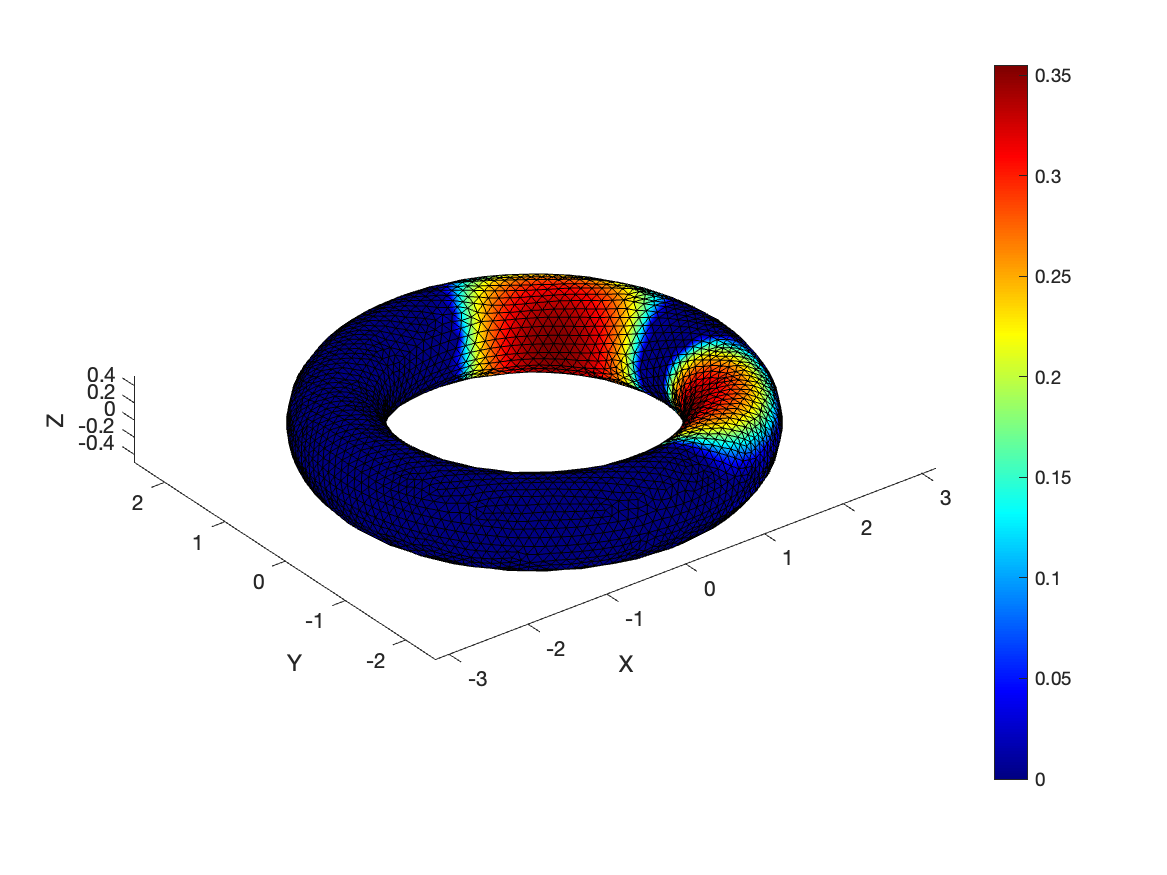}
  \caption{$T=0.2$}
 \end{subfigure}
 \begin{subfigure}[b]{0.45\textwidth}
  \includegraphics[width=\textwidth]{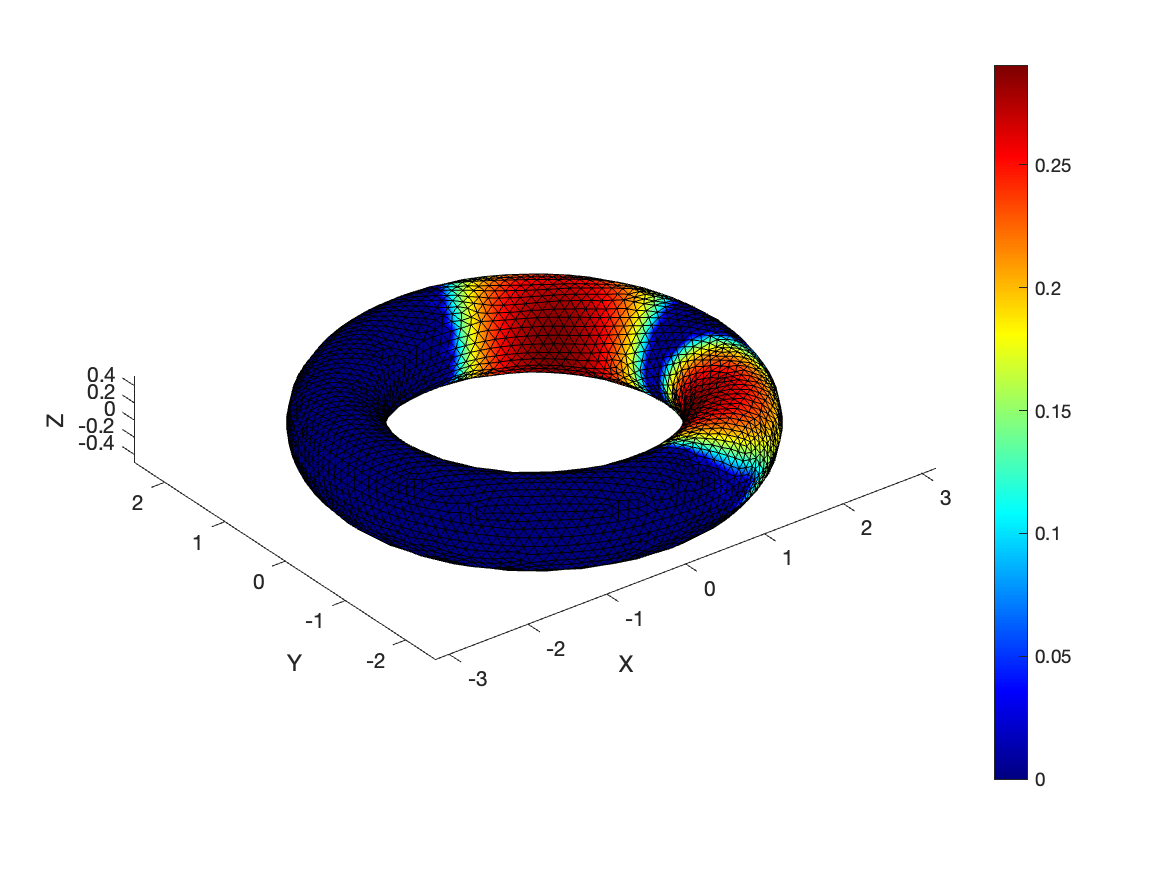}
  \caption{$T=0.5$}
 \end{subfigure}
 \begin{subfigure}[b]{0.45\textwidth}
  \includegraphics[width=\textwidth]{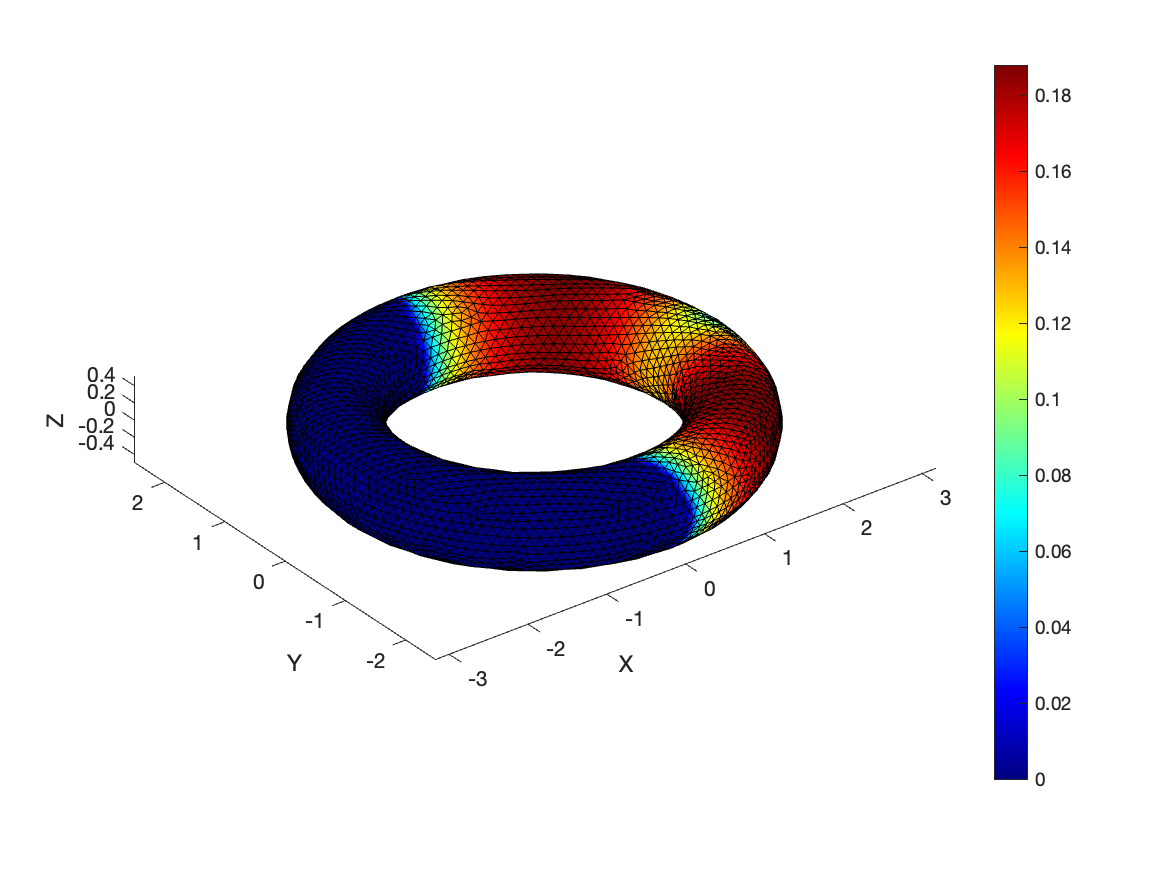}
  \caption{$T=5$}
 \end{subfigure}
 \caption{\textbf{Example \ref{ex:Torus}.} 
Numerical solutions of the PME on a torus at different times.}
 \label{fig:3DTorus}
\end{figure}

\section{Conclusions}\label{Sect:summary}
In this paper, we develop a class of efficient exponential time differencing Runge-Kutta discontinuous Galerkin methods for solving nonlinear degenerate parabolic equations.
The proposed DG methods are flexible for handling complex domain geometries, which improves our previous work on ETD-RK WENO methods \cite{xu2025high}, and they maintain both high-order accuracy and nonlinear stability of the simulations.
By applying a Rosenbrock-type treatment in the ETD-RK schemes, we extract the stiff linear component of the nonlinear diffusion term and absorb it exactly using an exponential integrating factor, which allows for large time-step sizes and greatly enhances stability of the computations. The stiffness of the nonlinear degenerate parabolic PDEs is resolved well, and a high efficiency of the simulations is achieved.
To facilitate the computation of the Jacobian matrix in the  Rosenbrock-type treatment and its implementation, we adopt the nodal formulation, which is highly vectorized and well suited to the proposed ETD-RKDG methods.
A heuristic justification for the improved stability is presented through a linear stability analysis of the one-dimensional case using the Fourier method.
Finally, numerical validation demonstrates the excellent performance of the proposed methods.

\section*{Acknowledgments}
The authors are grateful to Dr. Zheng Sun for bringing some results of IMEX methods to our attention and for inspiring discussions.

\section*{Conflict of Interest}
The authors declare that they have no known competing financial interests or personal
relationships that could have appeared to influence the work reported in this paper.


\begin{appendices}
\setcounter{figure}{0} 
\setcounter{equation}{0} 
\renewcommand{\theequation}{\thesection.\arabic{equation}} 
\renewcommand\thefigure{\Alph{section}\arabic{figure}}

\section{Exact coordinates of the Lagrange nodes on the reference element $\widehat{K}$}\label{appd:LagrangeNodes}

Since the Lagrange nodes of \( \mathcal{P}^k(\widehat{K}) \) coincide with the \( (k+1) \) Gauss–Lobatto points on each edge, we list only the coordinates of the interior nodes in Table~\ref{tab:barycoor}.
\begin{table}[!htbp]
\centering
\begin{tabular}{c|c}
\hline
$\mathcal{P}^k(K)$ & $(\widehat{x}_1,\widehat{x}_2)$ \\ 
\hline
\multirow{1}{*}{$k=1$}    & N/A \\ 
\hline
\multirow{1}{*}{$k=2$}    & N/A \\ 
\hline
\multirow{1}{*}{$k=3$}    & $(0.333333333333333, 0.333333333333333)$ \\ 
\hline
\multirow{3}{*}{$k=4$} & $(0.551583507555306, 0.224208246222347)$ \\  
                                  & $(0.224208246222347, 0.551583507555306)$ \\ 
                                  & $(0.224208246222347, 0.224208246222347)$ \\ 
\hline
\multirow{6}{*}{$k=5$}    & (0.684472514501909, 0.157763742749046) \\ 
                          & (0.414377261333963, 0.414377261333963) \\ 
                          & (0.157763742749046, 0.684472514501908) \\ 
                          & (0.414377261333963, 0.171245477332075) \\ 
                          & (0.171245477332074, 0.414377261333963) \\ 
                          & (0.157763742749046, 0.157763742749046) \\ 
\hline
\end{tabular}
\caption{Coordinates of the interior Lagrange nodes on the reference element \(\widehat{K}\)}\label{tab:barycoor}
\end{table}

\section{Computation of local matrices in \eqref{eq:LocalMatrices1} and \eqref{eq:LocalMatrices2}}\label{appd:LocalMatrices}

To facilitate the computation of the Lagrange basis functions in \eqref{eq:FEMbasis}, we introduce the Vandermonde matrix $\widehat{\mathcal{V}}$ and its gradient matrices $\widehat{\mathcal{V}}^{(1,0)}$ and $\widehat{\mathcal{V}}^{(0,1)}$ in $\mathbb{R}^{N_k \times N_k}$ on the reference element $\widehat{K}$, defined by
\begin{equation}
\widehat{\mathcal{V}}_{ij} = \widehat{\phi}_j(\widehat{\mathbf{x}}_i),\quad \widehat{\mathcal{V}}_{ij}^{(1,0)} = \frac{\partial \widehat{\phi}_{j}}{\partial \widehat{x}_1}(\widehat{\mathbf{x}}_i),\quad \widehat{\mathcal{V}}_{ij}^{(0,1)} = \frac{\partial \widehat{\phi}_{j}}{\partial \widehat{x}_2}(\widehat{\mathbf{x}}_i),\quad i,j=1,\ldots, N_{k}.
\end{equation}
The nodal basis functions $\widehat{\ell}_i(\widehat{\mathbf{x}})$ and the modal basis functions $\widehat{\phi}_i(\widehat{\mathbf{x}})$ in $\mathcal{P}^{k}(\widehat{K})$ are connected through the Vandermonde matrix as follows:
\begin{equation}
[\widehat{\ell}_1(\widehat{\mathbf{x}}),\ldots, \widehat{\ell}_{N_k}(\widehat{\mathbf{x}})]=[\widehat{\phi}_1(\widehat{\mathbf{x}}),\ldots, \widehat{\phi}_{N_k}(\widehat{\mathbf{x}})] \widehat{\mathcal{V}}^{-1}.
\end{equation}
The local matrices $\mathcal{M}_{\widehat{K}}$, 
$\mathcal{S}^{(1,0)}_{\widehat{K}}$, 
$\mathcal{S}^{(0,1)}_{\widehat{K}}$, 
$\mathcal{S}^{(2,0)}_{\widehat{K}}$, $\mathcal{S}^{(0,2)}_{\widehat{K}}$, and 
$\mathcal{S}^{(1,1)}_{\widehat{K}}$ in \eqref{eq:LocalMatrices1}, defined on the reference element $\widehat{K}$, can then be calculated as
\begin{equation}
\begin{split}
&\mathcal{M}_{\widehat{K}}=\widehat{\mathcal{V}}^{-T}\widehat{\mathcal{V}}^{-1},\quad \mathcal{S}^{(1,0)}_{\widehat{K}}={\mathcal{M}_{\widehat{K}}} {\mathcal{D}}^{(1,0)}_{\widehat{K}},\quad {\mathcal{S}}^{(0,1)}_{\widehat{K}}={\mathcal{M}_{\widehat{K}}} {\mathcal{D}}^{(0,1)}_{\widehat{K}},\\
&{\mathcal{S}}^{(2,0)}_{\widehat{K}}={\mathcal{M}_{\widehat{K}}} ({\mathcal{D}}^{(1,0)}_{\widehat{K}})^2,\quad {\mathcal{S}}^{(0,2)}_{\widehat{K}}={\mathcal{M}_{\widehat{K}}} ({\mathcal{D}}^{(0,1)}_{\widehat{K}})^2,\quad {\mathcal{S}}^{(1,1)}_{\widehat{K}}={\mathcal{M}_{\widehat{K}}}{\mathcal{D}}^{(1,0)}_{\widehat{K}}{\mathcal{D}}^{(0,1)}_{\widehat{K}}.
\end{split}
\end{equation}
Here, the differentiation matrices ${\mathcal{D}}^{(1,0)}_{\widehat{K}}$ and ${\mathcal{D}}^{(0,1)}_{\widehat{K}}$ are defined as
\begin{equation}
{\mathcal{D}}^{(1,0)}_{\widehat{K}}=\widehat{\mathcal{V}}^{(1,0)}\widehat{\mathcal{V}}^{-1},\quad {\mathcal{D}}^{(0,1)}_{\widehat{K}}=\widehat{\mathcal{V}}^{(0,1)}\widehat{\mathcal{V}}^{-1}.
\end{equation}
We denote by $\mathcal{J}^K \in \mathbb{R}^{2 \times 2}$ the Jacobian matrix of the affine mapping $\mathbb{T}^K : K \rightarrow \widehat{K}$, i.e., $\mathcal{J}^K = \frac{\partial \widehat{\mathbf{x}}}{\partial \mathbf{x}}$, where $\widehat{\mathbf{x}} = \mathbb{T}^K(\mathbf{x})$ for $\mathbf{x} \in K$.
Based on this transformation, the local matrices in \eqref{eq:LocalMatrices1} on the element $K \in \mathcal{T}$ can be computed as
\begin{equation}
\begin{split}
&\mathcal{M}_K=\frac{1}{\det{\mathcal{J}^K}}{\mathcal{M}_{\widehat{K}}},\\
&\mathcal{S}^{(1,0)}_K=\frac{1}{\det{\mathcal{J}^K}}(\mathcal{J}_{11}^K{\mathcal{S}}^{(1,0)}_{\widehat{K}}+\mathcal{J}_{21}^K{\mathcal{S}}^{(0,1)}_{\widehat{K}}),\\
&\mathcal{S}^{(0,1)}_K=\frac{1}{\det{\mathcal{J}^K}}(\mathcal{J}^K_{12}{\mathcal{S}}^{(1,0)}_{\widehat{K}}+\mathcal{J}^K_{22}{\mathcal{S}}^{(0,1)}_{\widehat{K}}),\\
&\mathcal{S}_K^{(2,0)}=\frac{1}{\det\mathcal{J}^K}((\mathcal{J}_{11}^K)^2{\mathcal{S}}^{(2,0)}_{\widehat{K}}+2\mathcal{J}^K_{11}\mathcal{J}^K_{21}{\mathcal{S}}^{(1,1)}_{\widehat{K}}+(\mathcal{J}_{21}^K)^2{\mathcal{S}}^{(0,2)}_{\widehat{K}}),\\
&\mathcal{S}_K^{(0,2)}=\frac{1}{\det\mathcal{J}^K}((\mathcal{J}_{12}^K)^2{\mathcal{S}}^{(2,0)}_{\widehat{K}}+2\mathcal{J}_{12}^K\mathcal{J}_{22}^K{\mathcal{S}}^{(1,1)}_{\widehat{K}}+(\mathcal{J}_{22}^K)^2{\mathcal{S}}^{(0,2)}_{\widehat{K}}),\\
\end{split}
\end{equation}

Some local matrices in \eqref{eq:LocalMatrices2} involve two neighboring elements.
To facilitate their computation on the reference element, we first define the relevant geometric quantities on $\widehat{K}$.
The vertices and their opposite edges are denoted by
\[
\widehat{\mathbf{v}}_1 = (1, 0)^T, \quad \widehat{\mathbf{v}}_2 = (0, 1)^T, \quad \widehat{\mathbf{v}}_3 = (0, 0)^T,
\]
and
\[
\begin{aligned}
\widehat{e}_1 &= \{(1 - s)\widehat{\mathbf{v}}_2 + s\widehat{\mathbf{v}}_3 : s \in [0, 1]\}, \\
\widehat{e}_2 &= \{(1 - s)\widehat{\mathbf{v}}_3 + s\widehat{\mathbf{v}}_1 : s \in [0, 1]\}, \\
\widehat{e}_3 &= \{(1 - s)\widehat{\mathbf{v}}_1 + s\widehat{\mathbf{v}}_2 : s \in [0, 1]\},
\end{aligned}
\]
respectively.
We then define the boundary matrices $\widehat{\mathcal{B}}_{m}$, $\widehat{\mathcal{B}}_{m}^{(1,0)}$, $\widehat{\mathcal{B}}_{m}^{(0,1)}$, $\widehat{\mathcal{B}}_{m,n}$, $\widehat{\mathcal{B}}_{m,n}^{(1,0)}$, and $\widehat{\mathcal{B}}_{m,n}^{(0,1)}$, for $m, n = 1, 2, 3$, on $\partial \widehat{K}$ as
\begin{equation}
\begin{split}
&(\widehat{\mathcal{B}}_{m})_{ij}=\int_{0}^{1}\widehat{\ell}_{i}((1-s)\widehat{\mathbf{v}}_{m+1}+s\widehat{\mathbf{v}}_{m+2})\widehat{\ell}_j((1-s)\widehat{\mathbf{v}}_{m+1}+s\widehat{\mathbf{v}}_{m+2}) ds,\\
&(\widehat{\mathcal{B}}^{(1,0)}_{m})_{ij}=\int_{0}^{1}\widehat{\ell}_{i}((1-s)\widehat{\mathbf{v}}_{m+1}+s\widehat{\mathbf{v}}_{m+2})\frac{\partial \widehat{\ell}_j}{\partial\widehat{x}_1 } ((1-s)\widehat{\mathbf{v}}_{m+1}+s\widehat{\mathbf{v}}_{m+2}) ds,\\
&(\widehat{\mathcal{B}}^{(0,1)}_{m})_{ij}=\int_{0}^{1}\widehat{\ell}_{i}((1-s)\widehat{\mathbf{v}}_{m+1}+s\widehat{\mathbf{v}}_{m+2})\frac{\partial \widehat{\ell}_j}{\partial\widehat{x}_{2} }((1-s)\widehat{\mathbf{v}}_{m+1}+s\widehat{\mathbf{v}}_{m+2})  ds,\\
&(\widehat{\mathcal{B}}_{m,n})_{ij}=\int_{0}^{1}\widehat{\ell}_{i}((1-s)\widehat{\mathbf{v}}_{m+1}+s\widehat{\mathbf{v}}_{m+2})\widehat{\ell}_j((1-s)\widehat{\mathbf{v}}_{n+2}+s\widehat{\mathbf{v}}_{n+1}) ds,\\
&(\widehat{\mathcal{B}}^{(1,0)}_{m,n})_{ij}=\int_{0}^{1}\widehat{\ell}_{i}((1-s)\widehat{\mathbf{v}}_{m+1}+s\widehat{\mathbf{v}}_{m+2})\frac{\partial \widehat{\ell}_j}{\partial\widehat{x}_1 }((1-s)\widehat{\mathbf{v}}_{n+2}+s\widehat{\mathbf{v}}_{n+1})  ds,\\
&(\widehat{\mathcal{B}}^{(0,1)}_{m,n})_{ij}=\int_{0}^{1}\widehat{\ell}_{i}((1-s)\widehat{\mathbf{v}}_{m+1}+s\widehat{\mathbf{v}}_{m+2})\frac{\partial \widehat{\ell}_j}{\partial \widehat{x}_2}((1-s)\widehat{\mathbf{v}}_{n+2}+s\widehat{\mathbf{v}}_{n+1}) ds.
\end{split}
\end{equation}
Finally, we denote the LGL nodes on the reference interval $\widehat{I}=[0, 1]$ and the corresponding quadrature weights by $\{\widehat{r}_{i}\}_{i=1}^{k+1}$ and $\{\widehat{\omega}_{i}\}_{i=1}^{k+1}$, respectively, where $\sum_{i=1}^{k+1}\widehat{\omega}_{i}=1$.

Since we adopt Lagrange nodes that coincide with the $(k+1)$-point LGL nodes on each edge, the Lagrange basis function $\widehat{\ell}_i$ is zero on an edge if $\widehat{\mathbf{x}}_i$ does not reside on that edge. 
This results in the boundary matrices $\widehat{\mathcal{B}}_{m}$, 
$\widehat{\mathcal{B}}_{m,n}$, 
$\widehat{\mathcal{B}}_{m}^{(1,0)}$, 
$\widehat{\mathcal{B}}_{m}^{(0,1)}$,
$\widehat{\mathcal{B}}_{m,n}^{(1,0)}$, and 
$\widehat{\mathcal{B}}_{m,n}^{(0,1)}$ exhibiting a sparse structure.
More specifically, the non-zero entries in $\widehat{\mathcal{B}}_{m}$ and 
$\widehat{\mathcal{B}}_{m,n}$ consist of only $(k+1)^2$ terms out of $N_k^2$ and are identical to those in the one-dimensional mass matrix  ${\mathcal{M}}_{\widehat{I}}$, whose entries are defined as 
\begin{equation}
({\mathcal{M}}_{\widehat{I}})_{i,j}=\int_{0}^{1}\widehat{\ell_{i}}^{\widehat{I}}(\widehat{r})\widehat{\ell}_{j}^{\widehat{I}}(\widehat{r})d\widehat{r}, \quad i,j=1,\ldots, k+1,
\end{equation}
and other matrices have entries given by
\begin{equation}
\begin{split}
(\widehat{\mathcal{B}}_{m}^{(1,0)})_{ij}=
\begin{cases}
\widehat{\omega}_\star\frac{\partial \widehat{\ell}_j}{\partial\widehat{x}_1 }(\widehat{\mathbf{x}}_i), &\widehat{\mathbf{x}}_i\in \widehat{e}_m,\\
0, &\text{otherwise},
\end{cases}
\quad
(\widehat{\mathcal{B}}_{m}^{(0,1)})_{ij}=
\begin{cases}
\widehat{\omega}_\star\frac{\partial \widehat{\ell}_j}{\partial\widehat{x}_2 }(\widehat{\mathbf{x}}_i), &\widehat{\mathbf{x}}_i\in \widehat{e}_m,\\
0, &\text{otherwise},
\end{cases}\\
(\widehat{\mathcal{B}}_{m,n}^{(1,0)})_{ij}=
\begin{cases}
\widehat{\omega}_\star\frac{\partial \widehat{\ell}_j}{\partial\widehat{x}_1 }(\widehat{\mathbf{x}}_\star), &\widehat{\mathbf{x}}_i\in \widehat{e}_m,\\
0, &\text{otherwise},
\end{cases}
\quad
(\widehat{\mathcal{B}}_{m,n}^{(0,1)})_{ij}=
\begin{cases}
\widehat{\omega}_\star\frac{\partial \widehat{\ell}_j}{\partial\widehat{x}_2 }(\widehat{\mathbf{x}}_\star), &\widehat{\mathbf{x}}_i\in \widehat{e}_m,\\
0, &\text{otherwise},
\end{cases}
\end{split}
\end{equation}
where the subscript $\star$ denotes the index satisfying
\begin{align*}
\widehat{\mathbf{x}}_i=(1-\widehat{r}_{\star})\widehat{\mathbf{v}}_{m+1}+\widehat{r}_{\star}\widehat{\mathbf{v}}_{m+2},\quad \text{if}~~\widehat{\mathbf{x}}_i\in\widehat{e}_m,
\end{align*}
and
\begin{equation*}
\widehat{\mathbf{x}}_{\star}=(1-\widehat{r}_{\star})\widehat{\mathbf{v}}_{n+2}+\widehat{r}_{\star}\widehat{\mathbf{v}}_{n+1}.
\end{equation*}

Based on the above setup, the local matrices in \eqref{eq:LocalMatrices2} on the element $K\in\mathcal{T}$ can be computed as
\begin{equation}
\begin{split}
&\mathcal{B}_m^K=|e_m^K|\widehat{\mathcal{B}}_m,\\
&\mathcal{B}^{K(1,0)}_{m}=|e_m^K|(\mathcal{J}_{11}^K\widehat{\mathcal{B}}^{(1,0)}_{m}+\mathcal{J}_{21}^K\widehat{\mathcal{B}}^{(0,1)}_{m}),\\
&\mathcal{B}^{K(0,1)}_{m}=|e_m^K|(\mathcal{J}_{12}^K\widehat{\mathcal{B}}^{(1,0)}_{m}+\mathcal{J}_{22}^K\widehat{\mathcal{B}}^{(0,1)}_{m}),\\
&\mathcal{B}^{K, K'}_{m,n}=|e_m^K|\widehat{\mathcal{B}}_{m,n},\\
&\mathcal{B}^{K,K'(1,0)}_{m,n}=|e_m^K|(\mathcal{J}_{11}^{K'}\widehat{\mathcal{B}}^{(1,0)}_{m,n}+\mathcal{J}_{21}^{K'}\widehat{\mathcal{B}}^{(0,1)}_{m,n}),\\
&\mathcal{B}^{K,K'(0,1)}_{m,n}=|e_m^K|(\mathcal{J}_{12}^{K'}\widehat{\mathcal{B}}^{(1,0)}_{m,n}+\mathcal{J}_{22}^{K'}\widehat{\mathcal{B}}^{(0,1)}_{m,n}).
\end{split}
\end{equation}

\section{Expressions of the local matrices in \eqref{eq:DG_Diff_vect1D} }\label{appd:1DMatrixEquation}
The expressions for the local matrices in the DG formulation \eqref{eq:DG_Diff_vect1D} are given below, where $\beta \sim O(1/h)$ is the penalty parameter, chosen sufficiently large to ensure stability.
\begin{itemize}
    \item $k=1$:
\begin{equation}
\begin{split}
&D_{-1}=\frac{1}{h^2}\begin{pmatrix}
2 & -5\\
-1 & 4
\end{pmatrix}+\frac{\beta}{h^2}\begin{pmatrix}
0& 4\\
0& -2
\end{pmatrix},\\
&D_{0}=\frac{1}{h^2}\begin{pmatrix}
0& 0\\
0& 0
\end{pmatrix}+\frac{\beta}{h^2}\begin{pmatrix}
-4& 2\\
2& -4
\end{pmatrix},\\
&D_{1}=\frac{1}{h^2}\begin{pmatrix}
4& -1\\
-5& 2
\end{pmatrix}+\frac{\beta}{h^2}\begin{pmatrix}
-2& 0\\
4& 0
\end{pmatrix}.
\end{split}
\end{equation}

\item $k=2$:
\begin{equation}
\begin{split}
&D_{-1}=\frac{1}{h^2}\begin{pmatrix}
-\frac{9}{2}& 18& -\frac{63}{2}\\
\frac{3}{4}& -3& \frac{39}{4} \\
-\frac{3}{2}& 6& -\frac{33}{2}
\end{pmatrix}
+\frac{\beta}{h^2}\begin{pmatrix}
0& 0& 9\\
0& 0& -\frac{3}{2}\\
0& 0& 3
\end{pmatrix},\\
&D_{0}=\frac{1}{h^2}\begin{pmatrix}
7& 16& 7\\
-\frac{1}{2}& -14& -\frac{1}{2}\\
7& 16& 7
\end{pmatrix}
+\frac{\beta}{h^2}\begin{pmatrix}
-9& 0& -3\\
\frac{3}{2}& 0& \frac{3}{2}\\
-3& 0& -9
\end{pmatrix},\\
&D_{1}=\frac{1}{h^2}\begin{pmatrix}
-\frac{33}{2}& 6& -\frac{3}{2}\\
\frac{39}{4}& -3& \frac{3}{4}\\
-\frac{63}{2}& 18& -\frac{9}{2}
\end{pmatrix}+\frac{\beta}{h^2}\begin{pmatrix}
3& 0& 0\\
-\frac{3}{2}& 0& 0\\
9& 0& 0
\end{pmatrix}.
\end{split}
\end{equation}

\item $k=3$:
\begin{equation}
\begin{split}
&D_{-1}=\frac{1}{h^2}\begin{pmatrix}
8& -20 (-1 + \sqrt{5})& 20 (1 + \sqrt{5})& -108\\
-\frac{2}{\sqrt{5}}& 5 - \sqrt{5}& -5 - \sqrt{5}& 3 + \frac{51}{\sqrt{5}}\\
\frac{2}{\sqrt{5}}& -5 + \sqrt{5}& 5 + \sqrt{5}& 3 - \frac{51}{\sqrt{5}}\\
-2& -5 + 5 \sqrt{5}& -5 (1 + \sqrt{5})& 42
\end{pmatrix}+\frac{\beta}{h^2}\begin{pmatrix}
0& 0& 0& 16\\
0& 0& 0& -\frac{4}{\sqrt{5}}\\
0& 0& 0& \frac{4}{\sqrt{5}}\\
0& 0& 0& -4
\end{pmatrix},\\
&D_{0}=\frac{1}{h^2}\begin{pmatrix}
30& 10(1 + \sqrt{5})& -10(-1 + \sqrt{5})& -20\\
2 - 2\sqrt{5}& -30& 20& 2(1 + \sqrt{5})\\
2(1 + \sqrt{5})& 20& -30& 2 - 2\sqrt{5}\\
-20& -10(-1 + \sqrt{5})& 10(1 + \sqrt{5})& 30
\end{pmatrix}+\frac{\beta}{h^2}\begin{pmatrix}
-16& 0& 0& 4\\
\frac{4}{\sqrt{5}}& 0& 0& -\frac{4}{\sqrt{5}}\\
-\frac{4}{\sqrt{5}}& 0& 0& \frac{4}{\sqrt{5}}\\
4& 0& 0& -16
\end{pmatrix},\\
&D_{1}=\frac{1}{h^2}\begin{pmatrix}
42& -5 (1 + \sqrt{5})& -5 + 5\sqrt{5}& -2\\
3 - \frac{51}{\sqrt{5}}& 5 + \sqrt{5}& -5 + \sqrt{5}& \frac{2}{\sqrt{5}}\\
3 + \frac{51}{\sqrt{5}}& -5 - \sqrt{5}& 5 - \sqrt{5}& -\frac{2}{\sqrt{5}}\\
-108& 20(1 + \sqrt{5})& -20(-1 + \sqrt{5})& 8
\end{pmatrix}+\frac{\beta}{h^2}\begin{pmatrix}
-4& 0& 0& 0\\
\frac{4}{\sqrt{5}}& 0& 0& 0\\
-\frac{4}{\sqrt{5}}& 0& 0& 0\\
16& 0& 0& 0
\end{pmatrix}.
\end{split}
\end{equation}

\item $k=4$:
\begin{equation}
\begin{split}
D_{-1}=&\frac{1}{h^2}\begin{pmatrix}
-\frac{25}{2}& -\frac{175}{12} (-7 + \sqrt{21})& -\frac{200}{3}& 
   \frac{175}{12} (7 + \sqrt{21})& -275\\
\frac{15}{14}& \frac{5}{4} (-7 + \sqrt{21})& \frac{40}{7}& -\frac{5}{4} (7 + \sqrt{21})& \frac{330}{7} + 15 \sqrt{\frac{3}{7}}\\
-\frac{15}{16}& -\frac{35}{32} (-7 + \sqrt{21})& -5& 
   \frac{35}{32} (7 + \sqrt{21})& -\frac{285}{8}\\
\frac{15}{14}& \frac{5}{4} (-7 + \sqrt{21})& \frac{40}{7}& -\frac{5}{4} (7 + \sqrt{21})& 
   \frac{330}{7} - 15 \sqrt{\frac{3}{7}}\\
-\frac{5}{2}& -\frac{35}{12} (-7 + \sqrt{21})& -\frac{40}{3}& \frac{35}{12} (7 + \sqrt{21})& -85
\end{pmatrix}+\frac{\beta}{h^2}\begin{pmatrix}
0& 0& 0& 0& 25\\
0& 0& 0& 0& -\frac{15}{7}\\
0& 0& 0& 0& \frac{15}{8}\\
0& 0& 0& 0& -\frac{15}{7}\\
0& 0& 0& 0& 5
\end{pmatrix},\\
D_{0}=&\frac{1}{h^2}\begin{pmatrix}
\frac{165}{2}& \frac{7}{6}(35 + 4 \sqrt{21})& \frac{16}{3}& -\frac{7}{6} (-35 + 4 \sqrt{21})& \frac{81}{2}\\
-\frac{135}{14} + 6 \sqrt{\frac{3}{7}}& -\frac{385}{6}& \frac{688}{21}& -\frac{133}{6}& -\frac{135}{14} - 6 \sqrt{\frac{3}{7}}\\
\frac{207}{16}& \frac{1519}{48}& -\frac{110}{3}& \frac{1519}{48}& \frac{207}{16}\\
-\frac{135}{14} - 6 \sqrt{\frac{3}{7}}& -\frac{133}{6}& \frac{688}{21}& -\frac{385}{6}& -\frac{135}{14} + 6 \sqrt{\frac{3}{7}}\\
\frac{81}{2}& -\frac{7}{6} (-35 + 4 \sqrt{21})& \frac{16}{3}& \frac{7}{6} (35 + 4 \sqrt{21})& \frac{165}{2}
\end{pmatrix}\\
&+\frac{\beta}{h^2}\begin{pmatrix}
-25& 0& 0& 0& -5\\
\frac{15}{7}& 0& 0& 0& \frac{15}{7}\\
-\frac{15}{8}& 0& 0& 0& -\frac{15}{8}\\
\frac{15}{7}& 0& 0& 0& \frac{15}{7}\\
-5& 0& 0& 0& -25
\end{pmatrix},\\
D_{1}=&\frac{1}{h^2}\begin{pmatrix}
-85& \frac{35}{12} (7 + \sqrt{21})& -\frac{40}{3}& -\frac{35}{12} (-7 + \sqrt{21})& -\frac{5}{2}\\
\frac{330}{7} - 15 \sqrt{\frac{3}{7}}& -\frac{5}{4} (7 + \sqrt{21})& \frac{40}{7}& \frac{5}{4} (-7 + \sqrt{21})& \frac{15}{14}\\
-\frac{285}{8}& \frac{35}{32} (7 + \sqrt{21})& -5& -\frac{35}{32} (-7 + \sqrt{21})& -\frac{15}{16}\\
\frac{330}{7} + 15 \sqrt{\frac{3}{7}}& -\frac{5}{4}(7 + \sqrt{21})& \frac{40}{7}& \frac{5}{4} (-7 + \sqrt{21})& \frac{15}{14}\\
-275& \frac{175}{12} (7 + \sqrt{21})& -\frac{200}{3}& -\frac{175}{12} (-7 + \sqrt{21})& -\frac{25}{2}
\end{pmatrix}+\frac{\beta}{h^2}\begin{pmatrix}
5& 0& 0& 0& 0\\
-\frac{15}{7}& 0& 0& 0& 0\\
\frac{15}{8}& 0& 0& 0& 0&\\
-\frac{15}{7}& 0& 0& 0& 0&\\
25& 0& 0& 0& 0
\end{pmatrix}.
\end{split}
\end{equation}

\end{itemize}

\end{appendices}

\end{document}